\documentclass[11pt]{amsart}
\usepackage{amsmath,amssymb,amsthm,tikz,mathrsfs}
\RequirePackage[numbers,sort&compress]{natbib}

\usepackage{graphicx,tikz,subfigure,color}
\usetikzlibrary{decorations.pathreplacing}

\RequirePackage{color}
\RequirePackage[colorlinks, urlcolor=my-blue,linkcolor=my-link,citecolor=my-green]{hyperref}
\definecolor{my-link}{rgb}{0.5,0.0,0.0}
\definecolor{my-blue}{rgb}{0.0,0.0,0.6}
\definecolor{my-red}{rgb}{0.5,0.0,0.0}
\definecolor{my-green}{rgb}{0.0,0.5,0.0}
\definecolor{nicos-red}{rgb}{0.75,0.0,0.0}
\definecolor{light-gray}{gray}{0.6}
\definecolor{really-light-gray}{gray}{0.8}

\usepackage{nicefrac}
\usepackage{enumerate}
\usepackage[normalem]{ulem} 
  {\par\medskip\noindent{\color{red}\bfseries[BEGINNING OF TEXT TO DELETE\ifx\relax#1\relax\else: #1\fi]}\par\color{gray}}%
  {\par\noindent{\color{red}\bfseries[END OF TEXT TO DELETE]}\par\medskip}

\newtheorem{theorem}{\sc Theorem}[section]
\newtheorem{lemma}[theorem]{\sc Lemma}
\newtheorem{proposition}[theorem]{\sc Proposition}
\newtheorem{corollary}[theorem]{\sc Corollary}

\newtheorem{definition}[theorem]{\it Definition}
\numberwithin{equation}{section}
\theoremstyle{remark}
\newtheorem{remark}[theorem]{Remark}

\newcommand{\be}{\begin{equation}}
\newcommand{\ee}{\end{equation}}

\providecommand{\abs}[1]{\vert#1\vert}

\providecommand{\pp}[1]{\langle#1\rangle}
\newcommand{\fl}[1]{\lfloor{#1}\rfloor} 
\newcommand{\ce}[1]{\lceil{#1}\rceil}

\def\Coc{\mathscr C}
\def\Cor{\Coc_0}
  
   \def\cE{\mathcal{E}}
\def\cG{\mathcal{G}}   \def\cW{\mathcal{W}}
\def\cH{\mathcal{H}}

\def\cP{\mathcal{P}}
\def\cR{\mathcal{R}}

  \def\cT{\mathcal{T}}

\def\esssup{\mathop{\mathrm{ess\,sup}}}

\def\Phat{\widehat\P}
\def\Omhat{\widehat\Omega}

\def\what{\hat\w}
\def\xbar{\bar x}
\def\t{\mathfrak t}
\def\tmin{{\underline\t}}
\def\tmax{{\bar\t}}

\newcommand{\one}{\mbox{\mymathbb{1}}}

\def\kS{\mathfrak{S}}

\def\bE{\mathbb{E}}
\def\bN{\mathbb{N}}
\def\bP{\mathbb{P}}

\def\bR{\mathbb{R}}
\def\bZ{\mathbb{Z}}

 \def\Z{\bZ}  \def\R{\bR}\def\N{\bN}

\def\w{\omega}
\def\om{\omega}

\def\e{\varepsilon}
\def\ind{\mathbf{1}}
\def\ddd{\displaystyle} 
\def\ri{{\mathrm{ri\,}}}

\def\m1{\mathbf{1}}

\def\kS{\mathfrak S}

 \def\Vvv{{\rm\mathbb{V}ar}}

 \def\wt{\widetilde}  \def\wh{\widehat} 
 
\def\E{\bE}
\def\P{\bP} 

\def\funct lp{L} 
\def\funct lpbar{\bar L} 

\font \mymathbb = bbold10 at 11pt

\def\range{\mathcal R}
\def\Uset{\mathcal U}
\def\Diff{\mathcal D}
\def\EP{\mathcal E}
\def\EPlim{{\widehat{\mathcal E}}}

\def\cV{\mathcal V}

\def\gpp{{g_{\text{pp}}}}
\def\gpl{{g_{\text{pl}}}}
\def\gppa{\gamma}

\def\Gsw{G}

\def\f{f}

\def\cA{{\mathcal A}}

\def\Gpl{G}
\def\Gpp{G}
\def\gpp{g_{\text{\rm pp}}}
\def\gpl{g_{\text{\rm pl}}}

\def\ri{{\mathrm{ri\,}}}
\def\f{{f}} 
\def\gppa{{\gamma}}
\def\amin{{\underline\alpha}}
\def\amax{{\overline\alpha}}

\def\smin{{\underline s}}
\def\smax{{\overline s}}
\def\ximin{{\underline\xi}}
\def\ximax{{\overline\xi}}
\def\zetamin{{\underline\zeta}}
\def\zetamax{{\overline\zeta}}

\def\B{{B}}

\def\cE{{\mathcal E}}

\def\W{{W}} 
\def\S{{S}}
\def\A{{A}}
\def\D{{D}}
\def\Wp{{\wt W}}
\def\Sp{{\wt S}}
\def\Ap{{\wt A}}
\def\M{{\mathcal M}}
\def\U{{U}}
\def\Ombig{{\widehat\Omega}}
\def\Pbig{{\widehat\P}}
\def\Ebig{{\widehat\E}}
\def\kSbig{{\widehat\kS}}
\def\OBP{(\Omega, \kS, \P)} 
\def\OBPbig{(\Ombig, \kSbig, \Pbig)}

\def\Gsw{G^{\rm SW}}
\def\Gne{G^{\rm NE}}

\def\Gn{G^{\rm N}}

\def\cid{\xi_*}
\def\Bci{B_*}
\def\cidmin{\underline\xi_*}
\def\cidmax{\overline\xi_*}

\def\pp{p}  
\def\Ew{m_0}   
\def\opc{\vec p_c} 
\def\Saset{V}  

\def\cla{\cO} 
\def\clb{\cO} 
\def\aaa{\ell}   
\def\etamin{{\underline\eta}}
\def\etamax{{\overline\eta}}

\def\opc{\vec p_c}
\def\cO{\mathcal O}

\def\cidl{\cid^{\text\rm{(l)}}}
\def\cidr{\cid^{\text\rm{(r)}}}

\def\cT{\mathcal T}

\definecolor{darkgreen}{rgb}{0.0,0.5,0.0}
\definecolor{darkblue}{rgb}{0.0,0.0,0.3}
\definecolor{nicosred}{rgb}{0.65,0.1,0.1}
\definecolor{light-gray}{gray}{0.7}
\allowdisplaybreaks[1]

\allowdisplaybreaks[1]
\begin{document}
\title[Corner growth model]{Stationary cocycles for the corner growth model}


\author[N.~Georgiou]{Nicos Georgiou}
\address{Nicos Georgiou\\ University of Sussex\\School of Mathematical and Physical Sciences\\ Department of Mathematics\\ Falmer Campus\\ Brighton BN1 9QH\\ UK.}
\email{n.georgiou@sussex.ac.uk}
\thanks{N.\ Georgiou was partially supported by a Wylie postdoctoral fellowship at the University of Utah.}

\author[F.~Rassoul-Agha]{Firas Rassoul-Agha}
\address{Firas Rassoul-Agha\\ University of Utah\\  Mathematics Department\\ 155S 1400E, Room 233.\\   Salt Lake City, UT 84112\\ USA.}
\email{firas@math.utah.edu}
\urladdr{http://www.math.utah.edu/~firas}
\thanks{F.\ Rassoul-Agha and N.\ Georgiou were partially supported by National Science Foundation grant DMS-0747758.}

\author[T.~Sepp\"al\"ainen]{Timo Sepp\"al\"ainen}
\address{Timo Sepp\"al\"ainen\\ University of Wisconsin-Madison\\  Mathematics Department\\ Van Vleck Hall\\ 480 Lincoln Dr.\\   Madison WI 53706-1388\\ USA.}
\email{seppalai@math.wisc.edu}
\urladdr{http://www.math.wisc.edu/~seppalai}
\thanks{T.\ Sepp\"al\"ainen was partially supported by  National Science Foundation grant   DMS-1306777 and by the Wisconsin Alumni Research Foundation.} 

\keywords{Busemann function, cocycle, competition interface, corrector,  directed percolation, geodesic, last-passage percolation, queueing fixed point,  variational formula.}
\subjclass[2000]{60K35, 65K37} 
\date{November 27, 2026}   
\begin{abstract}
 We study  the directed last-passage percolation model on the planar integer lattice  with nearest-neighbor steps and general i.i.d.\ weights on the vertices,  outside the class of exactly solvable models.   Shift-covariant cocycles are constructed for this percolation model from queueing fixed points.  These cocycles serve as boundary conditions for stationary last-passage percolation, define solutions to variational formulas that characterize limit shapes,  and   yield new results for Busemann functions, geodesics and the competition interface.  
\end{abstract}
\maketitle

\maketitle
\tableofcontents

\section{Introduction}

  We study nearest-neighbor  directed  last-passage percolation (LPP)  on the lattice $\Z^2$, also called the {\sl corner growth model}.  Random i.i.d.\ weights $\{\w_x\}_{x\in\Z^2}$ are  used to define {\sl last-passage times}  $\Gpp_{x,y}$ between  lattice points $x\le y$ in $\Z^2$ by 
 \be \Gpp_{x,y}=\max_{x_\centerdot}\sum_{k=0}^{n-1}\w_{x_k} 
	\label{Gxy1}\ee
where the maximum is over paths $x_\centerdot=\{x=x_0, x_1, \dotsc, x_n=y\}$  that satisfy  $x_{k+1}-x_k\in\{e_1,e_2\}$ (up-right paths).  

 When $\w_x\ge 0$   this defines  a growth model in the first quadrant $\Z_+^2$.  Initially  the growing cluster    is   empty.   The origin joins the cluster at time $\w_0$.  After both  $x-e_1$ and $x-e_2$   have joined the cluster,  point $x$ waits time $\w_x$ to  join.  (However,  if  $x$ is on the boundary of $\Z_+^2$, only  one of $x-e_1$ and $x-e_2$ is required to  have joined.)  The cluster    at  time $t$  is   $\cA_t=\{x\in\Z_+^2:  \Gpp_{0,x}+ \w_x \le t\}$.       Our convention to exclude the last weight $\w_{x_n}$ in \eqref{Gxy1} forces the clumsy addition of $\w_x$ in the definition of $\cA_t$, but is convenient for other purposes.  

The interest is in the large-scale behavior of the model.  This begins with  the deterministic limit  $\gpp(\xi)=\lim_{n\to\infty}  n^{-1}\Gpp_{0, \fl{n\xi}}$  for  $\xi\in\R_+^2$, the   fluctuations of $\Gpp_{0, \fl{n\xi}}$,   and the behavior of the maximizing paths in \eqref{Gxy1} called  {\sl geodesics}. 
Closely related are the  {\sl  Busemann functions} that  are limits  of gradients $\Gpp_{x, v_n}-\Gpp_{y,v_n}$ as   $v_n\to\infty$ in a particular direction and the {\sl competition interface} between subtrees of the geodesic tree.   To see how Busemann functions   connect with geodesics, note that by 
  \eqref{Gxy1}   the following identity holds  along any geodesic $x_\centerdot$ from $u$ to $v_n$: 
	\be\label{w-geod}\w_{x_i}=\min\bigl(\Gpp_{x_i,v_n}-\Gpp_{x_i+e_1,v_n}\,,\,\Gpp_{x_i,v_n}-\Gpp_{x_i+e_2,v_n}\bigr).\ee

Busemann functions arise also  in a limiting   description of  the $\Gpp_{x,y}$ process locally around a   point $v_n\to\infty$.    Take a finite subset $\cV$ of $\Z^2$.   A natural expectation is  that  the vector  $\{\Gpp_{0, v_n-u} -  \Gpp_{0, v_n} : u\in\cV\}$ converges in distribution as  $v_n\to\infty$ in a particular direction.   A shift by $-v_n$ and reflection    $\w_x\mapsto \w_{-x}$  turn this vector  into 
$\{  \Gpp_{u, v_n } - \Gpp_{0, v_n}  : u\in\cV \}$.  For this last collection of random gradients we can  expect almost sure convergence, in particular if the geodesics from $0$ and $u$ to $v_n$ coalesce eventually.  These types of results will be developed in the paper.



Here are some particulars of what follows, in relation to past work.

In \cite{Geo-Ras-Sep-16} we derived variational formulas for the point-to-point limit $\gpp(\xi)$ and its point-to-line counterpart (introduced in Section \ref{sec:results}) and developed a solution ansatz for these variational formulas  in terms of shift-covariant cocycles. In the present paper we construct these cocycles for the planar corner growth model with general i.i.d.\ weights bounded from below, subject to a moment bound.  This construction comes from the fixed points of  the associated queueing operator. The existence of these fixed points was proved by Mairesse and Prabhakar \cite{Mai-Pra-03}.  These cocycles are constructed on an extended space of weights.  The Markov process analogy is a simultaneous construction  of processes for  all invariant distributions, coupled by  common Poisson clocks that drive the evolution.   The i.i.d.\ weights $\w$ are the analogue of the clocks and the cocycles the analogues of the initial state variables.    With the help of the cocycles we establish new results for Busemann functions and directional geodesics for the corner growth model.

 A related recent result is 
Krishnan's  \cite[Theorem 2.3]{Kri-16} variational formula for the time constant of  first passage bond percolation. His formula is analogous to our \eqref{eq:g:K-var}.

  Under some moment assumptions on the weights, the corner growth  model   is expected to lie in the Kardar-Parisi-Zhang (KPZ)  universality class. (For a review of KPZ universality see \cite{Cor-12}.)   The fluctuations of $\Gpp_{0, \fl{n\xi}}$ are expected to have order of magnitude $n^{1/3}$ and  limit distributions from random matrix theory.   When the weights have exponential or geometric distribution the model is exactly solvable.  In these cases the  cocycles mentioned above have explicit product form distributions and last-passage values can be described with  the Robinson-Schensted-Knuth correspondence and determinantal point processes.  These techniques allow the derivation of exact fluctuation exponents and limit distributions \cite{Bal-Cat-Sep-06,Joh-00,Joh-06}. 
  The present paper can be seen as an attempt to begin development of techniques for studying the corner growth model beyond the exactly solvable cases.




On   Busemann functions and geodesics, past milestones are the first-passage percolation results of Newman et al.\ summarized in \cite{New-95},  the applications of his techniques to the exactly solvable exponential corner growth model by Ferrari and Pimentel \cite{Fer-Pim-05}, and the recent improvements to \cite{New-95} by Damron and Hanson \cite{Dam-Han-14}.  Coupier \cite{Cou-11} further sharpened the results for the exponential corner growth model.     Shift-covariant cocycles have not been developed  for first-passage percolation.    \cite{New-95} utilized a global curvature assumption to derive properties of geodesics, and then the existence of Busemann functions.   
\cite{Dam-Han-14} began with a weak limit of Busemann functions from which properties of geodesics follow.  

 In our setting everything flows from the cocycles, both almost sure existence of Busemann functions and properties of geodesics.  With a cocycle appropriately coupled to  the weights $\w$, geodesics can be defined locally in a constructive manner, simply by following   minimal gradients of the cocycle.    
 
 The role of the regularity of the function $\gpp$ in our paper needs to be explained.    Presently it is expected but not yet proved that under our assumptions (i.i.d.\ weights with some moment hypothesis)  $\gpp$ is differentiable and, if $\w_0$ has a continuous distribution,  strictly concave.     Our development of the cocycles and their consequences for Busemann functions, geodesics and the competition interface by and large do not rely on any regularity assumptions.  Instead the results are developed in a  general manner so that  points of nondifferentiability are allowed, as well as flat segments even if $\w_0$ has a continuous distribution.  After   these fundamental but at times technical results are in place,  we can invoke regularity assumptions to  state cleaner corollaries  where the underlying cocycles and their extended space do not appear.   We put these tidy results at the front of the paper in Section \ref{sec:results}.  The real work begins after that.  The point we wish to  emphasize is that  no unrealistic assumptions are made and we expect future work to verify the regularity assumptions that appear in this paper.

 \smallskip

{\bf Organization of the paper.}  
Section \ref{sec:results} describes the corner growth model and the main results of the paper.   These results are the cleanest ones stated under assumptions on the regularity of the limit function  $\gpp(\xi)$.   The properties we use as hypotheses are expected to be true   but  they are not presently known.   Later sections   contain more general results, but at a   price: (a) the statements are not as clean because they need to take corners and flat segments of $\gpp$ into consideration and (b) the results are valid on an extended space that supports additional edge  weights (cocycles)  in addition to vertex weights $\w$ in \eqref{Gxy1}.  

Section \ref{sec:duality} develops a convex duality between directions or velocities $\xi$ and tilts or external fields $h$ that comes from  the relationship of the point-to-point and point-to-line percolation models.  

Section \ref{sec:cocycles}  states the existence and properties of the cocycles on which all the results of the paper are based.   The cocycles define a stationary last-passage model.  The variational formulas for the percolation limits are first solved on the extended space of the cocycles.  

Section \ref{sec:busemann} develops the existence of  Busemann functions.  

Section \ref{sec:geod} studies cocycle geodesics and with their help proves our results for geodesics.

Section \ref{sec:ci-pf}  proves results for the competition interface.  
 
Section \ref{sec:solv} discusses examples with geometric and exponential   weights $\{\w_x\}$.  
These are of course exactly solvable models, but it is useful to see the   theory illustrated in its ideal form.  

Several appendixes come at the end.   Appendix \ref{app:q} proves the main theorem of Section \ref{sec:cocycles} by relying on queuing results from \cite{Mai-Pra-03, Pra-03}.   Appendix \ref{app:geod} proves the coalescence of cocycle geodesics by adapting the first-passage percolation proof of \cite{Lic-New-96}.  A short 
Appendix \ref{app:aux} states an ergodic theorem for cocycles proved in \cite{Geo-etal-15}.    

Appendix \ref{sec:cone-pf} proves properties of the limit $\gpp$ when the weights are bounded above by $1$ and have an atom at $1$, so that $\gpp(\xi)$ is at most $\xi\cdot(e_1+e_2)$. Below the critical density of oriented percolation this bound is strict in every direction. At the critical density $\gpp(\xi)=\xi\cdot(e_1+e_2)$ is attained only in the diagonal direction $\xi=(\tfrac12,\tfrac12)$, at which point $\gpp$ is differentiable. Above the critical density $\gpp(\xi)=\xi\cdot(e_1+e_2)$ holds  exactly on the percolation cone, and $\gpp$ is differentiable also at the edges of the cone. The proofs are adapted from the first-passage percolation work of \cite{Auf-Dam-13, Mar-02}. Appendix \ref{app:shape}  states the  almost sure shape theorem for the corner growth model from \cite{Mar-04} and proves an $L^1$ version.
\smallskip

{\bf AI use statement.}
  The new results in this version 4 of the paper are the differentiability of $\gpp$ at $\xi=(\tfrac12,\tfrac12)$ in the critical case $p_1=\opc$, the extension of Theorem \ref{th:flat-edge} to all $p_1<1$, and the weakening of its moment assumption to $\E\w_0^-<\infty$. The strategy for the critical case was clear, since \cite{Auf-Dam-13} already covered it for first-passage percolation. We used Claude Opus 5.5 (Medium) to adapt their proof to the corner growth model and to indicate the necessary changes to the text. This version also corrects the statement and the proof of the coalescence result, Theorem~\ref{th:geod:coal:t}. It is now stated for shift-covariant tie-breaking rules that satisfy a finite energy condition, and its proof in Appendix~\ref{app:geod} has been adjusted accordingly. The results that rely on it use only the two constant tie-breaking rules and are unaffected. In all cases we checked and edited the text ourselves, again using Claude to spot typos and mathematical slips.

   Results of this paper are inputs to the recent AI-assisted work \cite{Jan-Ras-Sep-26-diff-} by C.~Janjigian and the second and third authors, which proves the differentiability of $\gpp$ on the whole open first quadrant under general assumptions.

\smallskip

{\bf Notation and conventions.}   $\R_+=[0,\infty)$,  $\Z_+=\{0,1,2,3, \dotsc\}$, and $\N=\{1,2,3,\dotsc\}$. The standard basis vectors of $\R^2$ are  $e_1=(1,0)$ and $e_2=(0,1)$ and  the $\ell^1$-norm of $x\in\R^2$  is   $\abs{x}_1=\abs{x\cdot e_1} + \abs{x\cdot e_2}$.   For $u,v\in\R^2$ a  closed line segment on $\R^2$ is denoted by   $[u,v]=\{tu+(1-t)v:  t\in[0,1]\}$, and an open line segment by  $]u,v[=\{tu+(1-t)v:  t\in(0,1)\}$.    Coordinatewise ordering $x\le y$ means that $x\cdot e_i\le y\cdot e_i$ for both $i=1$ and $2$.  Its negation $x\not\le y$ means that   $x\cdot e_1> y\cdot e_1$ or   $x\cdot e_2> y\cdot e_2$.    An admissible or up-right path $x_{0,n}=(x_k)_{k=0}^n$  on $\Z^2$ satisfies $x_k-x_{k-1}\in\{e_1,e_2\}$.  

The basic environment space is $\Omega=\R^{\Z^2}$ whose elements are denoted by $\w$.  There is also a larger product space $\Ombig=\Omega\times\Omega'$ whose elements are denoted by $\what=(\w, \w')$ and $\tilde\w$.   

Parameter  $\pp>2$ appears in a  moment hypothesis  $\E[|\w_0|^{\pp}]<\infty$,  while $p_1$ is the  probability of an open site in an oriented site percolation process.  
 
A statement that  contains $\pm$ or $\mp$ is a combination of two statements: one for the top choice of the sign and another one for the bottom choice.  


\section{Main results}  
\label{sec:results}

\subsection{Assumptions} 
The two-dimensional  corner growth model is the last-passage percolation model on the planar square lattice $\Z^2$ with admissible steps   $\range=\{e_1,e_2\}$.  
    $\Omega=\R^{\Z^2}$ is the   space of   environments  or weight configurations $\w=(\w_x)_{x\in\Z^2}$.   The group of spatial translations   $\{T_x\}_{x\in\Z^2}$  acts on $\Omega$ by $(T_x\w)_y=\w_{x+y}$ for $x,y\in\Z^2$.   Let $\kS$ denote the Borel $\sigma$-algebra of $\Omega$.  $\P$ is  a Borel probability measure    on $\Omega$  under which the weights  $\{\w_x\}$  are  independent, identically distributed (i.i.d.) nondegenerate random variables with a $2+\e$ moment.      Expectation under $\P$ is denoted by $\E$.   For a technical reason we also assume $\P(\w_0\ge c)=1$ for some finite constant $c$.
    
For future reference we summarize our standing assumptions in this statement: 
	\be\begin{aligned} \label{2d-ass}
		&\text{$\P$ is i.i.d., \, $\E[|\w_0|^{\pp}]<\infty$ for some $\pp>2$,\,     $\sigma^2=\Vvv(\w_0)>0$, and }\\
		 &\text{$\P(\w_0\ge c)=1$ for some   $c>-\infty$.} 
	\end{aligned} 
	\ee  
 Except in Section \ref{sec:cone} and Appendixes \ref{app:aux}, \ref{sec:cone-pf}, and \ref{app:shape}, which specify their own hypotheses, assumption \eqref{2d-ass} is valid throughout the paper and will not be repeated in every statement. 	The constant 
\[\Ew=\E(\w_0)\]  will appear  frequently.  	The symbol   $\w$ is reserved for   these $\P$-distributed  i.i.d.\ weights, also later when they are embedded in a larger configuration $\what=(\w, \w')$.  

  Assumption $\P(\w_0\ge c)=1$ is required in only one part of our proofs, namely in  Appendix \ref{app:q} where we rely on  results from queueing theory.   In that context $\w_x$ is a service time, and the results we use have been proved only for $\w_x\ge 0$.  (The extension to $\w_x\ge c$  is immediate.)   The point we wish to make  is that once the queueing results have been extended to general real-valued  i.i.d.\ weights $\w_x$ subject to the moment assumption in \eqref{2d-ass},  everything in this paper is true  for these general real-valued weights.  
  
  
  \subsection{Last-passage percolation} 
Given an environment $\w$ and two points $x,y\in\Z^2$ with $x\le y$ coordinatewise,
define the {\sl point-to-point last-passage time} by
	\[ \Gpp_{x,y}=\max_{x_{0,n}}\sum_{k=0}^{n-1}\w_{x_k}.\]
The maximum is over paths $x_{0,n}=(x_k)_{k=0}^n$  that start at  $x_0=x$,  end at $x_n=y$ with $n=\abs{y-x}_1$,  and have increments $x_{k+1}-x_k\in\{e_1,e_2\}$. We call such paths {\sl admissible} or {\sl up-right}.

Given a vector $h\in\R^2$, an environment $\w$, and an integer $n\ge0$, define the  $n$-step  {\sl point-to-line last passage time} with {\sl tilt} (or {\sl external field}) $h$ by
 	\[ \Gpl_n(h)=\max_{x_{0,n}}\Bigl\{\,\sum_{k=0}^{n-1}\w_{x_k}+h\cdot x_n\Bigr\}. \]
The maximum is over all admissible $n$-step   paths that start at $x_0=0$.

 It is standard (see for example \cite{Mar-04} or \cite{Ras-Sep-14})   that under assumption \eqref{2d-ass},    for $\P$-almost every $\w$, simultaneously  for every $\xi\in\R_+^2$ and every $h\in\R^2$, 
the following  limits exist:
		\begin{align}
		 &\gpp(\xi)=\lim_{n\to\infty}n^{-1}\Gpp_{0,\lfloor{n\xi}\rfloor},\label{eq:g:p2p}\\
		 &\gpl(h)=\lim_{n\to\infty}n^{-1}\Gpl_n(h).\label{eq:g:p2l}
		\end{align}
In the definition above  integer parts are taken coordinatewise: $\fl v=(\fl{a},\fl{b}) \in\Z^2$ for $v=(a,b) \in\R^2$. 

 Under assumption \eqref{2d-ass}  the limits  above   are finite nonrandom continuous functions.  In particular, $\gpp$  is continuous up to the boundary of $\R_+^2$.   Furthermore,   $\gpp$ is a symmetric, concave,  $1$-homogeneous function on $\R_+^2$  and $\gpl$ is a convex Lipschitz function on $\R^2$. Homogeneity means that    $\gpp(c\xi)=c\gpp(\xi)$ for $\xi\in\R_+^2$ and $c\in\R_+$.    By homogeneity, for most purposes it suffices to consider $\gpp$ as a function  on  the convex hull    $\Uset=\{te_1+(1-t)e_2:t\in[0,1]\}$ of $\range$.    The relative interior $\ri\Uset$ is the open line segment $\{te_1+(1-t)e_2:t\in(0,1)\}$.  


Decomposing according to the endpoint of the path and some estimation (Theorem 2.2 in \cite{Ras-Sep-14}) give
 	\be\label{h-xi} \gpl(h)=\sup_{\xi\in\Uset}\{\gpp(\xi)+h\cdot\xi\}.\ee
By convex duality   for $\xi\in\ri\Uset$
	\[\gpp(\xi)=\inf_{h\in\R^2}\{\gpl(h)-h\cdot\xi\}.\] 
Let us  say $\xi\in\ri\Uset$ and $h\in\R^2$ are {\sl dual} if  
	\begin{align}\label{eq:duality} \gpp(\xi)=\gpl(h)-h\cdot\xi.\end{align}

Very little  is known in general about $\gpp$ beyond the soft properties mentioned above.  In the    exactly solvable case, with $\w_x$ either exponential or geometric,    
$\gpp(s,t)=(s+t)\Ew+2\sigma\sqrt{st}$. 
The Durrett-Liggett flat edge result  (\cite{Dur-Lig-81}, Theorem \ref{th:flat-edge} below) tells us that  this formula is   not true for all i.i.d.\ weights.   
It does  hold for general weights asymptotically at the boundary \cite{Mar-04}:  $\gpp(1,t)=\Ew+2\sigma\sqrt{t} + o(\sqrt t\,)$
as $t\searrow 0$.  

\subsection{Gradients and convexity} 
Regularity properties of $\gpp$ play a role in  our results,  so we   introduce notation for that purpose.   
Let 
\[  \Diff=\{\xi\in\ri\Uset:\text{ }\gpp\text{ is differentiable at $\xi$}\}.\]
To be   clear,  $\xi\in\Diff$ means that the gradient $\nabla\gpp(\xi)$
 exists in the usual sense of differentiability of functions of several variables.  
 At $\xi\in\ri\Uset$ this is equivalent to the differentiability of the single variable 
 function $s\mapsto \gpp(s,1-s)$ at $s={\xi\cdot e_1}/{\abs{\xi}_1}$.  
 By concavity the set $(\ri\Uset)\smallsetminus\Diff$ is at most countable.
 
 A point $\xi\in\ri\Uset$ is an {\sl exposed point} if 
 	\be \label{eq:epod}	\begin{aligned}
	   \forall \zeta\in\ri\Uset\smallsetminus \{\xi\}: \; \gpp(\zeta)\; <\; \gpp(\xi) + \nabla\gpp(\xi)\cdot(\zeta - \xi) .   \end{aligned}\ee
The set  of {\sl exposed points of differentiability} of $\gpp$ is  
$\EP=  \{ \xi\in\Diff:  \text{\eqref{eq:epod} holds}\}$.  
The condition for an exposed point is formulated entirely in terms of $\Uset$ because  $\gpp$  is 
a homogeneous function and therefore cannot have exposed points as a function on $\R_+^2$. 	
		
It is expected that $\gpp$ is differentiable on all of $\ri\Uset$.  
But since this is not known, our development must handle 
possible points of nondifferentiability.   For this purpose we take left and right
limits on $\Uset$. Our convention is that a {\sl left limit} $\xi\to\zeta$ on $\Uset$ means that 
$\xi\cdot e_1$ increases to $\zeta\cdot e_1$, while in a {\sl right limit}   
$\xi\cdot e_1$ decreases to $\zeta\cdot e_1$.  
  
For $\zeta\in\ri\Uset$ define one-sided gradient vectors $\nabla\gpp(\zeta\pm)$ by
\begin{align*}
\nabla\gpp(\zeta\pm)\cdot e_1&=\lim_{\e\searrow0}\frac{\gpp(\zeta\pm\e e_1)-\gpp(\zeta)}{\pm\e} \\  \text{and}\quad 
\nabla\gpp(\zeta\pm)\cdot e_2&=\lim_{\e\searrow0}\frac{\gpp(\zeta\mp\e e_2)-\gpp(\zeta)}{\mp\e}.
\end{align*}
Concavity of $\gpp$ ensures that the  limits   exist.   $\nabla\gpp(\xi\pm)$ coincide (and equal $\nabla\gpp(\xi)$) 
 if and only if $\xi\in\Diff$.  
Furthermore,  on $\ri\Uset$, 
\begin{align}\label{nabla-g-lim}  
\nabla\gpp(\zeta-)=\lim_{\xi\cdot e_1\nearrow\zeta\cdot e_1} \nabla\gpp(\xi\pm)
\quad\text{and}\quad 
\nabla\gpp(\zeta+)=\lim_{\xi\cdot e_1\searrow\zeta\cdot e_1} \nabla\gpp(\xi\pm). 
\end{align}
  
For $\xi\in\ri\Uset$ define maximal  line segments on which $\gpp$ is linear,  $\Uset_{\xi-}$ for the left gradient at $\xi$ 
and $\Uset_{\xi+}$ for the right gradient at $\xi$,  
by 
	\begin{align}\label{eq:sector1}
\Uset_{\xi\pm}=\{\zeta\in\ri\Uset: 	\gpp(\zeta)-\gpp(\xi)=\nabla g(\xi\pm)\cdot(\zeta-\xi)\}.  
	\end{align}
Either or both segments can degenerate to a point. 
Let \be\label{eq:sector2} \Uset_\xi=\Uset_{\xi-}\cup\,\Uset_{\xi+}=[\ximin, \ximax]
\qquad\text{with $\ximin\cdot e_1\le \ximax\cdot e_1$.}
\ee
  If $\xi\in\Diff$ then $\Uset_{\xi+}=\Uset_{\xi-}=\Uset_\xi$, while if $\xi\notin\Diff$ then $\Uset_{\xi+}\cap\Uset_{\xi-}=\{\xi\}$.  If $\xi\in\EP$ then $\Uset_{\xi}= \{\xi\}$.     Figure \ref{fig:defs} illustrates.
  
For $\zeta\cdot e_1<\eta\cdot e_1$ in $\ri\Uset$,  $[\zeta, \eta]$ is a     
  {\sl maximal linear segment} for  $\gpp$   if $\nabla\gpp$ exists and  is constant in $]\zeta, \eta[$  but not on any strictly larger open line segment in $\ri\Uset$.  
Then    $[\zeta, \eta]=\Uset_{\zeta+}=\Uset_{\eta-}=\Uset_\xi$ for any $\xi\in\;]\zeta, \eta[$.   If $\zeta, \eta\in\Diff$ we say that $\gpp$ is differentiable at the endpoints of this maximal linear segment.  This  hypothesis will be invoked several times.

  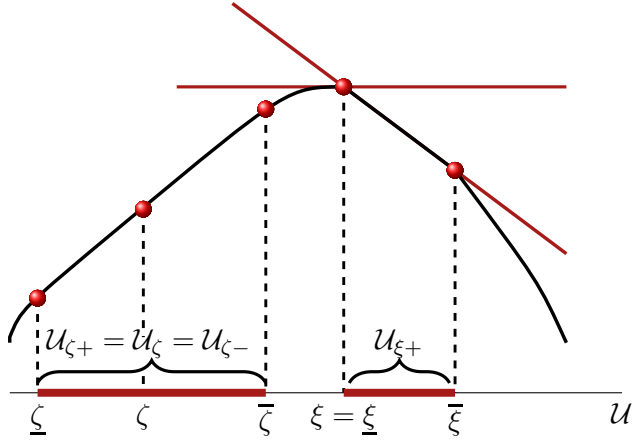
\begin{figure}[h]	
	\begin{center}
		\begin{tikzpicture}[>=latex,scale=0.74]

	\draw[-](0,0)--(11,0)node[below]{$\Uset$};
		
	\draw[ line width = 1.3 pt, color=nicosred] (4, 7)node[right,color=black]{$ $}--(6,5.5)--(8,4)--(10,2.5);
		
\draw[ line width = 1.2 pt, color=nicosred] (3, 5.5)--(5,5.5)--(7,5.5) node[color=black, below right]{$ $}--(10,5.5);

		\draw[line width = 1.3pt] (0,.9)--(0,1) .. controls (0.2,1.5)..(2,3)..controls (5,5.5)..(6,5.5) -- (8,4)..controls (9.5,2)..(10,.9); 
		
		\draw[dashed, line width = 1.1pt] (6,5.5)--(6,0)node[below]{$\xi = \ximin$};
		\draw[dashed, line width = 1.1pt] (8,4)--(8,0)node[below]{$\ximax$};
		
		\draw[line width=3pt, color= nicosred] (6,0)--(8,0);
		
		\shade[ball color=red](8,4)circle(1.5mm);
		\shade[ball color=red](6,5.5)circle(1.5mm); 
		
		\draw [decorate,decoration={brace, amplitude = 0.3cm},xshift=0pt,yshift=5, line width=1.2pt]
						(6.1,0) --(7,0)-- (7.9,0);
		\draw(7,.44)node[above]{$\Uset_{\xi+}$};
		
				
		\draw[dashed, line width = 1.1pt] (.5,1.7)--(.5,0)node[below]{$ \zetamin$};
		\draw[dashed, line width = 1.1pt] (4.6,5.1)--(4.6,0)node[below]{$\zetamax$};
		\draw[dashed, line width = 1.1pt] (2.4,3.3)--(2.4,1);
		\draw[dashed, line width = 1.1pt] (2.4,.3)--(2.4,0)node[below]{$\zeta$};
		
		\shade[ball color=red](.5,1.7)circle(1.5mm);
		\shade[ball color=red](4.6,5.1 )circle(1.5mm); 
		\shade[ball color=red](2.4, 3.3 )circle(1.5mm); 
		
		\draw [decorate,decoration={brace, amplitude = 0.3cm},xshift=0pt,yshift=5, line width=1.2pt]
						(.5,0) -- (4.6,0);
		\draw(2.5,.44)node[above]{$\Uset_{\zeta+} = \Uset_{\zeta} = \Uset_{\zeta-}  $};
		
		\draw[line width=3pt, color= nicosred] (.5,0)--(4.6,0);	
		
		\end{tikzpicture}
	\end{center}
	\caption{\small    A graph of a concave function over $\Uset$ to illustrate the definitions.  
	   $\zetamin$, $\zeta$ and $\zetamax$ are points of differentiability while $\ximin = \xi$ and  $\ximax$ are not.   $\Uset_{\zetamin}=\Uset_\zeta=\Uset_\zetamax=[\zetamin,\zetamax]$.
	    The red lines represent supporting hyperplanes at $\xi$.   The slope from the left at $\xi$ is zero, and the  horizontal red line touches the graph only at $\xi$.  Hence  $ \Uset_{\xi-} = \{\xi\}$.  	Points on  the line segments $[\zetamin, \zetamax]$ and $]\xi, \ximax[$ are   not exposed. 
  $\cE=\ri\Uset \smallsetminus \bigl([\zetamin, \zetamax] \cup[\xi, \ximax]\bigr).$}
  \label{fig:defs}
\end{figure}



\subsection{Cocycles} 
The next definition is central to the paper.  

\begin{definition}[Cocycle]
\label{def:cK}
	A measurable function $B:\Omega\times\Z^2\times\Z^2\to\R$ is   a {\rm shift-covariant $L^1(\P)$  cocycle}
	if it satisfies the following three conditions. 
		\begin{enumerate}[\ \ \rm(a)]
			\item\label{def:cK:int} Integrability: for each $z\in\{e_1,e_2\}$, $\E\abs{B(0,z)} <\infty$.
			\item\label{def:cK:stat} Shift-covariance:  for $\P$-a.e.\ $\w$ and  all $x,y,z\in\Z^2$, 
					$B(\w, z+x,z+y)=B(T_z\w, x,y)$.
			\item\label{def:cK:coc} Cocycle property:   $\P$-a.s.\ and for all $x,y,z\in\Z^2$, $B(x,y)+B(y,z)=B(x,z)$.
		\end{enumerate}
	The space  of shift-covariant $L^1(\P)$  cocycles on $\OBP$  is denoted by $\Coc(\Omega)$.

	A cocycle  $F(\w, x,y)$  is {\rm centered} if 
  $\E[F(x,y)]=0$ for all $x,y\in\Z^2$. 
	The space  of centered shift-covariant $L^1(\P)$  cocycles  on $\OBP$  is denoted by $\Cor(\Omega)$.
\end{definition}
The cocycle property \eqref{def:cK:coc} implies that $B(x,x) = 0$ for all $x \in \Z^2$ and the antisymmetry property $B(x,y) = -B(y,x)$. 
$\Cor(\Omega)$ is the $L^1(\P)$ closure of gradients 
$F(\w,x,y)=\varphi(T_y\w)-\varphi(T_x\w)$, $\varphi\in L^1(\P)$ (see 
\cite[Lemma~C.3]{Ras-Sep-Yil-13}).     Our convention for  centering a
  shift-covariant $L^1$ cocycle $B$  is to let   
 $h(\B)\in\R^2$ denote the vector  that   satisfies 
 	\be 	
		\E[\B(0,e_i)]=-h(\B)\cdot e_i \qquad\text{for  $i\in\{1,2\}$ }  \label{eq:EB}
	\ee
and then define $F\in\Cor(\Omega)$ by 
	\be 
		F(x,y)=  h(\B)\cdot (x-y)-\B(x,y).  \label{FF}
	\ee



\subsection{Busemann functions}
We can now state the theorem on the existence of Busemann functions.   This theorem is proved in Section \ref{sec:busemann}.  

\begin{theorem}\label{thm:buse}  
Let  $\xi\in\ri\Uset$ with  $\Uset_\xi =[\ximin, \ximax]$   defined in \eqref{eq:sector2}.
 Assume that $\ximin, \xi, \ximax$  are points of differentiability of $\gpp$.   {\rm(}The degenerate case $\ximin=\xi=\ximax$ is also acceptable.{\rm)}
 %
There exists a shift-covariant $L^1(\P)$ cocycle $\{B(x,y):x,y\in\Z^2\}$ and an event  $\Omega_0$ with $\P(\Omega_0)=1$ such that  the following holds for each 
$\w\in\Omega_0$:  for each sequence $v_n\in\Z_+^2$ such that 
\begin{align}\label{eq:vn}
\abs{v_n}_1\to\infty\quad\text{and}\quad\ximin\cdot e_1\le\varliminf_{n\to\infty}\frac{v_n\cdot e_1}{\abs{v_n}_1}\le\varlimsup_{n\to\infty}\frac{v_n\cdot e_1}{\abs{v_n}_1}\le\ximax\cdot e_1,
\end{align}
we have the limit
\be \label{eq:grad:coc1} 
							B(\w, x,y) = \lim_{n\to \infty} \big( \Gpp_{x, v_n}(\w) - \Gpp_{y, v_n}(\w) \big)  
						\ee 
						 for all  $x,y\in\Z^2$. 
%
Furthermore,  
\be\label{EB=Dg}     \nabla \gpp(\zeta) = \bigl( \, \E[B(x,x+e_1)]\,,\,   \E[B(x,x+e_2)] \,\bigr)\quad\text{for all $\zeta\in\Uset_\xi$}.
\ee 
\end{theorem}

To paraphrase the theorem,  Busemann functions $B^\xi$ exist in directions $\xi\in\EP$,  and furthermore,  if $\gpp$ is differentiable at the endpoints of a maximal linear segment,    then Busemann functions exist and agree in all directions on this line segment.    (Note that if $\ximin\ne\ximax$, the statement of the theorem is the same for any $\xi\in\;]\ximin,\ximax[$\,.)  
In particular, if $\gpp$ is differentiable everywhere on $\ri\Uset$,  then (i)  for each direction $\xi\in\ri\Uset$ there is a Busemann function $B^\xi$  such that, almost surely, $B^\xi(\w,x,y)$ equals  the limit  in \eqref{eq:grad:coc1}   for any sequence $v_n/\abs{v_n}_1\to\xi$ and (ii) the $B^\xi$'s match on linear segments of $\gpp$.



   We shall not derive the  cocycle property of $B^\xi$  from the limit \eqref{eq:grad:coc1}.  Instead in Section \ref{sec:cocycles} and Appendix \ref{app:q}  we construct a family of cocycles on an extended   space $\Ombig=\Omega\times\Omega'$ and show that  one  of these   cocycles equals the limit on the right of \eqref{eq:grad:coc1}.

The Busemann limits \eqref{eq:grad:coc1}  can also be interpreted as convergence of the last-passage process  to a  stationary last-passage process, described in Section \ref{sec:stat-lpp}.

	Equation \eqref{EB=Dg}  
  was anticipated in \cite{How-New-01} (see paragraph after the proof of Theorem 1.13)   for  Euclidean first passage percolation (FPP) where $\gpp(x,y)=c\sqrt{x^2+y^2}$. 
    	A version of this formula appears also in Theorem 3.5 of \cite{Dam-Han-14} for lattice  FPP.






\subsection{Variational formulas} 

Cocycles arise in variational formulas that describe the limits of last-passage percolation models.    
 In Theorems 3.2 and 4.4 in \cite{Geo-Ras-Sep-16} we proved these variational formulas:  for $h\in\R^2$ 
	\begin{align}
		\gpl(h)&=\inf_{F\in\Cor(\Omega)}\; \P\text{-}\esssup_\w\;  \max_{i\in\{1,2\}} \{\w_0+h\cdot e_i+F(\w, 0,e_i)\}\label{eq:g:K-var}\\
		\intertext{and for $\xi\in\ri\Uset$}  
		\gpp(\xi)&=\inf_{B\in\Coc(\Omega)} \;\P\text{-}\esssup_\w\;  \max_{i\in\{1,2\}} \{\w_0 -B(\w, 0,e_i)-h(B)\cdot\xi\}.\label{eq:gpp:K-var}
	\end{align}
The next theorem states that the Busemann functions found in Theorem \ref{thm:buse} give minimizing cocycles.  	
	
	\begin{theorem}\label{thm:var-buse}  Let  $\xi\in\ri\Uset$ with  $\Uset_\xi =[\ximin, \ximax]$   defined in \eqref{eq:sector2}.
 Assume that $\ximin, \xi, \ximax\in\Diff$.  Let $B^\xi\in\Coc(\Omega)$ be given  by \eqref{eq:grad:coc1}.  We have  $h(B^\xi)=-\nabla\gpp(\xi)$ by \eqref{EB=Dg} and \eqref{eq:EB}.  Define   $F(x,y)=h(B^\xi)\cdot(x-y)-B^\xi(x,y)$ as in  \eqref{FF}.  
 
 \begin{enumerate}[\ \ \rm(i)]
 \item\label{thm:var-buse:i}   Let $h=h(B^\xi)+(t,t)$ for some $t\in\R$.  Then for $\P$-a.e.~$\w$ 
\be\label{var15}  \gpl(h)=  \max_{i\in\{1,2\}} \{\w_0+h\cdot e_i+F(\w, 0,e_i)\} = t. \ee
In other words,  $F$ is a minimizer in \eqref{eq:g:K-var} and the essential supremum vanishes.  
  
 \item\label{thm:var-buse:ii}   For $\P$-a.e.~$\w$ 
\be\label{var151}\gpp(\xi)=  \max_{i\in\{1,2\}} \{\w_0 -B(\w, 0,e_i)-h(B)\cdot\xi\}. 
\ee
In other words,  $B^\xi$ is a minimizer in \eqref{eq:gpp:K-var} and the essential supremum vanishes.  
\end{enumerate}
\end{theorem}

The condition $h=h(B^\xi)+(t,t)$ for some $t\in\R$ is equivalent to $h$ dual to $\xi$. Every $h$ has a dual $\xi\in\ri\Uset$ as we show in Section \ref{sec:duality}.  Consequently, if $\gpp$ is differentiable everywhere on $\ri\Uset$,  each $h$ has a minimizing Busemann cocycle $F$ that satisfies \eqref{var15}.    Theorem \ref{thm:var-buse}  is proved in Section \ref{sec:busemann}.  

The choice of $i\in\{1,2\}$ in \eqref{var15} and \eqref{var151} must depend on $\w$. This choice  is determined if   $\xi$  is not the asymptotic direction 
 of the competition interface (see Remark \ref{rem:cif-var} below).  

Borrowing from the homogenization literature (see e.g.\ page 468 of \cite{Arm-Sou-12}), 
a minimizer of \eqref{eq:g:K-var} or \eqref{eq:gpp:K-var} 
that also removes the essential supremum, that is,  a cocycle that satisfies \eqref{var15} or 
\eqref{var151}, is called a {\sl corrector}.

 	
\smallskip	
	
\subsection{Geodesics}\label{subsec:geo} 	


For $u\le v$ in $\Z^2$ an admissible path $x_{0,n}$ from  $x_0 = u$ to  $x_n=v$  (with $n=\abs{v-u}_1$)  
is   a (finite) {\sl geodesic} from $u$ to $v$ 
if 
	\[
		\Gpp_{u,v} = \sum_{k=0}^{n-1} \w_{x_k}.
	\]
An up-right path $x_{0,\infty}$ is an {\sl infinite  geodesic} emanating from $u\in\Z^2$ if $x_0=u$ and for any $j>i\ge0$, $x_{i,j}$ is a geodesic between $x_i$ and $x_j$.	Two infinite geodesics $x_{0, \infty}$ and $y_{0, \infty}$ {\sl coalesce} if there exist $m,n\ge0$ such that $x_{n,\infty}=y_{m,\infty}$.


A geodesic $x_{0,\infty}$ is $\xi$-directed  or  a $\xi$-geodesic if $x_n/\abs{x_n}_1\to\xi$ for $\xi\in\Uset$, and simply     directed if it is $\xi$-directed for some $\xi$.   Flat segments of $\gpp$  on $\Uset$ prevent us from asserting that all geodesics are directed.  Hence we say more generally for  a subset $\cV\subset\Uset$ that  a geodesic $x_{0,\infty}$ is $\cV$-directed 
if all the limit points of 
$x_n/\abs{x_n}_1$ lie in $\mathcal V$.

Recall that $\gpp$ is {\sl strictly concave}  if   there is no nondegenerate line segment on $\ri\Uset$ on which $\gpp$ is   linear.    Recall also the definition of $\Uset_{\xi\pm}$ from \eqref{eq:sector1}  and  $\Uset_{\xi}=\Uset_{\xi+}\cup\Uset_{\xi-}$.

\begin{theorem}
\begin{enumerate}[\ \ \rm(i)]	
	\item\label{thm-1:geod:exist}   The following statements hold for $\P$-almost every $\w$.  For every $u\in\Z^2$ and  $\xi\in\Uset$ there exist  infinite $\Uset_{\xi+}$-  and $\Uset_{\xi-}$-directed geodesics starting from $u$.  Every geodesic   is $\Uset_\xi$-directed for some $\xi\in\Uset$.
\item\label{thm-1:geod:direct}	If $\gpp$ is strictly concave   then,  	
		with $\P$-probability one, every geodesic is directed.
\item\label{thm-1:geod:cont}Suppose  $\P\{\w_0\le r\}$ is a continuous function of $r\in\R$.
Fix $\xi\in\Uset$ and assume $\ximin, \xi,\ximax\in\Diff$.  Then   
		  $\P$-almost surely  there is a unique $\Uset_\xi$-geodesic out of every $u\in\Z^2$  and all these geodesics coalesce.  
\end{enumerate}
\label{thm-1:geod}\end{theorem}



In the next theorem we repeat  the  assumptions of   Theorem \ref{thm:buse} to have a Busemann function and then show that in a direction that satisfies the differentiability assumption there can be no other geodesic  except a Busemann geodesic.   
 
\begin{theorem}\label{thm:buse-geo}
As in  Theorem \ref{thm:buse} let  $\xi\in\ri\Uset$ with  $\Uset_\xi=[\ximin,\ximax]$ satisfy $\ximin, \xi, \ximax\in\Diff$.  Let $B$ be the limit from \eqref{eq:grad:coc1}. The following events have $\P$-probability one.
\begin{enumerate}[\ \ \rm(i)]
\item\label{thm:buse-geo:i} Every up-right path $x_{0,\infty}$ such that $\w_{x_k}=B(x_k,x_{k+1})$ for all $k\ge0$
is an infinite geodesic. We call such a path a {\sl Busemann geodesic}.
\item\label{thm:buse-geo:ii} Every geodesic $x_{0,\infty}$ that satisfies 
	\be\label{geod-98} \ximin\cdot e_1\le\varliminf_{n\to\infty}\frac{x_n\cdot e_1}n\le\varlimsup_{n\to\infty}\frac{x_n\cdot e_1}n\le\ximax\cdot e_1\ee
 is a Busemann geodesic.
\item\label{thm:buse-geo:iii}  For each $m\ge0$, for any sequence $v_n$ as in \eqref{eq:vn},  there exists $n_0$ such that if $n\ge n_0$, then for any geodesic $x_{0,\abs{v_n}_1}$ from  $x_0=0$ to $v_n$ we have $B(\w,x_i,x_{i+1})=\w_{x_i}$ for all $0\le i\le m$. 
\end{enumerate}
\end{theorem}

When the  distribution of  $\w_0$ is  not   continuous uniqueness of geodesics (Theorem \ref{thm-1:geod}\eqref{thm-1:geod:cont}) cannot hold. Hence we consider leftmost and rightmost geodesics.  The  {\sl leftmost } geodesic $\underline{x}_{\;\centerdot}$ (between two given points or in a given direction)  satisfies 
	$	\underline{x} _k \cdot e_1 \le x _k \cdot e_1 $ 
for any geodesic $x_\centerdot$  of the same category.   The rightmost geodesic satisfies the opposite inequality.  

\begin{theorem}\label{thm:lr-geod}
Let  $\xi\in\ri\Uset$ with  $\Uset_\xi=[\ximin,\ximax]$ satisfying $\ximin, \xi, \ximax\in\Diff$.   The following statements hold $\P$-almost surely.
\begin{enumerate}[\ \ \rm(i)]
\item\label{thm:lr-geod:i} There exists a leftmost $\Uset_\xi$-geodesic from each $u\in\Z^2$ and 
all these leftmost geodesics coalesce.  Same statement for rightmost. 
\item\label{thm:lr-geod:ii}  For any $u\in\Z^2$ and  sequence $v_n$ as in \eqref{eq:vn}, 
the leftmost geodesic from $u$ to $v_n$ converges to the leftmost $\Uset_\xi$-geodesic  from $u$ given in part 
\eqref{thm:lr-geod:i}.   
A similar statement holds for  rightmost geodesics. 
\end{enumerate}
\end{theorem}

Theorems \ref{thm-1:geod}, \ref{thm:buse-geo}, and \ref{thm:lr-geod} are proved at the end of Section \ref{sec:geod}.

\subsection{Competition interface}  \label{sec:ci} 

For this subsection assume that  $\P\{\w_0\le r\}$ is a continuous function of $r\in\R$.  Then  with probability one  no two finite paths of any lengths have equal weight and consequently for  any $v\in\Z_+^2$ there is a unique finite geodesic between $0$ and $v$.  Together these finite geodesics form the {\sl geodesic tree} $\cT_0$ rooted at $0$  that spans $\Z_+^2$.   The two subtrees 
rooted at $e_1$ and $e_2$ are separated by an up-right path $\varphi=(\varphi_k)_{k\ge 0}$  on the lattice $(\frac12,\frac12)+\Z_+^2$ with $\varphi_0=(\frac12,\frac12)$.    The path $\varphi$ is called the {\sl competition interface}.  
The term  comes from the interpretation that the subtrees 
 are two competing   infections 
on the lattice \cite{Fer-Mar-Pim-09,Fer-Pim-05}.  See Figure \ref{cif-fig}.

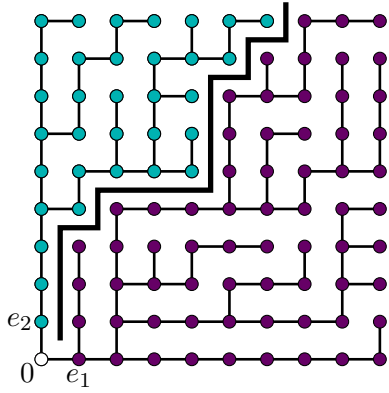
\begin{figure}[h]
\begin{center}

\begin{tikzpicture}[>=latex,scale=0.5]

\definecolor{sussexg}{rgb}{0,0.7,0.7}
\definecolor{sussexp}{rgb}{0.4,0,0.4}

\draw(-0.37,-0.35)node{$0$};
\draw(-0.57,1)node{$e_2$};
\draw(1,-0.5)node{$e_1$};


\draw[line width=1pt](9,0)--(0,0)--(0,9);     
\draw[line width=1pt](2,0)--(2,4)--(5,4)--(5,7)--(7,7)--(7,9)--(9,9);
\draw[line width=1pt](1,0)--(1,3);
\draw[line width=1pt](0,4)--(1,4)--(1,5)--(2,5)--(2,7);
\draw[line width=1pt](2,5)--(4,5)--(4,6);
\draw[line width=1pt](3,5)--(3,6)--(3,8)--(5,8)--(5,9)--(6,9);
\draw[line width=1pt](3,7)--(4,7);
\draw[line width=1pt](6,7)--(6,8);
\draw[line width=1pt](0,6)--(1,6)--(1,8)--(2,8)--(2,9)--(3,9);
\draw[line width=1pt](4,8)--(4,9);
\draw[line width=1pt](0,9)--(1,9);
\draw[line width=1pt](5,4)--(7,4)--(7,5)--(9,5)--(9,8);
\draw[line width=1pt](6,4)--(6,6)--(7,6);
\draw[line width=1pt](8,5)--(8,8);
\draw[line width=1pt](2,1)--(8,1)--(8,2)--(9,2);
\draw[line width=1pt](2,2)--(4,2)--(4,3)--(6,3);
\draw[line width=1pt](9,0)--(9,1);

\draw[line width=1pt](5,1)--(5,2)--(7,2)--(7,3);
\draw[line width=1pt](8,2)--(8,4)--(9,4);
\draw[line width=1pt](3,2)--(3,3);
\draw[line width=1pt](8,3)--(9,3);


\draw[line width=2.0pt](.5,.5)--(0.5,3.5)--(1.5,3.5)--(1.5, 4.5)--(4.5,4.5)--(4.5,7.5)--(5.5,7.5)--(5.5,8.5)--(6.5,8.5)--(6.5, 9.5);   
\foreach \x in {0,...,9}{
             \foreach \y in {0,...,9}{
\draw[ fill =sussexp](\x,\y)circle(1.7mm);    
}}

\foreach \x in{0,...,6}{
		\draw[ fill =sussexg](\x, 9)circle(1.7mm);    
				}

\foreach \x in{0,...,5}{
		\draw[ fill =sussexg](\x, 8)circle(1.7mm);    
				}

\foreach \x in{0,...,4}{
		\foreach \y in {5,6,7}{
		\draw[fill =sussexg](\x, \y)circle(1.7mm);    
				}}

\foreach \x in{0,...,1}{
		\draw[ fill =sussexg](\x, 4)circle(1.7mm);    
				}

\foreach \x in{0,...,0}{
		\foreach \y in {1,2,3}{
		\draw[fill =sussexg](\x, \y)circle(1.7mm);    
				}}

\draw[fill=white](0,0)circle(1.7mm); 

\end{tikzpicture}

\end{center}
\caption{\small   The geodesic tree $\mathcal{T}_0$ rooted at $0$.  
The competition interface  (solid line)  emanates from $(\tfrac12, \tfrac12)$ and separates 
the   subtrees of $\mathcal{T}_0$ rooted at $e_1$   and $e_2$. }
\label{cif-fig}\end{figure}

 Adopt  the convention that $\Gpp_{e_i, ne_j}=-\infty$ for $i\ne j$ and $n\ge 0$ (there is  no admissible  path from $e_i$ to $ ne_j$).    Fix $n\in\N$.   As $v$ moves to the right along $\abs{v}_1=n$, the function $\Gpp_{e_2,v}-\Gpp_{e_1,v}$ is nonincreasing.  This  is a consequence of Lemma \ref{lm:new:comp0}  below.   There is a unique $0\le k<n$ such that   
 \be\label{ci:19} \Gpp_{e_2,(k,n-k)}-\Gpp_{e_1,(k,n-k)}>0> \Gpp_{e_2,(k+1,n-k-1)}-\Gpp_{e_1,(k+1,n-k-1)}. \ee  This identifies the point  $\varphi_{n-1}=(k+\tfrac12,n-k-\tfrac12)$. 
 
 \begin{theorem}\label{thm:ci-1} 
Assume $\P\{\w_0\le r\}$ is continuous in $r$ and that $\gpp$ is differentiable at the endpoints of all   its linear segments.
 Then we have the law of large numbers  
\begin{align}\label{ci-lln2}
\cid(\w)= \lim_{n\to\infty} n^{-1}\varphi_n(\w)\qquad \text{$\P$-a.s.}   \end{align}
The   limit $\cid$ is almost surely an exposed point in $\ri\Uset$  and the 
  support  of its distribution intersects every open interval outside  the closed  line segments on which $\gpp$ is linear.    
\end{theorem}

\begin{remark}\label{rem:cif-var}    Assume  that  $\P\{\w_0\le r\}$ is continuous and  that differentiability holds everywhere on $\ri\Uset$ so that no caveats are needed.  Connecting back  to the variational formulas,   the maximum   in \eqref{var15} and \eqref{var151} is taken at $i=2$ if $\xi\cdot e_1<\cid\cdot e_1$ and at $i=1$ if $\xi\cdot e_1>\cid\cdot e_1$.   This is a consequence of the following two facts:   (i)  $\w_0=B^\xi(0,e_1)\wedge B^\xi(0,e_2)$  as follows from \eqref{w-geod}, and (ii)  for $\xi\cdot e_1<\cid\cdot e_1<\zeta\cdot e_1$  we have   $B^\xi(0,e_1)> B^\xi(0,e_2)$ and  $B^\zeta(0,e_1)< B^\zeta(0,e_2)$.  The second fact will become clear in Section \ref{sec:ci-pf}.  
\end{remark}

The competition interface is a natural direction in which there are two geodesics from $0$.  Note that   nonuniqueness in the random direction $\cid$ does not violate the almost sure uniqueness in a fixed direction given in Theorem \ref{thm-1:geod}\eqref{thm-1:geod:cont}.    

 \begin{theorem}\label{thm:ci-2} 
Assume $\P\{\w_0\le r\}$ is continuous in $r$.
\begin{enumerate}[\ \ \rm(i)]
\item\label{thm:ci-2:i} Assume  $\gpp$ is differentiable at the endpoints of all   its linear segments.   Then  
$\P$-almost surely, for every $x\in\Z^2$, there exist at least two $\Uset_{\cid(T_x\w)}$-geodesics out of $x$ that do not coalesce.  
\item\label{thm:ci-2:ii}  Assume  $\gpp$ is strictly concave.  Then  with $\P$-probability one and for any $x\in\Z^2$, there cannot be two distinct geodesics from $x$ with a common direction other than   $\cid(T_x\w)$.  
\end{enumerate} 
\end{theorem}

 For the exactly solvable  corner growth model with exponential weights Coupier \cite{Cou-11} proved that the set of directions with two non-coalescing  geodesics in $\Z_+^2$  is countable and dense in $\Uset$.   Here we have a partial result towards characterizing this set as $\{\cid(T_x\w)\}_{x\in\Z_+^2}$.   Partial, because we consider only pairs of geodesics from a common initial point.   

Point \eqref{thm:ci-2:i} of Theorem \ref{thm:ci-2}  is actually true without the differentiability assumption,  but at this stage of the paper we have no definition of $\cid$ without that assumption.  This will change in Theorem \ref{th:two-geo} in Section \ref{sec:ci-pf}.     
In point \eqref{thm:ci-2:ii} above $\gpp$ has no linear segments and so the differentiability of $\gpp$ at endpoints of linear segments is vacuously true.  
 
 Theorems \ref{thm:ci-1} and  \ref{thm:ci-2} are  proved in Section \ref{sec:ci-pf}.   An additional fact proved there is that   $\P(\cid=\xi)>0$ is possible only if $\xi\notin\Diff$.   In light of the expectation that $\gpp$ is differentiable, the expected result is that $\cid$ has a continuous distribution.  
 

When weights $\w_0$ do  not have a  continuous distribution, there are two competition interfaces: one  for the  tree of leftmost geodesics   and one 
for the tree of  rightmost geodesics.      We compute the limit distributions of the two competition interfaces for geometric weights in Sections \ref{subsec:solv} and \ref{sec:solv}. 


\subsection{Exactly solvable models}\label{subsec:solv}
We illustrate our results in the two exactly solvable cases: the distribution of the  mean $\Ew$ weights $\w_x$ is 
\be\label{cases7} \begin{aligned}
\text{exponential:  } \P\{\w_x\ge t\}&=\Ew^{-1} e^{-t/\Ew} \text{ for  $t\ge0$ with $\sigma^2=\Ew^2$,   }\\
\text{or} \; \; \text{geometric: } \P\{\w_x=k\}&=\Ew^{-1}(1-\Ew^{-1})^{k-1} \text{ for  $k\in\N$ with  $\sigma^2=\Ew(\Ew-1)$.}
\end{aligned}\ee
Calculations behind the claims below are sketched in Section \ref{sec:solv}.

For both cases in \eqref{cases7}  the point-to-point limit   function is 
 \[\gpp(\xi)=\Ew(\xi\cdot e_1+\xi\cdot e_2)+2\sigma\sqrt{(\xi\cdot e_1)(\xi\cdot e_2)}\,.\]
 In the exponential case this formula was first derived by Rost \cite{Ros-81} (though without the last-passage formulation)  while early derivations of the geometric case appeared in  \cite{Coh-Elk-Pro-96, Joc-Pro-Sho-26-, Sep-98-mprf-1}.  
  Convex duality \eqref{eq:duality} becomes 
\begin{align*}
&\xi\in\ri\Uset\text{ is dual to $h$ if and only if }\\
&\qquad  h=\big(\Ew+\sigma^2\sqrt{{\xi\cdot e_1}/{\xi\cdot e_2}}+t,\, \Ew+\sqrt{{\xi\cdot e_2}/{\xi\cdot e_1}}+t\big),\quad t\in\R\,.\end{align*}
This in turn gives an explicit formula for $\gpl(h)$.

Since the $\gpp$ above  is  differentiable and strictly concave, all points of $\ri\Uset$ are exposed points of differentiability.   
 Theorem \ref{thm:buse} implies that Busemann functions \eqref{eq:grad:coc1} exist in all directions  $\xi\in\ri\Uset$. 
They  minimize formulas \eqref{eq:g:K-var} and \eqref{eq:gpp:K-var}
as given in \eqref{var15} and \eqref{var151}. 
For each $\xi\in\ri\Uset$ the processes $\{ B^\xi(ke_1, (k+1)e_1): k\in\Z_+\}  $
and $\{ B^\xi(ke_2, (k+1)e_2): k\in\Z_+\}  $ are i.i.d.\ processes  independent of each other, exponential or geometric depending on the case,  with means 
\be \label{geom:B}  
\begin{aligned} 
\E[B^\xi(ke_1, (k+1)e_1)]&=\Ew+\sigma  \sqrt{\xi\cdot e_2/\xi\cdot e_1}\\
 \E[B^\xi(ke_2, (k+1)e_2)]&=\Ew+\sigma \sqrt{\xi\cdot e_1/\xi\cdot e_2}\,.
 \end{aligned}\ee

  Section \ref{subsec:geo} gives the following results on geodesics.
Every infinite geodesic has a direction and for every fixed direction $\xi\in\ri\Uset$ there exists a $\xi$-geodesic. 
In the exponential case $\xi$-geodesics are unique and coalesce. In the geometric case uniqueness cannot hold, but there exists a unique leftmost $\xi$-geodesic out of each lattice point and
these leftmost $\xi$-geodesics coalesce. The same holds for rightmost $\xi$-geodesics. Finite (leftmost/rightmost) geodesics from $u\in\Z^2$ to $v_n$ converge to infinite (leftmost/rightmost) 
$\xi$-geodesics out of $u$, as $v_n/\abs{v_n}_1\to\xi$ with $\abs{v_n}_1\to\infty$. 

In the exponential case the  distribution of the asymptotic direction   $\cid$ of the competition interface  given by Theorem \ref{thm:ci-1}  can be computed explicitly.  For the angle  $\theta_*=\tan^{-1}(\cid\cdot e_2/\cid\cdot e_1)$ of the vector $\cid$, \begin{align}\label{eq:ci-exp}
\P\{\theta_*\le t\}=\frac{\sqrt{\sin t}}{\sqrt{\sin t}+\sqrt{\cos t}},\qquad t\in[0,\pi/2].
\end{align}

In the exponential  case these results for  geodesics and the competition interface were shown in \cite{Fer-Pim-05}.  This paper utilized   techniques  for geodesics  from \cite{New-95} and   the coupling of the exponential corner growth model with the totally asymmetric simple exclusion process (TASEP).   For this case our approach provides new proofs.  

The model with  geometric weights has a tree of leftmost geodesics with competition interface $\varphi^{(l)}=(\varphi^{(l)}_k)_{k\ge 0}$  and a tree of  rightmost geodesics with competition interface $\varphi^{(r)}=(\varphi^{(r)}_k)_{k\ge 0}$.    Note that  $\varphi^{(r)}$ is to the {\sl left} of  $\varphi^{(l)}$  because in \eqref{ci:19} there is now a middle range   $\Gpp_{e_2,(k,n-k)}-\Gpp_{e_1,(k,n-k)}=0$ that is to the right (left) of  $\varphi^{(r)}$ ($\varphi^{(l)}$).  Strict concavity of the limit $\gpp$ implies (with the arguments of Section \ref{sec:ci-pf})  the almost sure  limits 
\[  n^{-1}  \varphi^{(l)}_n \to \cid^{(l)} \quad \text{and}\quad 
n^{-1}  \varphi^{(r)}_n \to \cid^{(r)}. \]
The angles $\theta_*^{(a)} =\tan^{-1} ( \cid^{(a)}\!\cdot e_2/\cid^{(a)}\!\cdot e_1)$  ($a\in\{l,r\}$)  have these distributions (with $p_0=m_0^{-1}$ denoting the success probability of the geometric): for  $t\in[0,\pi/2]$ 
\be\label{geom:theta} \begin{aligned} 
\P\{\theta_*^{(r)}\le t\}&=\frac{\sqrt{(1-p_0)\sin t}}{\sqrt{(1-p_0)\sin t}+\sqrt{\cos t}}\,  \\
\text{and}\qquad  \P\{\theta_*^{(l)}\le t\}&=\frac{\sqrt{\sin t}}{\sqrt{\sin t}+\sqrt{(1-p_0)\cos t}}\, .\end{aligned}\ee
Taking $p_0\to0$ recovers \eqref{eq:ci-exp} of the exponential case.  
For the details, see Section \ref{sec:solv}.

\subsection{Flat edge in the percolation cone}\label{sec:cone}

We describe a known nontrivial example where the assumption of  differentiable endpoints of  a maximal linear segment is satisfied.  A short detour into oriented percolation is needed.  

In {\sl oriented site percolation} vertices of $\Z^2$ are assigned i.i.d.\ $\{0,1\}$-valued random variables $\{\sigma_z\}_{z\in\Z^2}$ with $p_1=\P\{\sigma_0=1\}$.
For  points  $u\le v$ in $\Z^2$  we write   $u\to v$ (there is an open path from $u$   to $v$) if there exists an   up-right path $u=x_0, x_1,\dotsc, x_m=v$ with $x_{i}-x_{i-1}\in\{e_1,e_2\}$, $m=\abs{v-u}_1$, and such that $\sigma_{x_i}=1$ for $i=1,\dotsc,m$.  (The openness of a path does not depend on the weight at the initial point of the path.)  
  The {\sl percolation}  event  $\{u\to\infty\}$ is the existence of  an infinite open up-right path from point $u$.  
There exists a critical threshold $\opc \in (0,1)$ such that if $p_1\le\opc$ then $\P\{ 0\to{\infty}\} = 0$ and if $p_1 > \opc$ then $\P\{ 0\to\infty \} > 0$. The case $p_1=\opc$ is due to \cite{Bez-Gri-90}. The other facts we need about oriented site percolation are proved in \cite{Dur-84} for oriented edge percolation. The proofs apply to site percolation just as well.

Let  $\cla_n = \{ u \in \Z^2_+: \abs{u}_1=n, \, 0 \to u\}$ denote the set of vertices on level $n$ that can be reached from the origin along open paths.   The  right edge   $a_n = \max_{u \in \cla_n}\{ u \cdot e_1 \}$ is  defined on the event $\{\cla_n \neq \varnothing\}$.    When $p_1\in(\opc,1)$ there exists a  constant $\beta_{p_1} \in (1/2, 1)$ such that   \cite[(7) on p.~1005]{Dur-84}
\[\lim_{n\to \infty} \frac{a_n}{n}\one\{ 0\to{\infty}\} = \beta_{p_1} \one\{ 0\to{\infty}\} \qquad\P\textrm{-a.s. }  
\]
Let  $\etamax=(\beta_{p_1}, 1-\beta_{p_1})$ and $\etamin=(1-\beta_{p_1}, \beta_{p_1})$. The {\sl percolation cone} is  the set $\{\xi\in\R_+^2:  \xi/\abs{\xi}_1\in [\etamin, \etamax]\}$.  The edge speed can also be read off the process started from the antidiagonal half-line $\wt\Z_-=\{(-k,k):k\in\Z_+\}$. Let $a_n(\wt\Z_-)$ be the largest first coordinate of a point of level $n$ that is reached by an open path from some point of $\wt\Z_-$. Such points exist almost surely, since the starting points give infinitely many independent chances. By \cite[(5)--(6) on pp.~1004--1005 and p.~1007]{Dur-84}, for every $p_1\in[\opc,1)$ there is a constant $\beta_{p_1}$ such that
\begin{equation}\label{an/n}
\lim_{n\to\infty}\frac{a_n(\wt\Z_-)}n=\beta_{p_1}\qquad\text{almost surely and in }L^1 .
\end{equation}
For $p_1>\opc$, this $\beta_{p_1}$ coincides with the edge speed defined above \cite[(2) and (7) on pp.~1003--1005]{Dur-84}. Moreover, $\beta_{\opc}=\tfrac12$ \cite[(1) on p.~1025]{Dur-84}, and the map $p\mapsto\beta_p$ is right-continuous at $\opc$ \cite[Remarks on p.~1007]{Dur-84} and continuous on $(\opc,1)$ \cite[(3) on p.~1031]{Dur-84}.\footnote{\cite{Dur-84} works in rotated coordinates: the half-line process starts from the points $(y,0)$ with $y\le0$ even, and at time $n$ its positions $x$ satisfy $x\le n$ and $x+n$ even. His edge speed $\alpha_p$ corresponds to $\beta_p=(1+\alpha_p)/2$; in particular $\beta_p>\tfrac12$ is $\alpha_p>0$, and $\beta_{\opc}=\tfrac12$ is $\alpha(p_c)=0$. The right edge $\bar r_n$ of the half-line process satisfies $\bar r_n/n\to\alpha$ almost surely \cite[(5) on p.~1004]{Dur-84}, where $\alpha=\inf_{n\ge1}E(\bar r_n/n)$ \cite[(6) on p.~1004]{Dur-84}, and the convergence holds in $L^1$ when $\alpha>-\infty$ \cite[p.~1007]{Dur-84}, which is the case for $p\ge\opc$.}

The point of this for the corner growth model is that if the $\w$ weights have a  maximum that percolates,  $\gpp$ is linear on the percolation cone and differentiable on the edges.   This is the content of the next theorem.  



%


\begin{theorem}\label{th:flat-edge}
Assume that $\{\w_x\}_{x\in\Z^2}$ are i.i.d., $\w_x\le1$, $p_1=\P\{\w_0=1\}<1$, and $\E[\w_0^-]<\infty$.
Then $\gpp\le1$ on $\Uset$. Moreover:
\begin{enumerate}[\rm(a)] \itemsep=3pt
\item If $0\le p_1<\opc$, then $\gpp(\xi)<1$ for all $\xi\in\Uset$.
\item If $p_1=\opc$, then $\gpp(\xi)=1$ if and only if $\xi=(\tfrac12,\tfrac12)$, and $\gpp$ is differentiable at $(\tfrac12,\tfrac12)$.
\item If $\opc<p_1<1$, then $\gpp(\xi)=1$ if and only if $\xi\in[\etamin,\etamax]$, and $\gpp$ is differentiable at $\etamin$ and $\etamax$.
\end{enumerate}
\end{theorem}

\begin{remark}
Theorem \ref{th:flat-edge} does not require the standing assumption \eqref{2d-ass}. It assumes neither $p>2$ moments nor a lower bound on the weights, only $\E\w_0^-<\infty$ and  $\w_x\le1$. Under these hypotheses the limit $\gpp$ exists and is finite, concave, homogeneous and continuous, and the almost sure shape theorem \eqref{lln5} holds by Theorem \ref{th:shape}. These are the only consequences of the moment assumption used in the proof in Appendix~\ref{sec:cone-pf}. 
\end{remark}


Theorem \ref{th:flat-edge} summarizes a development  that goes through   papers \cite{Auf-Dam-13,Dur-Lig-81,Mar-02}.  The proofs in the literature are for first-passage percolation.    We give a proof of Theorem \ref{th:flat-edge} in Appendix \ref{sec:cone-pf}, by adapting and simplifying  the  first-passage percolation arguments for the directed corner growth model.  

As a corollary,  our results that assume differentiable endpoints of a maximal linear segment are valid for the percolation cone.  

\begin{theorem}
Assume \eqref{2d-ass}, $ \w_x\le 1$ and   $\opc<p_1=\P\{\w_0=1\}<1$.  
There exists a shift-covariant $L^1(\P)$ cocycle $\{B(x,y):x,y\in\Z^2\}$ and an event  $\Omega_0$ with $\P(\Omega_0)=1$ such that  the following statements hold for each 
$\w\in\Omega_0$.  Let   $v_n\in\Z_+^2$ be a sequence  such that 
\[\abs{v_n}_1\to\infty\quad\text{and}\quad 1-\beta_{p_1}\le\varliminf_{n\to\infty}\frac{v_n\cdot e_1}{\abs{v_n}_1}\le\varlimsup_{n\to\infty}\frac{v_n\cdot e_1}{\abs{v_n}_1}\le\beta_{p_1}. \]
Then 
	\[ B(\w, x,y) = \lim_{n\to \infty} \big( \Gpp_{x, v_n}(\w) - \Gpp_{y, v_n}(\w) \big)  \]
						 for all  $x,y\in\Z^2$. 
  For each $m\ge0$  there exists $n_0$ such that if $n\ge n_0$, then  any geodesic $x_{0,\abs{v_n}_1}$ from  $x_0=0$ to $v_n$ satisfies  $B(\w,x_i,x_{i+1})=\w_{x_i}$ for all $0\le i\le m$. 

Furthermore,  
\[\E[B(x,x+e_1)]=  \E[B(x,x+e_2)]=1. \] 
\end{theorem}

This completes the presentation of results and we turn to developing the proofs. 

\section{Convex duality}
\label{sec:duality}

  By homogeneity we can represent  $\gpp$ by a single variable function.  A way of doing this that ties in naturally with the  queuing  theory arguments we use later  is   to   define  
  	\begin{align}
		\gppa(s)=\gpp(1,s)=\gpp(s,1)\quad \text{for}  \quad 0\le  s<\infty.    \label{eq:gbar}
	\end{align}
Function   $\gppa$ is real-valued, continuous and  concave.  Consequently 
one-sided derivatives $\gppa'(s\pm)$ exist and are monotone: $\gppa'(s_0+)\ge\gppa'(s_1-)\ge\gppa'(s_1+)$  for $0\le s_0<s_1$.  
Symmetry and homogeneity of $\gpp$ give  $\gppa(s)=s\gppa(s^{-1})$. 
 
\begin{lemma} 
\label{gppa-lm}
The derivatives satisfy  $\gppa'(s\pm)>\Ew$ for all $s\in\R_+$, 	
 	$\gppa'(0+)=\infty$, and    
	$\gppa'(\infty-)\equiv\ddd\lim_{s\to\infty}\gppa'(s\pm) =\gppa(0)=\Ew$.   
	 \end{lemma}

\begin{proof}
	The boundary shape universality of J.~Martin 
	\cite[Theorem 2.4]{Mar-04} says that 
		\begin{align}
		\label{eq:g-asym}
			\gppa(s)=\Ew+2\sigma\sqrt{ s}+o(\sqrt{s}\,)  \quad \text{as} \;  s\searrow 0.  
		\end{align}
	This gives $\gppa(0)=\Ew$ and 
  	$\gppa'(0+)=\infty$.
 	Lastly,  
		\[
			\gppa'(\infty-)=\lim_{s\to\infty}s^{-1}\gppa(s)=\lim_{s\to\infty}\gppa(s^{-1})=\gppa(0)=\Ew.
		\]
	Martin's asymptotic \eqref{eq:g-asym}  and $\gppa(s)=s\gppa(s^{-1})$ give 
		\begin{align}
		\label{eq:g-asym2}
			\gppa(s)=s\Ew+2\sigma\sqrt{ s}+o(\sqrt{s}\,)  \quad \text{as} \;  s\nearrow \infty.  
		\end{align}
	This is incompatible with having $\gppa'(s)=\Ew$ for $s\ge s_0$ for any
	$s_0<\infty$.  
\end{proof}

The lemma above has two 
important geometric consequences: (i)   
any subinterval of $\Uset$ on which $\gpp$ is linear must lie in the interior $\ri\Uset$ and (ii) 
 the boundary $\{\xi:  \gpp(\xi)=1\}$  of the limit shape is asymptotic to the axes.  

Define  
	\begin{align}
		\f(\alpha)=\sup_{s\ge0}\{\gppa(s)-s\alpha\} \quad \text{for}  \quad  \Ew< \alpha<\infty. \label{eq:flux}
	\end{align}

\begin{lemma}\label{lm:f-properties}
	Function  $f$ is a strictly decreasing, continuous and convex involution 
	of the interval $(\Ew,\infty)$ onto itself, with  limits $\f(\Ew+)=\infty$ and $\f(\infty-)=\Ew$. That $f$ is an   involution means that $\f(\f(\alpha))=\alpha$.  
\end{lemma}

\begin{proof}
   	Asymptotics 
   	\eqref{eq:g-asym} and  \eqref{eq:g-asym2} imply  that   
	$\Ew< f(\alpha)<\infty$ for all $\alpha>\Ew$ and also that 
	the supremum in \eqref{eq:flux} is attained at some $s$. 
 	Furthermore, $\alpha<\beta$ implies $\f(\beta)=\gamma(s_0)-s_0\beta$ 
	with $s_0>0$ and $\f(\beta)<\gamma(s_0)-s_0\alpha\le\f(\alpha)$. 
	As a supremum of linear functions $\f$ is convex, and  
	hence  continuous on the open interval $(\Ew,\infty)$. 

We show how 	the   symmetry of $\gpp$ implies   that $\f$ is an involution.  
	By  concavity of $\gppa$,   
		\begin{align}
		\label{eq:f=g-sa}
			\f(\alpha)=\gppa(s)-s\alpha\quad\text{if and only if  $\alpha\in[\gppa'(s+),\gppa'(s-)]$}
		\end{align}
	and by Lemma \ref{gppa-lm} the intervals on the right cover $(\Ew,\infty)$. 
	Since $\f$ is strictly decreasing the above is the same as 
		\begin{align}
		\label{eq:a=g-f/s}
			\alpha=\gppa(s^{-1})-s^{-1}\f(\alpha)\quad\text{if and only if $\f(\alpha)\in[\f(\gppa'(s-)),\f(\gppa'(s+))]$.}
		\end{align}
	Differentiating $\gppa(s)=s\gppa(s^{-1})$ gives 
		\begin{align}
		\label{eq:gppa(1/s)}
			\gppa'(s\pm)=\gppa(s^{-1})-s^{-1}\gppa'(s^{-1}\mp).
		\end{align} 
	By \eqref{eq:f=g-sa} and  \eqref{eq:gppa(1/s)} 
	the condition in \eqref{eq:a=g-f/s} can be rewritten as
		 \begin{align}\label{eq:f=ga'}
			\f(\alpha)\in[\gppa(s)-s\gppa'(s-),\gppa(s)-s\gppa'(s+)]=[\gppa'(s^{-1}+),\gppa'(s^{-1}-)].
		 \end{align}
 Combining this with  \eqref{eq:f=g-sa} and  \eqref{eq:a=g-f/s} shows that $\alpha=\f(\f(\alpha))$.  
	The claim about the limits follows from $\f$ being a decreasing involution.
\end{proof}
%
%

Extend these functions to the entire real line by  $\gppa(s)=-\infty$ when $s<0$ and $\f(\alpha)=\infty$ when $\alpha\le\Ew$.  Then convex duality gives 
	\be\label{eq:ga=f*}
		\gppa(s)=\inf_{\alpha>\Ew}\{\f(\alpha)+s\alpha\}.
	\ee

The natural bijection between $s\in(0,\infty)$ and $\xi\in\ri\Uset$ that goes together with \eqref{eq:gbar} is 
\be\label{s-xi} s=\xi\cdot e_1/\xi\cdot e_2. \ee  
Then direct differentiation, \eqref{eq:f=g-sa} and \eqref{eq:gppa(1/s)} give 
		\begin{align}
		\label{grad g-1}
			\nabla\gpp(\xi\pm) 
			=\bigl(\gppa'(s\pm),\gppa'(s^{-1}\mp)\bigr)
			=\bigl(\gppa'(s\pm),\f(\gppa'(s\pm))\bigr).
		\end{align}
Since  $\f$ is linear on $[\gppa'(s+),\gppa'(s-)]$,    we  get the following connection between the gradients of   $\gpp$  and the graph of $\f$:   
		\be
		\label{f-9} 
			[\nabla\gpp(\xi+),\nabla\gpp(\xi-)]=\{(\alpha,\f(\alpha)):\alpha\in[\gppa'(s+),\gppa'(s-)]\} \qquad\text{for $\xi\in\ri\Uset$}.
		\ee
		
%

The next theorem details the duality between tilts $h$ and velocities $\xi$. 
 
\begin{theorem}\label{th:tilt-velocity}  
 	\begin{enumerate}[\ \ \rm(i)]
%
	\item\label{th:tilt-velocity:h->xi} Let $h\in\R^2$. There exists a unique $t=t(h)\in\R$ such that 
		\be
			h-t(e_1+e_2)\in-[\nabla\gpp(\xi+),\nabla\gpp(\xi-)]
		\label{eq:h-t}
		\ee
	 for some $\xi\in\ri\Uset$. 
	 The set  of  $\xi$  for which \eqref{eq:h-t} holds is a nonempty {\rm(}but possibly degenerate{\rm)} line segment 
 	$[\ximin(h),\ximax(h)]\subset\ri\Uset$.    If  $\ximin(h)\ne\ximax(h)$ then  $[\ximin(h),\ximax(h)]$ is a maximal line segment on which $\gpp$ is linear.  
	\item\label{th:tilt-velocity:xi-h} $\xi\in\ri\Uset$ and $h\in\R^2$ satisfy duality \eqref{eq:duality} if and only if     \eqref{eq:h-t} holds. 
	\end{enumerate}
\end{theorem}

\begin{proof}
	The graph $\{(\alpha,\f(\alpha)):\alpha>\Ew\}$ is strictly decreasing with 
	limits $f(\Ew+)=\infty$ and $f(\infty-)=\Ew$.  Since every 45 degree diagonal 
	intersects it at a unique point, the equation
	\begin{align}\label{eq:h-(a,t)}
			h=-(\alpha,\f(\alpha))+t(e_1+e_2) 
		\end{align}
defines  a bijection  $\R^2\ni h\longleftrightarrow (\alpha,t)\in(\Ew,\infty)\times\R$  illustrated in Figure \ref{fig:3.3}. 
Combining this with \eqref{f-9} shows that \eqref{eq:h-t} happens for a unique $t$ and for at least one $\xi\in\ri\Uset$.  

\begin{figure}[h]
	\begin{center}
		\begin{tikzpicture}[>=latex,  scale=.4]

			\draw[<->](0,15)--(0,0)--(15,0)node[right]{$\alpha$};
			\draw[dashed] (2,0)node[below]{$m_0$}--(2,15);
			\draw[dashed] (0,2)node[left]{$m_0$}--(15,2);
			\draw[semitransparent] (0,0) -- (13,13);
			
%
			\draw [line width=2.5pt, color= nicosred](3,6)--(4,4)--(6,3);
			\draw[domain=6:15,variable=\x,color=nicosred,line width=2.5pt]plot[smooth]({\x},{2+6/\x});
			\draw[domain=6:15,variable=\y,color=nicosred,line width=2.5pt]plot[smooth]({2+6/\y},{\y});
			
			\draw[semitransparent] (25/4,1)--(15,39/4);
			\draw(11,5.9)node[left]{$t(e_1+e_2)$}; 
			\draw[->, line width=2pt](13,31/4)node[right]{$-h$}--(8,11/4);
			\shade[ball color=red](8,11/4)circle(2.5mm);
			\shade[ball color=red](13,31/4)circle(2.5mm);
			\draw(11.5,11/4) node[above]{$(\alpha, f(\alpha))$}; 
			\shade[ball color=red](13,31/4)circle(2.5mm);
		\end{tikzpicture}
	\end{center}
	\caption{The graph of $\f$ and bijection \eqref{eq:h-(a,t)} between $(\alpha, t)$ and $h$.
			}
	\label{fig:3.3}
\end{figure}
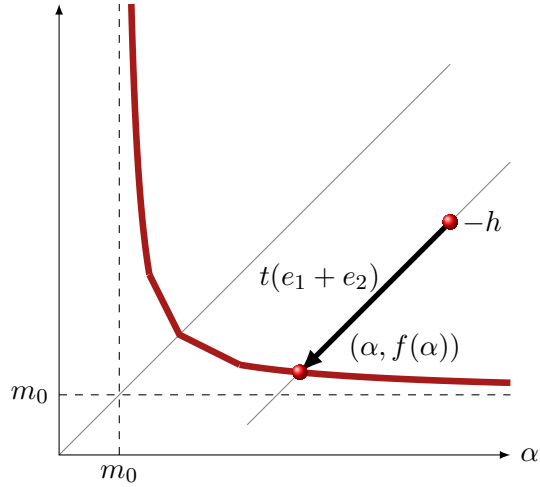

%
%
%
%
%

Once $h$ and $t=t(h)$ are given,  the geometry of the gradients (\eqref{grad g-1}--\eqref{f-9} and   limits \eqref{nabla-g-lim}) can be used to argue the claims about the $\xi$   that satisfy \eqref{eq:h-t}.    This proves part \eqref{th:tilt-velocity:h->xi}.  

That $h$ of the form \eqref{eq:h-t} is dual to $\xi$ follows readily from the fact that gradients are dual and  $\gpl(h+t(e_1+e_2))=\gpl(h)+t$ (this last from Definition \eqref{eq:g:p2l}).   

Note the following general facts for  any $q\in[\nabla\gpp(\zeta+),\nabla\gpp(\zeta-)]$.    By concavity   $\gpp(\eta)\le \gpp(\zeta)+q\cdot(\eta-\zeta)$ for all $\eta$.   Combining this with homogeneity gives $\gpp(\zeta)=q\cdot\zeta$.  Together with duality \eqref{h-xi}  we have 
	\begin{align}\label{g(h)=0}
	 \gpl(-q)=0  \; \text{ for } \; q\in\bigcup_{\zeta\in\ri\Uset}[\nabla\gpp(\zeta+),\nabla\gpp(\zeta-)]. 
	\end{align}  

It remains to show that if  $h$  is dual to $\xi$ then it satisfies \eqref{eq:h-t}.  Let $(\alpha, t)$ be determined by \eqref{eq:h-(a,t)}.  From the last two paragraphs  
\[  \gpl(h)=  \gpl(-\alpha, -\f(\alpha))+t =t. \]
  Let $s=\xi\cdot e_1/\xi\cdot e_2$  so that 
 	\[\gpp(\xi)+ h\cdot\xi=  \frac{\gppa(s)}{1+s}-\frac{\alpha s+\f(\alpha)}{1+s}  +t. \]
 Thus duality $\gpl(h)=\gpp(\xi)+ h\cdot\xi$ implies   $\gppa(s)=\alpha s+\f(\alpha)$ which happens if and only if $\alpha\in[\gppa'(s+),\gppa'(s-)]$.    \eqref{f-9} now implies \eqref{eq:h-t}. 
\end{proof}

\section{Shift-covariant cocycles}\label{sec:cocycles}

In this section we begin with the   shift-covariant cocycles, then show how these define stationary last-passage percolation processes and also solve the variational formulas for $\gpp(\xi)$ and $\gpl(h)$. 

\subsection{Existence and properties of shift-covariant cocycles} 
We come to the technical centerpiece of the paper.  By appeal to queueing fixed points,   in Appendix \ref{app:q} we construct a family of cocycles  $\{B^{\xi\pm}\}_{\xi\in\ri\Uset}$  on an extended space   $\Ombig=\Omega\times\Omega'$.     The next theorem gives the existence statement  and summarizes the properties of these cocycles.   
Assumption  \eqref{2d-ass} is in force.  This is the only place where our proofs actually use the assumption $\P(\w_0\ge c)=1$,  and the only reason   is that the queueing results we reference have been proved only for $\w_0\ge 0$.  

    The cocycles of interest are related to the last-passage weights in the manner described in the next definition.  A potential is simply   a measurable function $V:\Omhat\to\R$.   The case relevant to us will be  $V(\what)=\w_0$ where $\what=(\w, \w')\in\Ombig$ is a configuration in the extended space and contains the original weights $\w$ as a component. 
	\begin{definition}
		\label{def:bdry-model}
			A shift-covariant $L^1$ cocycle $B$ on $\Omhat$  {\rm recovers} 
			potential $V$ if 
				\be \label{eq:VB}
					V(\what)=\min_{i\in\{1,2\}} B(\what, 0,e_i)\quad\text{ for }\Pbig\text{-a.e.\ }\what.
				\ee
	 \end{definition}

 The extended space 
 is  the Polish space $\Ombig=\Omega\times \R^{\{1,2\}\times\cA_0\times\Z^2}$  where $\cA_0$ is a certain countable subset of the interval $(\Ew, \infty)$.  A precise  description of $\cA_0$ appears  in the beginning of the proof of Theorem \ref{th:cocycles} on page \pageref{A_0}  in Appendix \ref{app:q}.  Let   $\kSbig$ denote the Borel $\sigma$-algebra of $\Ombig$.     Generic elements  of $\Ombig$ are denoted by 
$\what=(\w, \w')$ with $\w=(\w_x)_{x\in\Z^2}\in\Omega=\R^{\Z^2}$ as before and $\w'=(\w^{i,\alpha}_x)_{i\in\{1,2\}, \,\alpha\in\cA_0, \, x\in\Z^2}$.   Spatial translations act on the $x$-index in the usual manner:  $(T_x\what)_y=\what_{x+y}$ for $x,y\in\Z^2$ where $\what_x=(\w_x,  \w'_x)=(\w_x,  (\w^{i,\alpha}_x)_{i\in\{1,2\}, \,\alpha\in\cA_0})$.

 
 \begin{theorem}\label{th:cocycles}
There exist functions  $\{B^{\xi+}(\what, x,y),  B^{\xi-}(\what, x,y):x,y\in\Z^2,\xi\in\ri\Uset\}$ on $\Ombig$ and a translation invariant 
  Borel probability measure $\Pbig$ on the space $(\Ombig, \kSbig)$ such that  the following properties hold.  
\begin{enumerate}[\ \ \rm(i)]  
\item\label{th:cocycles:indep} 
For each $\xi\in\ri\Uset$, $x\in\Z^2$, and   $i\in\{1,2\}$,  the  function   $B^{\xi\pm}(\what, x,x+e_i)$ is a function only of  $(\w_x^{i,\alpha}:\alpha\in\cA_0)$.  
Under $\Pbig$, the marginal distribution of the configuration $\w$ is the i.i.d.\ measure $\P$ specified in assumption \eqref{2d-ass}. 
The  $\R^3$-valued  process  $\{\varphi^{\xi+}_x\}_{x\in\Z^2}$ defined by   
\[ \varphi^{\xi+}_x(\what)=(\w_x, B^{\xi+}(x,x+e_1), B^{\xi+}(x,x+e_2))\] is separately ergodic under both translations $T_{e_1}$ and $T_{e_2}$.  The same holds with $\xi+$ replaced by $\xi-$.  For each $z\in\Z^2$,  the variables 
$\{ (\w_x, B^{\xi+}(\what,x,x+e_i), B^{\xi-}(\what,x,x+e_i)) : x\not\le z, \, \xi\in\ri\Uset, \, i\in\{1,2\} \}$ are  independent of $\{ \w_x : x\le z \}$.   \\[-8pt]

 \item\label{th:cocycles:exist}  Each process $B^{\xi+}=\{B^{\xi+}(x,y)\}_{x,y\in\Z^2}$ is a shift-covariant $L^1(\Pbig)$ cocycle {\rm(}Definition \ref{def:cK}{\rm)}  that recovers the  potential $V(\what)=\w_0$ {\rm(}Definition \ref{def:bdry-model}{\rm)}, and the same is true of $B^{\xi-}$.   The associated tilt vectors   $h(\xi\pm)=h(B^{\xi\pm})$ defined by \eqref{eq:EB}  satisfy
					\begin{align}\label{eq:h=grad} h(\xi\pm)=-\nabla\gpp(\xi\pm)\end{align}
				and are dual  to velocity $\xi$ as in \eqref{eq:duality}.   \\[-8pt]
				
	\item\label{th:cocycles:flat} 			 No two distinct cocycles have  a common tilt vector.  That is, if $h(\xi+)=h(\zeta-)$ then $B^{\xi+}(\what,x,x+e_i)=B^{\zeta-}(\what,x,x+e_i)$, 
and similarly $h(\xi+)=h(\zeta+)$ implies  $B^{\xi+}(\what,x,x+e_i)=B^{\zeta+}(\what,x,x+e_i)$.    These equalities  hold   without any  almost sure modifier    because they come for each $\what$  from the construction.   In particular,  if   $\xi$ is a point of differentiability for $\gpp$ then 
	\[  	 B^{\xi+}(\what,x,x+e_i)= B^{\xi-}(\what,x,x+e_i) = B^{\xi}(\what,x,x+e_i)
	 \]
	where the second equality defines the cocycle $B^{\xi}$.     \\[-8pt]
	
\item\label{th:cocycles:cont}     The following inequalities hold  $\Pbig$-almost surely, simultaneously for all    $x\in\Z^2$ and  $\xi,\zeta\in\ri\Uset$: if     $\xi\cdot e_1<\zeta\cdot e_1$ then  
					\be\begin{aligned}  
						B^{\xi-}(x,x+e_1) &\ge B^{\xi+}(x,x+e_1) \ge B^{\zeta-}(x,x+e_1) \\\quad\text{and}\quad B^{\xi-}(x,x+e_2) &\le B^{\xi+}(x,x+e_2) \le B^{\zeta-}(x,x+e_2). 
					\end{aligned} \label{eq:monotone} 
					\ee
					
Fix $\zeta\in\ri\Uset$ and  	let  $\xi_n\to\zeta$ in $\ri\Uset$.  If $\xi_n\cdot e_1\searrow\zeta\cdot e_1$  then for all $x\in\Z^2$ and  $i=1,2$ 
					\begin{align}\label{eq:cont}
						\lim_{n\to\infty}B^{\xi_n\pm}(x,x+e_i) = B^{\zeta+}(x,x+e_i)\qquad \text{$\Pbig$-a.s.\ and in $L^1(\Pbig)$.}   
					\end{align}
Similarly,  if $\xi_n\cdot e_1\nearrow\zeta\cdot e_1$ the limit holds $\Pbig$-a.s.\ and in $L^1(\Pbig)$ with $\zeta-$ on the right.  
\end{enumerate} 
\end{theorem}\smallskip

\begin{remark}\label{rm:cocycles} 
The construction of the cocycles has this property:  there is  a   countable dense set $\Uset_0$ of $\Uset$ such that, for $\xi\in\Uset_0$,   the cocycles are coordinate projections  $B^{\xi\pm}(x,x+e_i)=\w^{i, \gppa(s\pm)}_x$ where $s$ is defined by \eqref{s-xi}.  A point  $\zeta\in\Uset\smallsetminus\Uset_0$ will lie in $\Diff$ and we define $B^\zeta$ through one-sided limits from  $B^{\xi\pm}$, $\xi\in\Uset_0$.   We comment next on various technical properties of the cocycles that are important for the sequel.  

\smallskip

(a) A natural question is whether  $B^{\xi\pm}(\what, x, y)$    can be defined as a function of $\w$ alone, or equivalently, whether it is  $\kS$-measurable.  This is important because we  use the cocycles to solve the variational formulas for the limits and to  construct geodesics, and it would be desirable to work on the original weight space $\Omega$ rather than on the artificially extended space $\Ombig=\Omega\times\Omega'$.    We can make this $\kS$-measurability claim for those cocycles that arise as Busemann functions or their limits. 
(see Remark \ref{bu:kS} below).  

\smallskip

(b)   By part (\ref{th:cocycles:flat}),   if $\gpp$ is linear on the line segment $[\xi',\xi'']\subset\ri\Uset$ with $\xi'\cdot e_1<\xi''\cdot e_1$, then  
\begin{align*}
  B^{\xi'+}(\what, x, x+e_i)=B^{\xi}(\what, x, x+e_i)&=B^{\xi''-}(\what, x, x+e_i) \\
    & \qquad \qquad \forall\,   \what\in\Ombig, \, \xi\in\;]\xi', \xi''[ \,, \ i\in\{1,2\}. 
    \end{align*}
The equalities in part (\ref{th:cocycles:flat}) do not extend to $B^{\xi+}(x,y)$ for all $x, y$ without exceptional $\Pbig$-null sets because the additivity   $ B^{\xi+}(x,z)=B^{\xi+}(x,y)+B^{\xi+}(y,z)$ cannot be  valid  for each $\what$, only almost surely.   

\smallskip

(c)  When we use these cocycles to construct geodesics in Section \ref{sec:geod},  it is convenient to have a  single null set outside of which the ordering \eqref{eq:monotone} is valid for all $\xi, \zeta$.   For the countable family  $\{ B^{\xi\pm}\}_{\xi\in\Uset_0}$ we can arrange for  \eqref{eq:monotone} to hold outside a single $\Pbig$-null event.  By defining  $\{B^\zeta\}_{\zeta\in\Uset\smallsetminus\Uset_0}$ through   limits from the left,  we extend inequalities   \eqref{eq:monotone} to all $\xi, \zeta\in\ri\Uset$ outside a single null set.  This is good enough for a definition of the entire family $\{B^{\xi\pm}\}_{\xi\in\ri\Uset}$.   But in order to claim that limits from left and right agree at a particular $\zeta$, we have to allow for an exceptional $\Pbig$-null event that is specific to $\zeta$.   Thus the limit \eqref{eq:cont} is not claimed outside a single null set for all $\zeta$.  


(d)  When $\P\{\w_0\le r\}$ is a continuous function of $r$ it is natural to ask whether  $B^\xi(x,y)$ can be modified to be continuous in $\xi$.    We do not know the answer.  




\end{remark}

\subsection{Stationary last-passage percolation} \label{sec:stat-lpp}

Fix a  cocycle $B(\what, x,y)=B^{\xi\pm}(\what, x,y)$ from Theorem \ref{th:cocycles}.  Fix a point $v\in\Z^2$ that will serve as an origin.   
By part \eqref{th:cocycles:indep} of Theorem \ref{th:cocycles}, the weights $\{\w_x: x\le v-e_1-e_2\}$ are independent of  $\{ B(v-(k+1)e_i, v-ke_i): k\in\Z_+, \, i\in\{1,2\}\}$.  These define a stationary last-passage percolation process in the third quadrant relative to the  origin at $v$ in the following sense.   Define  passage times $\Gne_{u,v}$ that  use the cocycle as 
edge weights   on the 
north and  east boundaries   and  weights  $\w_x$  in the bulk  $x\le v-e_1-e_2 $: 	\begin{align}
	\label{eq:Gnecorner}
		\begin{split}
		\Gne_{u,v}&=B(u,v)\quad\text{for $u\in\{v-ke_i: k\in\Z_+, \, i\in\{1,2\}\}$}\\
	\qquad 	\text{and} \qquad 		\Gne_{u,v}
			&=\w_u+\Gne_{u+e_1,v}\vee \Gne_{u+e_2,v}\qquad \text{for } \  u\le v-e_1-e_2\,.
		\end{split}
	\end{align}
It is immediate from  recovery 
$\w_x=B(x,x+e_1)\wedge B(x,x+e_2)$     and additivity of $B$ that 
\[\Gne_{u,v}=  B(u,v)  \qquad \text{for all $u\le v$.  } \] 
Process $\{\Gne_{u,v}: u\le v\}$ is stationary in the sense that the increments
\be\label{stat-lpp5} \Gne_{x,v}-\Gne_{x+e_i,v} = B(x,x+e_i)  \ee
are stationary under lattice translations and, as the equation above reveals,  do not depend on the choice of the origin $v$ (as long as we stay southwest of the origin).

 \begin{remark}   In the exactly solvable cases where $\w_x$ is either exponential or geometric,  more is known.  Given the shift-covariant cocycle,  define weights 
 \[Y_x=B(x-e_1,x)\wedge B(x-e_2,x).  \]
 Then the  weights   $\{Y_x\}$ have the same i.i.d.\ distribution as the original weights $\{\w_x\}$.  Furthermore,  $\{Y_x: x\ge v+e_1+e_2\}$ are independent of  $\{ B(v+ke_i, v+(k+1)e_i): k\in\Z_+, \, i\in\{1,2\}\}$.   Hence a  stationary last-passage percolation process can be defined   in the first quadrant with cocycles on the south and west boundaries: 
\begin{align*} 	
		\begin{split}
		\Gsw_{v, x}&=B(v,x)\quad\text{for $x\in\{v+ke_i: k\in\Z_+, \, i\in\{1,2\}\}$}\\
	\qquad 	\text{and} \qquad 		\Gsw_{v,x}
			&=Y_x+\Gsw_{v, x-e_1}\vee \Gsw_{v, x-e_2}\qquad \text{for } \  x\ge v+e_1+e_2\,.
		\end{split}
	\end{align*}

   This feature appears in \cite{Bal-Cat-Sep-06}  as the ``Burke property'' of the exponential last-passage model.  It also works for the log-gamma polymer in positive temperature \cite{Geo-etal-15, Sep-12}.  We do not know presently if this works in the general last-passage case.  It would follow for example if we knew that the distributions of the  cocycles of Theorem \ref{th:cocycles} satisfy this lattice  symmetry:  $\{B(x,y): x,y\in\Z^2\}\overset{d}=\{B(-y,-x): x,y\in\Z^2\}$.  
 \end{remark}
  
\subsection{Solution to the variational formulas}

We solve  variational formulas  \eqref{eq:g:K-var}--\eqref{eq:gpp:K-var}    with the cocycles on the extended space $\OBPbig$.    Once we have identified some cocycles as Busemann functions in Section \ref{sec:busemann},   we prove  Theorem \ref{thm:var-buse}   as a corollary at the end of Section \ref{sec:busemann}.  

	Theorem 3.4 in \cite{Geo-Ras-Sep-16}  says that if a cocycle  $B$ recovers $V(\what)$,  $h(B)$ is defined by  \eqref{eq:EB},  and  centered cocycle $F$ is defined by   \eqref{FF}, then
$F$ is a minimizer in  \eqref{eq:g:K-var} for any $h\in\R^2$ that 
satisfies $(h-h(B))\cdot(e_2-e_1)=0$.
For  such $h$,  the essential supremum over $\what$  in \eqref{eq:g:K-var} disappears and we have 
\be \label{eq:Kvar:min}	\begin{aligned}
	\gpl(h)&\;=\;\max\{V(\what)+h\cdot e_1+F(\what, 0,e_1),V(\what)+h\cdot e_2+F(\what, 0,e_2)\}\\
	&\;=\;(h-h(B))\cdot z  \qquad \text{ for $\Pbig$-a.e.\ $\what$ and any $z\in\{e_1,e_2\}$.
} 
	\end{aligned}\ee

 Recall from Theorem \ref{th:tilt-velocity} that $\gpp$ is linear over each line segment   $[\ximin(h),\ximax(h)]$ and hence, 
by part \eqref{th:cocycles:flat} of Theorem \ref{th:cocycles}, 
cocycles $B^\xi$ (and hence the  tilts $h(\xi)$ they define) coincide for all $\xi\in\,]\ximin(h),\ximax(h)[$. 

\begin{theorem}\label{th:var-sol}
	Let $\{ B^{\xi\pm}\}$ be the cocycles given  in  Theorem \ref{th:cocycles}.  Fix $h\in\R^2$. Let $t(h)$, $\ximin(h)$, and $\ximax(h)$ be as in Theorem \ref{th:tilt-velocity}. One has the following three cases.
		\begin{enumerate}[\ \ \rm(i)]
			\item\label{th:var-sol:flat} $\ximin(h)\ne\ximax(h)$:   For any {\rm(}and hence all{\rm)} $\xi\in\; ]\ximin(h),\ximax(h)[\,$  let 
				\begin{align}
				\label{eq:F-solution}
					F^\xi(x,y)=h(\xi)\cdot(x-y)-B^\xi(x,y). 
				\end{align}
	Then   $F^\xi$ solves \eqref{eq:g:K-var}: 
	for $\Pbig$-almost every $\what$
			\begin{align}
			\label{eq:solution}
				\gpl(h)=\max\{\w_0+h\cdot e_1+F^\xi(\what, 0,e_1),\,\w_0+h\cdot e_2+F^\xi(\what, 0,e_2)\} = t(h).
			\end{align}
		\item\label{th:var-sol:diff} $\ximin(h)=\ximax(h)=\xi\in\Diff$: \eqref{eq:solution} holds for $F^\xi$ defined as in \eqref{eq:F-solution}.
		\item\label{th:var-sol:notdiff} $\ximin(h)=\ximax(h)=\xi\not\in\Diff$: Let $\theta\in[0,1]$ be such that \[ h-t(h)(e_1+e_2)=\theta h(\xi-)+(1-\theta)h(\xi+)\] and define 
			\[
				F^{\xi\pm}(x,y)=h(\xi\pm)\cdot(x-y)-B^{\xi\pm}(x,y)\quad\text{and}\quad F(x,y)=\theta F^{\xi-}(x,y)+(1-\theta)F^{\xi+}(x,y).
			\]
		Then $F$ solves \eqref{eq:g:K-var}: 
	for $\P$-almost every $\w$
			\begin{align}
			\label{eq:solution2}
				\gpl(h)=\Pbig\text{-}\esssup_{\what}\max\{\w_0+h\cdot e_1+F(\what, 0,e_1),\w_0+h\cdot e_2+F(\what, 0,e_2)\} = t(h).
			\end{align}
	If $\theta\in\{0,1\}$, then the essential supremum is not needed in \eqref{eq:solution2}, i.e.\ \eqref{eq:solution} holds almost surely with $F$ in place of $F^\xi$.
	\end{enumerate}
\end{theorem}

Here is the qualitative descriptions of the cases above. 

(i) The graph  of $\f$ has a corner at the point $(\alpha, \f(\alpha))$ where it crosses the $45^o$ diagonal that contains $-h$.  Correspondingly,  $\gppa$ is linear on the interval $[\smin, \smax]$ and $\gpp$ is linear on $[\ximin(h),\ximax(h)]$ with gradient $\nabla\gpp(\xi)=-(\alpha, \f(\alpha))$ at  interior points $\xi\in\; ]\ximin(h),\ximax(h)[$. 

(ii)    $\gpp$ is strictly concave and differentiable at  $\xi$ dual  to tilt $h$.

(iii)   $\gpp$ is strictly concave but not differentiable at  $\xi$ dual  to tilt $h$.

\begin{proof}[Proof of Theorem \ref{th:var-sol}]
	By \eqref{th:cocycles:exist} of Theorem \ref{th:cocycles} the cocycles $B^\xi$ and $B^{\xi\pm}$ that appear in claims (\ref{th:var-sol:flat})--(\ref{th:var-sol:notdiff}) recover the potential as required by Definition \ref{def:bdry-model} and hence conclusion \eqref{eq:Kvar:min} is in force. 
	
	In cases \eqref{th:var-sol:flat} and \eqref{th:var-sol:diff} \eqref{eq:Kvar:min} implies
		\[
			\gpl(h)=\max\{\w_0+h\cdot e_1+F^\xi(0,e_1),\w_0+h\cdot e_2+F^\xi(0,e_2)\} = (h-h(\xi))\cdot e_1.
		\]
	  The last term equals $t(h)$, by combining  \eqref{eq:h-t}  and \eqref{eq:h=grad}. The same proof works for \eqref{th:var-sol:notdiff} when $\theta\in\{0,1\}$.
	
	For the last case \eqref{th:var-sol:notdiff},  
	\eqref{eq:h=grad} and \eqref{g(h)=0} imply $\gpl(h(\xi\pm))=0$.
	Then Theorem \ref{th:tilt-velocity}  implies  
		\begin{align*}
			\gpl(h)&=\gpp(\xi)+h\cdot\xi=t(h)+\theta(\gpp(\xi)+h(\xi-)\cdot\xi)+(1-\theta)(\gpp(\xi)+h(\xi+)\cdot\xi)\\
			          &=t(h)+\theta\gpl(h(\xi-))+(1-\theta)\gpl(h(\xi+))\\
			          &=t(h).
		\end{align*}
	Furthermore, we have for $\P$-almost every $\w$
		\[ \min\{\theta B^{\xi-}(0,e_1)+(1-\theta)B^{\xi+}(0,e_1),\theta B^{\xi-}(0,e_2)+(1-\theta)B^{\xi+}(0,e_2)\}\ge\w_0.\]
	This translates into 
		\[ \Pbig\text{-}\esssup_{\what} \max\{\w_0+h\cdot e_1+F(0,e_1),\w_0+h\cdot e_2+F(0,e_2)\} \le t(h) = \gpl(h).  \]
	Formula \eqref{eq:g:K-var} implies then that the above inequality is in fact an equality and \eqref{eq:solution2} is proved. 
\end{proof}

We state also the corresponding theorem for the point-to-point case, though there is nothing to prove.  

\begin{theorem}\label{thm:var-p2p}
	Let $\xi\in\ri\Uset$. Cocycles $B^{\xi\pm}$ solve \eqref{eq:gpp:K-var}:
			\begin{align}
			\label{eq:solution-pp}
			\begin{split}
				\gpp(\xi)&=\max\{\w_0-B^{\xi\pm}(\what, 0,e_1)-h(\xi\pm)\cdot\xi\,,\,\w_0-B^{\xi\pm}(\what, 0,e_2)-h(\xi\pm)\cdot\xi\} \\
				&= -h(\xi\pm)\cdot\xi \qquad \text{for $\Pbig$-a.e.~$\what$.} 
			\end{split}
			\end{align}
	If $\xi\not\in\Diff$ and $\theta\in(0,1)$, then cocycle $B=\theta B^{\xi-}+(1-\theta) B^{\xi+}$ also solves \eqref{eq:gpp:K-var}:
			\begin{align*}
				\gpp(\xi)&=\Pbig\text{-}\esssup_{\what}\max\{\w_0-B(\what, 0,e_1)-h(B)\cdot\xi\,,\,\w_0-B(\what, 0,e_2)-h(B)\cdot\xi\}\\
				& = -h(B)\cdot\xi.
			\end{align*}
\end{theorem}

The above theorem follows directly from \eqref{eq:duality}, Theorem \ref{th:cocycles}\eqref{th:cocycles:exist}, \eqref{eq:solution}, and 
the fact that cocycles $B^{\xi\pm}$ recover $V(\what)=\w_0$. The last claim follows similarly to \eqref{eq:solution2}.
 
\section{Busemann functions}
\label{sec:busemann}

In this section we prove the  existence of   Busemann functions.   As before    \eqref{2d-ass} is a standing assumption.
Recall the line segment $\Uset_\xi=[\ximin, \ximax]$ with $\ximin\cdot e_1\le \ximax\cdot e_1$ from \eqref{eq:sector1}--\eqref{eq:sector2}
and the cocycles $B^{\xi\pm}$ constructed on the extended space $\OBPbig$ in Theorem \ref{th:cocycles}.   

\begin{theorem}\label{th:construction} 
Fix $\xi\in\ri\Uset$.  
Then there exists an event  $\Ombig_0$ with $\Pbig(\Ombig_0)=1$ such that for each 
$\what\in\Omega_0$ and  for any  sequence $v_n\in\Z_+^2$ that satisfies 
\begin{align}\label{eq:vn66}
\abs{v_n}_1\to\infty\quad\text{and}\quad\ximin\cdot e_1\le\varliminf_{n\to\infty}\frac{v_n\cdot e_1}{\abs{v_n}_1}\le\varlimsup_{n\to\infty}\frac{v_n\cdot e_1}{\abs{v_n}_1}\le\ximax\cdot e_1,
\end{align}
  we have
\be\label{bu:G1}  \begin{aligned}   B^{\ximax+}(\what, x,x+e_1)&\le \varliminf_{n\to \infty} \big( \Gpp_{x, v_n}(\w) - \Gpp_{x+e_1, v_n}(\w) \big) \\ &\le \varlimsup_{n\to \infty} \big( \Gpp_{x, v_n}(\w) - \Gpp_{x+e_1, v_n}(\w)  \big) \le B^{\ximin-}(\what, x,x+e_1)  
\end{aligned}  \ee
and 
\be\label{bu:G2}  
\begin{aligned}  B^{\ximin-}(\what, x,x+e_2)  &\le \varliminf_{n\to \infty} \big( \Gpp_{x, v_n}(\w) - \Gpp_{x+e_2, v_n}(\w) \big)\\
&\le \varlimsup_{n\to \infty} \big( \Gpp_{x, v_n}(\w) - \Gpp_{x+e_2, v_n}(\w) \big) \le B^{\ximax+}(\what, x,x+e_2).  
\end{aligned}  \ee
\end{theorem}


The interesting cases are of course the ones where we have a limit.   For the corollary note that if  
$\xi, \ximin, \ximax\in\Diff$  then by Theorem \ref{th:cocycles}\eqref{th:cocycles:flat}  $B^{\ximin\pm}=B^{\xi}=B^{\ximax\pm}$.    

\begin{corollary}\label{cor:buse} 
 Assume that $\xi$, $\ximin$ and $\ximax$ are points of differentiability of $\gpp$.
Then there exists an event  $\Ombig_0$ with $\Pbig(\Ombig_0)=1$ such that for each 
$\what\in\Omega_0$,   for any  sequence $v_n\in\Z_+^2$ that satisfies \eqref{eq:vn66}, 
and  for all $x,y\in\Z^2$, 
\be \label{eq:grad:coc} 
							B^{\xi}(\what,x,y) = \lim_{n\to \infty} \big( \Gpp_{x, v_n}(\w) - \Gpp_{y, v_n}(\w)  \big). 
						\ee 
						
In particular, if $\gpp$ is differentiable everywhere on $\ri\Uset$,  then  for each direction $\xi\in\ri\Uset$ there is an event of full $\Pbig$-probability on which   limit \eqref{eq:grad:coc}  holds for any sequence $v_n/\abs{v_n}_1\to\xi$.  
 \end{corollary}

Before turning to proofs, let us derive the relevant results of Section \ref{sec:results} and address the question of measurability of cocycles raised in Remark \ref{rm:cocycles}(a). 
 
\begin{proof}[Proof of Theorem \ref{thm:buse}]  Immediate consequence of Corollary \ref{cor:buse}.  
Equation \eqref{EB=Dg} follows from \eqref{eq:h=grad}.
\end{proof}

\begin{proof}[Proof of Theorem \ref{thm:var-buse}] 
The theorem follows from  
Theorems \ref{th:var-sol} and  \ref{thm:var-p2p}
because the Busemann function $B^\xi$ is the  cocycle $B^\xi$ from Theorem \ref{th:cocycles}. 
\end{proof}


\begin{remark}[$\kS$-measurability of cocycles]\label{bu:kS} 
A consequence of limit \eqref{eq:grad:coc}  is that the cocycle  $B^{\xi}(\what,x,y)$ is actually a function of $\w$  alone, in other words,    $\kS$-measurable.   
Furthermore,  all cocycles $B^{\zeta\pm}$  that can be obtained as limits,  as these $\xi$-points converge to $\zeta\pm$ on $\ri\Uset$, are also  $\kS$-measurable.     In particular,  if  $\gpp$ is differentiable at the endpoints of its linear segments (if any),  all the cocycles $\{B^{\zeta\pm}:\zeta\in\ri\Uset\}$ described in Theorem \ref{th:cocycles}  are $\kS$-measurable.  
At points $\zeta\notin\Diff$ of strict concavity this follows  because $\zeta$ can be approached from both sides by points  $\xi\in\EP$ which satisfy \eqref{eq:grad:coc}.  
\end{remark}
 
 The remainder of this section proves Theorem \ref{th:construction}.   
 We begin with   a general comparison lemma.   
With  arbitrary real weights $\{\wt Y_x\}_{x\in\Z^2}$ define  last passage times   
 	\[\wt \Gpp_{u,v}=\max_{x_{0,n}}\sum_{k=0}^{n-1}\wt Y_{x_k}.\]
The maximum is over up-right paths from $x_0=u$ to $x_n=v$ with $n=|v-u|_1$.  The convention is $\wt \Gpp_{v,v}=0$.  
For  $x\le v-e_1$ and $y\le v-e_2$  denote the increments  by 
	\[	\wt I_{x,v} = \wt \Gpp_{x,v} - \wt \Gpp_{x+e_1,v} \qquad  \text{ and } \qquad  \wt J_{y,v} = \wt \Gpp_{y,v} - \wt \Gpp_{y+e_2,v}\,. \]
	
\begin{lemma}\label{lm:new:comp0}
	For  $x\le v-e_1$ and $y\le v-e_2$ 
		\be
		\label{eq:new:comp0}
			\wt I_{x,v+e_2} \ge \wt I_{x,v} \ge \wt I_{x,v+e_1} 
			\qquad  \text{ and } \qquad 
			\wt J_{y, v+e_2} \le \wt J_{y,v} \le \wt J_{y, v+e_1}\,.
		\ee
\end{lemma}

\begin{proof}
	Let $v=(m,n)$. 
	The proof goes   by an induction argument. 
	Suppose $x = (k,n)$ for some $k < m$. Then on the north boundary
		\begin{align*}
			\wt I_{(k,n), (m, n+1)} 
			&= \wt \Gpp_{(k,n), (m, n+1)} - \wt \Gpp_{(k +1,n), (m, n+1)} \\
			&= \wt Y_{k,n} + \wt \Gpp_{(k+1,n), (m, n+1)} \vee \wt \Gpp_{(k,n+1), (m, n+1)} - \wt \Gpp_{(k +1,n), (m, n+1)}\\
			&\ge \wt Y_{k,n}= \wt \Gpp_{(k,n), (m, n)} - \wt \Gpp_{(k+1,n), (m, n)} = \wt I_{(k,n), (m,n)}\,.
		\end{align*}
	On the east boundary, when $y= (m,\ell)$ for some $\ell < n$
		\begin{align*}
			\wt J_{(m,\ell),(m,n+1)}&=\wt \Gpp_{(m,\ell),(m,n+1)} - \wt \Gpp_{(m,\ell+1),(m,n+1)}\\
			&=\wt Y_{m,\ell}  = \wt \Gpp_{(m,\ell),(m,n)}-\wt \Gpp_{(m,\ell+1),(m,n)}  = \wt J_{(m,\ell),(m,n)}\,. 
		\end{align*}
	These inequalities start the induction. Now let   $u\le v- e_1- e_2$.
	Assume by induction  that \eqref{eq:new:comp0} holds for $x=u+e_2$ and $y=u+e_1$.  
		\begin{align*}
			\wt I_{u,v+e_2}
			&=\wt \Gpp_{u,v+e_2}-\wt \Gpp_{u+e_1,v+e_2}
			= \wt Y_u + (\wt \Gpp_{u+e_2,v+e_2}-\wt \Gpp_{u+e_1,v+e_2})^+ \\
			&=\wt Y_u+(\wt I_{u+e_2,v+e_2} - \wt J_{u+e_1,v+e_2})^+  \\
			&\ge  \wt Y_u+(\wt I_{u+e_2,v} - \wt J_{u+e_1,v})^+ = \wt I_{u,v}\,.  
		\end{align*} 
		For the last equality simply reverse the first three equalities with $v$ instead of $v+e_2$. 
	A similar argument works for $\wt I_{u,v}\ge \wt I_{u,v+e_1}$ and a symmetric argument works for the $\wt J$ inequalities. 	
\end{proof}

The estimates needed for the proof of Theorem \ref{th:construction} come from coupling $\Gpp_{u,v}$ with the stationary LPP described in Section \ref{sec:stat-lpp}.  For the next two lemmas fix a  
  cocycle $B(\what, x,y)=B^{\zeta\pm}(\what, x,y)$ from Theorem \ref{th:cocycles} and   let $r=\zeta\cdot e_1/\zeta\cdot e_2$ so that $\alpha=\gppa'(r\pm)$ satisfies
\be\label{bu:a}  \alpha=\Ebig[B(x,x+e_1)]  \quad\text{and}\quad  \f(\alpha)=\Ebig[B(x,x+e_2)]. \ee
 As in \eqref{eq:Gnecorner} define 
 	\begin{align}
	\label{Gnecorner5}
		\begin{split}
		\Gne_{u,v}&=B(u,v)\qquad\text{for $u\in\{v-ke_i: k\in\Z_+, \, i\in\{1,2\}\}$}\\[3pt]
	\qquad 	\text{and} \qquad 		\Gne_{u,v}
			&=\w_u+\Gne_{u+e_1,v}\vee \Gne_{u+e_2,v}\qquad \text{for } \  u\le v-e_1-e_2\,.
		\end{split}
	\end{align}

Let  $\Gne_{u,v}(A)$ denote a maximum over paths restricted to the set $A$.  In particular, below we use 
  \[  \Gne_{0,v_n}(v_n-e_i\in x_\centerdot) = \max_{x_\centerdot\,: \, x_{\abs{v_n}_1-1}=v_n-e_i} \sum_{k=0}^{\abs{v_n}_1-1}  \wt Y_{x_k}
  \]
  where the maximum is restricted to  paths that go through the point $v_n-e_i$, and the 
 weights are from \eqref{Gnecorner5}:  $\wt Y_x=\w_x$ for $x\le v-e_1-e_2$ while 
 $\wt Y_{v-ke_i}=B(v-ke_i, v-(k-1)e_i)$.


Figure \ref{fig:7.5}  makes the limits of the next lemma obvious.   But  a.s.\ convergence requires some technicalities  because the north-east  boundaries themselves are translated as the limit is taken.  

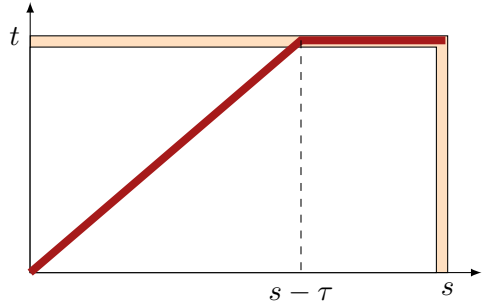
\begin{figure}[h]
	\begin{center}
		\begin{tikzpicture}[>=latex, scale=0.6]

\draw[<->] (0,6)--(0,0)--(10,0); 

			\fill[color=orange, nearly transparent](0,0)rectangle(9.25,5.25); 
			\draw(0,0)rectangle(9.25,5.25); 
			\fill[color=white](0,0)rectangle(9,5);
			\draw(0,0)rectangle(9,5);
			
			\draw(9.25,0)node[below]{$s$};
			\draw(0,5.25)node[left]{$t$};
			\draw[color=nicosred, line width=3pt](0,0)--(6,5.15)--(9.2, 5.15); 
			
			\draw[dashed] (6,5.15)--(6,0)node[below]{$s-\tau$};

%
%
%
%
%
%
%
%
%
%
%
%
%

		\end{tikzpicture}
	\end{center}
	\caption{\small Illustration of \eqref{eq:hor:lln}.  Forcing the last step to be $e_1$ restricts the maximization   to paths that hit the north boundary instead of the east boundary.  The path from $0$ to $(s-\tau,t)$ contributes $\gpp(s-\tau,t)$ and the remaining segment of length $\tau$ on the north boundary  contributes $\alpha \tau$. 
			}
	\label{fig:7.5}
\end{figure}

 \begin{lemma}
	Assume \eqref{2d-ass}. Fix $(s,t) \in \R^2_+$. Let $v_n\in\Z_+^2$ be such that    $v_n/\abs{v_n}_1\to(s,t)/(s+t)$ as $n\to\infty$ and  $\abs{v_n}_1\ge\eta_0n$ for some constant $\eta_0>0$. Then we have the following almost sure  limits:
		\be\label{eq:hor:lln}
			\abs{v_n}_1^{-1}\,\Gne_{0,v_n}(v_n-e_1\in x_\centerdot) \; { \longrightarrow }\;
			(s+t)^{-1}\sup_{0 \le \tau \le s}\{ \alpha\tau + \gpp(s-\tau,t) \}  
		\ee

and 
 		\be\label{eq:vert:lln}
			\abs{v_n}_1^{-1}\,\Gne_{0,v_n}(v_n-e_2\in x_\centerdot) \; {\longrightarrow} \;
			(s+t)^{-1}\sup_{0 \le \tau \le t}\{ f(\alpha)\tau + \gpp(s, t-\tau) \}. 
		\ee
 \end{lemma}

\begin{proof}
	We prove \eqref{eq:hor:lln}. Fix $\e>0$, 
	  let  $M=\fl{\e^{-1}}$,  and 
		\[
			q^n_j = j\Bigl\lfloor\frac{\e\abs{v_n}_1 s}{s+t}\Bigr\rfloor \, \text{ for }0\le j \le M-1, 
				\text{ and } q^n_{M}= v_n\cdot e_1.
		\]
		For large enough $n$ it is the case that  $q^n_{M-1}<v_n\cdot e_1$.  

Suppose    a maximal path for $\Gne_{0,v_n}(v_n-e_1\in x_\centerdot)$ enters 
	 the north boundary from  the bulk at the point 
  $v_n-(\ell,0)$  with $q^n_{j} < \ell\le q^n_{j+1}$.  Then 
		\begin{align*}
			&\Gne_{0,v_n}(v_n-e_1\in x_\centerdot) 
= \Gpp_{0, v_n-(\ell,1)} + \w_{v_n-(\ell,1)} + 
B( v_n-(\ell,0), v_n)      
\\
			&\quad 
	\le \Gpp_{0, v_n-(q^n_j,1)}   +  q^n_j\alpha - \sum_{k=q^n_j+1}^{\ell-1} \bigl(\w_{v_n-(k,1)} -\Ew\bigr)  +(\ell-1-q^n_j)  \Ew \\
&\qquad\qquad  + \; \;  
\bigl( B( v_n-(\ell,0), v_n) -\ell\alpha\bigr)      
+ (\ell-q^n_j)\alpha . 	
		\end{align*}
 The two main  terms come  right after the inequality above  and the rest are errors.  	The   inequality 	comes from 
 \[ \Gpp_{0, v_n-(\ell,1)} +   \sum_{k=q^n_j+1}^{\ell} \w_{v_n-(k,1)} \le  \Gpp_{0, v_n-(q^n_j,1)}   \]
 and  algebraic rearrangement.  
		
Define the centered cocycle 
$   F(x,y)= h(B)\cdot(x-y)-B(x,y) $   so that 
\[   
B( v_n-(\ell,0), v_n) -\ell\alpha= F(0, v_n-(\ell,0))  -F(0, v_n).  
\] 
The potential-recovery property \eqref{eq:VB}   $\w_0=B(0,e_1)\wedge B(0,e_2)$ gives 
\[  F(0, e_i) \le \alpha\vee\f(\alpha) -\w_0\qquad\text{for $i\in\{1,2\}$.}   \]  		
The i.i.d.\ distribution of $\{\w_x\}$ and  $\E(\abs{\w_0}^{\pp})<\infty$ with $\pp>2$  are  strong enough to guarantee that Lemma  \ref{th:Atilla}   from Appendix \ref{app:aux} applies and gives 
\be\label{b:F-erg}      \lim_{N\to\infty} \; \frac1N \; \max_{x\ge 0\,:\,\abs{x}_1\le N} \abs{F(\what,0,x)} =0\qquad
\text{for a.e.\ $\what$.}    	
\ee	
		
 Collect the bounds for all the intervals $(q^n_j, q^n_{j+1}]$ and let $C$ denote a constant.  Abbreviate $S^n_{j,m}=\sum_{k=q^n_j+1}^{q^n_j+m} \bigl(\w_{v_n-(k,1)} -\Ew\bigr)$. 
	\be	\begin{aligned}
			&\Gne_{0,v_n}(v_n-e_1\in x_\centerdot)  
				\;\le\;\max_{0\le j \le M-1 }\!\!\Big\{  \Gpp_{0, v_n-(q^n_j,1)}   +  q^n_j\alpha + C(q^n_{j+1}-q^n_j)  \\
	&\quad +   \max_{0\le m< q^n_{j+1}-q^n_j} \abs{S^n_{j,m}} 
+  \max_{q^n_{j} < \ell\le q^n_{j+1}} F(0, v_n-(\ell,0))  -F(0, v_n) 
\Big\}.
		\end{aligned}\label{G:bnd} \ee 
Divide through by $\abs{v_n}_1$ and let $n\to\infty$.   Limit   \eqref{lln5}  gives convergence of the $G$-term on the right.  
  We claim that the terms on the second line of \eqref{G:bnd} vanish.  Limit  \eqref{b:F-erg}    takes care of the $F$-terms. 		
 Combine Doob's maximal inequality for martingales with 	Burkholder's inequality 
 \cite[Thm.~3.2]{Bur-73}	to obtain, for $\delta>0$,  
\begin{align*}
&\P\bigl\{  \,\max_{0\le m< q^n_{j+1}-q^n_j} \abs{S^n_{j,m}} \ge \delta\abs{v_n}_1 \bigr\} \le \frac{ \E\bigl[ \abs{S^n_{j, q^n_{j+1}-q^n_j}}^{\pp}\bigr] }{\delta^{\pp}\abs{v_n}_1^{\pp} } \\
&\qquad\qquad\qquad
 \le  \frac{C}{\delta^{\pp}\abs{v_n}_1^{\pp} } \, { \E\biggl[ \,\biggl\lvert \,\sum_{i=1}^{q^n_{j+1}-q^n_j} \bigl(\w_{i,0} -\Ew\bigr)^2\biggr\rvert^{\pp/2} \,  \biggr] }  
\le \frac{C}{\abs{v_n}_1^{\pp/2} } . 
\end{align*} 	
Thus Borel-Cantelli takes 	care of the $S^n_{j,m}$-term on the second line of \eqref{G:bnd}.  (This is the place where the assumption $\abs{v_n}_1\ge \eta_0n$ is used.)   We have the upper bound
\begin{align*}
\varlimsup_{n\to\infty}   \abs{v_n}_1^{-1} \Gne_{0,v_n}(v_n-e_1\in x_\centerdot)  
    \le 
 (s+t)^{-1} \max_{0\le j \le M-1 }\big[\gpp(s-sj\e, t) +sj\e \alpha + C\e s \big].
 \end{align*} 
 Let $\e\searrow 0$ to complete the proof of the upper bound.

  To get the  matching lower bound let  the supremum $\sup_{\tau\in [0,s]} \{ \tau\alpha +  \gpp(s-\tau , t)\}$ be attained at $\tau^* \in [0,s]$.  
	With $m_n=\abs{v_n}_1/(s+t)$ we have 
 		\begin{align*}
	\Gne_{0,v_n}(v_n-e_1\in x_\centerdot) &\ge 	\Gpp_{0,v_n-(\fl{m_n\tau^*}\vee1,1)}\; +\; \w_{v_n-(\fl{m_n\tau^*}\vee1,1)} \\[2pt]  & \qquad  
	+  \; B(v_n- (\fl{m_n\tau^*}\vee1,0), \, v_n). 
		\end{align*}
Use again the cocycle $F$ from above,  and let $n\to\infty$ to get 
		\begin{align*}
\varliminf_{n\to\infty}  		\abs{v_n}_1^{-1}\Gne_{0,v_n}(v_n-e_1\in x_\centerdot)
		\ge  (s+t)^{-1}[\gpp(s-\tau^*,t) + \tau^* \alpha] . 
		\end{align*}
This completes  the proof of  \eqref{eq:hor:lln}.
\end{proof}

 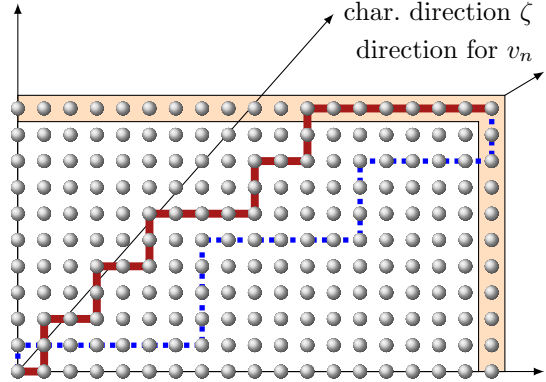
\begin{figure}[h]
	\begin{center}
		\begin{tikzpicture}[ >=latex, scale=0.7]

			\draw[<->] (0,7)--(0,0)--(10,0); 

			\fill[color=orange, nearly transparent](0,0)rectangle(9.25,5.25); 
			\draw(0,0)rectangle(9.25,5.25); 
			\fill[color=white](0,0)rectangle(8.75,4.75);
			\draw(0,0)rectangle(8.75,4.75);

			\draw[->] (0,0)--(6,6.8) node[right]{\small char.~direction $\zeta$};

			\draw[line width=3pt, color=nicosred](9,5)--(5.5,5)--(5.5,4)--(4.5,4)--(4.5,3)--(2.5, 3)--(2.5, 2)--(1.5,2)--(1.5,1)--(.5,1)--(.5,0)--(0,0); 
			
			\draw[line width=2pt, dotted, color=blue](9,5)--(9,4)--(6.5,4)--(6.5,2.5)--(3.5, 2.5)--(3.5,1.5)--(3.5,0.5)--(.5,.5)--(0,.5)--(0,0); 
			
			 \foreach \x in {0,...,18}{
              			 \foreach \y in {0,...,10}{
					\shade[ball color=really-light-gray](\x*0.5,\y*0.5)circle(1.3mm);  
								}
							} 
			\draw[->] (9.25, 5.25)--(10,5.7)node[above left]{\small direction for $v_n$ }; 
			


			
			

					\end{tikzpicture}
	\end{center}
	\caption{\small  Illustration of Lemma \ref{lm:exit:con}.  With $\alpha$-boundaries geodesics tend to go in the $\alpha$-characteristic direction $\zeta$.  If   $v_n$ converges in a direction  below $\zeta$,  maximal paths to $v_n$ tend  to hit  the north boundary. The dotted path that hits  the east boundary is unlikely to  be maximal
	for large $n$. }
	\label{fig:7.6}
\end{figure}

Continue with the stationary LPP defined by \eqref{Gnecorner5} in terms of a cocycle $B=B^{\zeta\pm}$, with $r=\zeta\cdot e_1/\zeta\cdot e_2$ and $\alpha$ as in \eqref{bu:a}.    Let us call the direction $\zeta$ {\sl characteristic} for $\alpha$.  
  The next lemma shows 	that in stationary LPP a maximizing path  to a point   below the characteristic direction    will eventually hit    the north boundary before the east boundary.  (Illustration in Figure \ref{fig:7.6}.)    We leave to the reader the analogous result to a point above the characteristic line.   

\begin{lemma} \label{lm:exit:con}    Let $s\in(r, \infty)$.  Let   $v_n\in\Z_+^2$  be such that   $v_n/\abs{v_n}_1\to(s,1)/(1+s)$ and $\abs{v_n}_1\ge\eta_0n$ for some constant $\eta_0>0$.
Assume that $\gppa'(r+)>\gppa'(s-)$.  
	Then
 	$\Pbig$-a.s.\ there exists a random $n_0 < \infty$ such that for all $n \ge n_0$, 
		\be\label{eq:move:e1}
		\Gne_{0,v_n} = 
		\Gne_{0, v_n}(v_n-e_1\in x_\centerdot). 
		\ee
\end{lemma}

\begin{proof} 
	The right derivative at $\tau=0$ of 
$\alpha \tau + \gpp(s-\tau, 1) = \alpha \tau +  \gppa(s- \tau)$	
	equals 
		\[
			\alpha - \gppa'(s-) > \alpha - \gppa'(r+) \ge 0.
		\]
	The last inequality above follows from the assumption on $r$.  
	Thus we can find  $\tau^*\in(0,r)$ such that 
		\be\label{eq:hor:ineq}
			\alpha \tau^*+\gpp(s-\tau^*, 1) > \gpp(s,1).
		\ee
		
	  To produce a contradiction let $A$ be the event  on which $\Gne_{0,v_n} = \Gne_{0, v_n}(v_n-e_2\in x_\centerdot)$ for infinitely many $n$ and assume $\Pbig(A)>0$.    Let $m_n=\abs{v_n}_1/(1+s)$.   
	On $A$ we have for infinitely many $n$ 
		\begin{align*} 
			&\abs{v_n}^{-1}	\Gne_{0,v_n}(v_n-e_2\in x_\centerdot) \; = \; \abs{v_n}^{-1} \Gne_{0, v_n}\\[3pt]
			&\qquad \ge \abs{v_n}^{-1} B(v_n-(\fl{m_n\tau^*}+1)e_1, v_n)
			+ \abs{v_n}^{-1}\Gpp_{0,v_n-(\fl{m_n\tau^*}+1,1)}\\
			&\qquad\qquad +\abs{v_n}^{-1}\w_{v_n -(\fl{m_n\tau^*}+1,1)}\notag.
		\end{align*}
		Apply \eqref{eq:vert:lln} to the leftmost quantity.  Apply limits \eqref{lln5} and \eqref{b:F-erg} and stationarity and integrability of $\w_x$  to the expression on the right. 
	  Both extremes  of the above inequality converge almost surely.   
Hence on the event $A$ the inequality is preserved to the limit and yields (after multiplication by $1+s$) 
		\begin{align*}
		\begin{split}
		\sup_{0\le \tau \le 1} \{ \f(\alpha) \tau + \gpp(s,1-\tau)\}  
		\ge \alpha \tau^* + \gpp(s-\tau^*,1).
		\end{split}
		\end{align*}
The   supremum of the left-hand side is achieved at $\tau = 0$ because the right derivative  equals 
\[  \f(\alpha) -\gppa'(\tfrac{1-\tau}s-) \le \f(\alpha) -\gppa'(r^{-1}-)  \le 0
\]
where the  first   inequality comes from  $s^{-1}<r^{-1}$ and the second from   \eqref{eq:f=ga'}. Therefore
		\[  \gpp(s,1) \ge \alpha \tau^* + \gpp(s-\tau^*, 1) \]
	which contradicts \eqref{eq:hor:ineq}. Consequently $\Pbig(A)=0$ and \eqref{eq:move:e1} holds for $n$ large.
\end{proof}

\begin{proof}[Proof of Theorem \ref{th:construction}.] The proof  goes in two steps.\medskip

{\bf Step 1.} First consider a fixed $\xi=(\frac{s}{1+s}, \frac{1}{1+s})\in\ri\Uset$ and a sequence $v_n$ such 
  that $v_n/\abs{v_n}_1\to \xi$ and $\abs{v_n}_1\ge \eta_0 n$ for some $\eta_0>0$. 
  We prove that  the last inequality of \eqref{bu:G1} holds almost surely.  
	Let $\zeta=(\frac{r}{1+r}, \frac{1}{1+r})$ satisfy  $\zeta\cdot e_1<\ximin\cdot e_1$ so that  $\gppa'(r+)>\gppa'(s-)$ and Lemma \ref{lm:exit:con} can be applied. 
	Use cocycle  $B^{\zeta+}$ from Theorem \ref{th:cocycles} to define   last-passage times    $\Gne_{u, v}$ as in \eqref{Gnecorner5}.   Furthermore,  define last-passage times    $\Gn_{u, v}$ that use cocycles only on the north boundary and bulk weights elsewhere: 
		\begin{align*}
		\begin{split}
			\Gn_{v-ke_1,v}=B^{\zeta+}(v-ke_1,v) \,,
			&\qquad \Gn_{v-\ell e_2,v}=\sum_{j=1}^{\ell} \w_{v-je_2}\,,\\
			\quad\text{and}\qquad 	\Gn_{u,v}
			&=\w_u+\Gn_{u+e_1,v}\vee\Gn_{u+e_2,v}\quad \text{ for } \ u\le v-e_1-e_2\,.
		\end{split}
	\end{align*}
For large  $n$  we have 
		\begin{align*}
			\Gpp_{x,v_n} -  \Gpp_{x+e_1,v_n} &\le   \Gn_{x,v_n+e_2}-  \Gn_{x+e_1,v_n+e_2} \\
			 &  =  \Gne_{x,v_n+e_1+e_2}(v_n+ e_2\in x_\centerdot)-  \Gne_{x+e_1,v_n+e_1+e_2}(v_n+ e_2\in x_\centerdot)\\
		&  =  \Gne_{x,v_n+e_1+e_2} -  \Gne_{x+e_1,v_n+e_1+e_2}
		= B^{\zeta+}(x, x+e_1).   		\end{align*}
The first inequality above is the first inequality of \eqref{eq:new:comp0}.  The first equality above  is obvious.  The second  equality is Lemma \ref{lm:exit:con} and the last     equality is \eqref{stat-lpp5}. 	Thus 
		\[\varlimsup_{n\to \infty} \big(\Gpp_{x,v_n} -  \Gpp_{x+e_1,v_n}\big) \le B^{\zeta+}(x,x+e_1).\]

	Let $\zeta\cdot e_1$  increase to $\ximin\cdot e_1$.  Theorem \ref{th:cocycles}\eqref{th:cocycles:cont} implies
		\[\varlimsup_{n\to \infty}\big( \Gpp_{x,v_n} -  \Gpp_{x+e_1,v_n}\big) \le B^{\ximin-}(x,x+e_1).\]
		
	An analogous  argument gives the matching lower bound (first inequality of \eqref{bu:G1}) by taking $\zeta\cdot e_1>\ximax\cdot e_1$ and by reworking Lemma \ref{lm:exit:con} for the case where the direction of $v_n$ is above the characteristic direction $\zeta$.  
	Similar  reasoning  works for vertical increments  $\Gpp_{x,v_n} -  \Gpp_{x+e_2,v_n}$.\medskip

{\bf Step 2.} We prove the full statement of the theorem. 
Let $\eta_\ell$ and $\zeta_\ell$ be two sequences in $\ri\Uset$ such that
$\eta_\ell\cdot e_1<\ximin\cdot e_1$, $\ximax\cdot e_1<\zeta_\ell\cdot e_1$, $\eta_\ell\to\ximin$, and $\zeta_\ell\to\ximax$. Let $\Ombig_0$ be the event on 
  which  limits \eqref{eq:cont} holds for directions $\ximax$ and $\ximin$ (with sequences $\zeta_\ell$ and $\eta_\ell$, respectively) and
\eqref{bu:G1} holds  for each  direction $\zeta_\ell$ with sequence $\fl{n\zeta_\ell}$, and for each    direction $\eta_\ell$ with  sequence $\fl{n\eta_\ell}$.  
$\Pbig(\Ombig_0)=1$ by Theorem \ref{th:cocycles}\eqref{th:cocycles:cont} and Step 1.

Fix $\ell$ and a sequence   $v_n$ as in \eqref{eq:vn66}.  Abbreviate $a_n=\abs{v_n}_1$.  For large  $n$  
\[\fl{a_n\eta_\ell}\cdot e_1<v_n\cdot e_1<\fl{a_n\zeta_\ell}\cdot e_1
\quad\text{and}\quad 
\fl{a_n\eta_\ell}\cdot e_2>v_n\cdot e_2>\fl{a_n\zeta_\ell}\cdot e_2.
\]
By repeated application of Lemma \ref{lm:new:comp0}  
	\[\Gpp_{x,\fl{a_n\zeta_\ell}}-\Gpp_{x+e_1,\fl{a_n\zeta_\ell}}\le\Gpp_{x,v_n}-\Gpp_{x+e_1,v_n}\le\Gpp_{x,\fl{a_n\eta_\ell}}-\Gpp_{x+e_1,\fl{a_n\eta_\ell}}.\]
Take $n\to\infty$ and apply \eqref{bu:G1} to the sequences $\fl{a_n\zeta_\ell}$ and  $\fl{a_n\eta_\ell}$.  This works because $\fl{a_n\zeta_\ell}$  is a subset of $\fl{n\zeta_\ell}$ that escapes to infinity.  Thus for  $\what\in\Ombig_0$
	\begin{align*}
	B^{\overline\zeta_\ell+}(\what,x,x+e_1)&\le\varliminf_{n\to\infty}(\Gpp_{x,v_n}(\w)-\Gpp_{x+e_1,v_n}(\w))\\
	&\le\varlimsup_{n\to\infty}(\Gpp_{x,v_n}(\w)-\Gpp_{x+e_1,v_n}(\w))
	\le B^{\underline\eta_\ell-}(\what,x,x+e_1).
	\end{align*}
	
Take $\ell\to\infty$ and apply \eqref{eq:cont} to arrive at \eqref{bu:G1} as stated.
  \eqref{bu:G2} follow similarly. 
\end{proof}

\section{Directional geodesics}
\label{sec:geod}

This section proves the results on geodesics.  We  work on the extended space $\Ombig=\Omega\times\Omega'$ and define geodesics in terms of   the cocycles $B^{\xi\pm}$ constructed in Theorem \ref{th:cocycles}.    The idea   is in the next  lemma, followed by  the definition of cocycle geodesics.

\begin{lemma}\label{lm:grad flow}
Let $B$ be any shift-covariant cocycle that recovers potential $V(\what)=\w_0$, as in Definitions \ref{def:cK} and \ref{def:bdry-model}. 
Fix $\what$ so that properties \eqref{def:cK:stat}--\eqref{def:cK:coc}   of Definition \ref{def:cK} and \eqref{eq:VB} hold for all translations $T_x\what$.

\smallskip

\begin{enumerate}[\ \ \rm(a)]
\item\label{lm:grad flow:a}  Let $x_{m,n}=(x_k)_{k=m}^n$ be any up-right path  that  follows  minimal  gradients of $B$, that is, 
\[\w_{x_k}=B(\what, x_k,x_{k+1}) \qquad\text{for all $m\le k<n$.}\]
  Then $x_{m,n}$ is a geodesic from $x_m$ to $x_n$:
	\be\label{g:GB} \Gpp_{x_m,x_n}(\w)=\sum_{k=m}^{n-1}\w_{x_k}=B(\what, x_m,x_n).\ee
\item\label{lm:grad flow:b}
Let $x_{m,n}=(x_k)_{k=m}^n$ be an up-right path such that for all $m\le k<n$  
\begin{align*}
\begin{split}
\text{either}\quad\w_{x_k}&=B(x_k,x_{k+1})<B(x_k,x_k+e_1)\vee B(x_k,x_k+e_2)\\
\text{or}\quad x_{k+1}&=x_k+e_2\text{ and }B(x_k,x_k+e_1)=B(x_k,x_k+e_2).
\end{split}
\end{align*}
In other words, path $x_{m,n}$ follows   minimal gradients of $B$    and  takes an $e_2$-step in a tie. 
Then $x_{m,n}$ is the leftmost geodesic from $x_m$ to $x_n$.  Precisely,  if $\bar x_{m,n}$ is an up-right path from $\bar x_m=x_m$ to $\bar x_n=x_n$ and $\Gpp_{x_m,x_n}=\sum_{k=m}^{n-1}\w_{\bar x_k}$,
then $x_k\cdot e_1\le \bar x_k\cdot e_1$ for all $m\le k\le n$.

If ties are broken by   $e_1$-steps the resulting geodesic is   the rightmost geodesic between $x_m$ and $x_n$:
$x_k\cdot e_1\ge \bar x_k\cdot e_1$ for all $m\le k<n$.
\end{enumerate}
\end{lemma}

\begin{proof}
Part \eqref{lm:grad flow:a}. Any up-right path $\xbar_{m,n}$ from $\xbar_m=x_m$ to $\xbar_n=x_n$ satisfies 
	\[\sum_{k=m}^{n-1}\w_{\xbar_k}\le \sum_{k=m}^{n-1}B(\xbar_k,\xbar_{k+1})=B(x_m,x_n)=\sum_{k=m}^{n-1}B(x_k,x_{k+1})=\sum_{k=m}^{n-1}\w_{x_k}.  \]

Part \eqref{lm:grad flow:b}. 
$x_{m,n}$ is a geodesic by part \eqref{lm:grad flow:a}. To prove that it is the leftmost geodesic 
assume $\bar x_k=x_k$ and $x_{k+1}=x_k+e_1$. Then $\w_{x_k}=B(x_k,x_k+e_1)<B(x_k,x_k+e_2)$.   Recovery of the weights gives    $\Gpp_{x,y}\le B(x,y)$ for all $x\le y$. 
Combined   with \eqref{g:GB},  
	\[\w_{x_k}+\Gpp_{x_k+e_2,x_n}<B(x_k,x_k+e_2)+B(x_k+e_2,x_n)=B(x_k,x_n)=\Gpp_{x_k,x_n}.\]
Hence also  $\bar x_{k+1}=\bar x_k+e_1$ and  the claim about being the leftmost geodesic is proved. The other claim is symmetric. 
\end{proof}

Next we define a cocycle geodesic, that is, a geodesic constructed by following minimal gradients of a   cocycle $B^{\xi\pm}$  constructed in Theorem \ref{th:cocycles}.   Since our treatment allows discrete distributions, we introduce a    function $\t$ on $\Z^2$   to resolve ties.   
For $\xi\in\ri\Uset$, $u\in\Z^2$, and $\t\in\{e_1,e_2\}^{\Z^2}$, let $x^{u,\t,\xi\pm}_{0,\infty}$ be the up-right path (one path for $\xi+$, one for $\xi-$) starting at $x^{u,\t,\xi\pm}_0=u$ and satisfying for all $n\ge0$
	\[x^{u,\t,\xi\pm}_{n+1}
	=\begin{cases}
	x^{u,\t,\xi\pm}_n+e_1&\text{if }
	B^{\xi\pm}(x^{u,\t,\xi\pm}_n,x^{u,\t,\xi\pm}_n+e_1)<B^{\xi\pm}(x^{u,\t,\xi\pm}_n,x^{u,\t,\xi\pm}_n+e_2),\\
	x^{u,\t,\xi\pm}_n+e_2&\text{if }
	B^{\xi\pm}(x^{u,\t,\xi\pm}_n,x^{u,\t,\xi\pm}_n+e_2)<B^{\xi\pm}(x^{u,\t,\xi\pm}_n,x^{u,\t,\xi\pm}_n+e_1),\\
	x^{u,\t,\xi\pm}_n+\t(x^{u,\t,\xi\pm}_n)&\text{if }
	B^{\xi\pm}(x^{u,\t,\xi\pm}_n,x^{u,\t,\xi\pm}_n+e_1)=B^{\xi\pm}(x^{u,\t,\xi\pm}_n,x^{u,\t,\xi\pm}_n+e_2).
	\end{cases}\]
Cocycles 	
  $B^{\xi\pm}$  satisfy  $\w_x=B^{\xi\pm}(\what, x,x+e_1)\wedge B^{\xi\pm}(\what, x,x+e_2)$   
  (Theorem \ref{th:cocycles}(\ref{th:cocycles:exist})) and so  by Lemma \ref{lm:grad flow}\eqref{lm:grad flow:a}, $x^{u,\t,\xi\pm}_{0,\infty}$ is an infinite geodesic.  Since the cocycles are 
 defined on the  space $\Ombig$,  the geodesics are measurable functions on $\Ombig$.  
Recall from  Remark \ref{bu:kS} that under certain conditions a cocycle   $B^{\xi\pm}$ is  $\kS$-measurable.   When that happens,  the geodesic $x^{u,\t,\xi\pm}_{0,\infty}$ can be defined on $\Omega$ without the artificial extension to the space $\Ombig=\Omega\times\Omega'$.  
In particular, if $\gpp$ is differentiable at the endpoints of its linear segments (if any), all geodesics $x^{u,\t,\xi\pm}_{0,\infty}$ are   $\kS$-measurable.

If we restrict ourselves to the event $\Ombig_0$ of full $\Pbig$-measure on which monotonicity 
  \eqref{eq:monotone} holds for all $\xi,\zeta\in\ri\Uset$,  we can order  these geodesics in a  natural way from left to right.  
Define a partial ordering 
on $\{e_1,e_2\}^{\Z^2}$ by   $e_2\preceq e_1$   
and then $\t\preceq\t'$   coordinatewise.   Then  on the event $\Ombig_0$,  
for any  $u\in\Z^2$, $\t\preceq\t'$,  
$\xi,\zeta\in\ri\Uset$ with $\xi\cdot e_1<\zeta\cdot e_1$, and for all $n\ge0$,
	\begin{align}\label{eq:geod-order}
	 x_n^{u,\t,\xi\pm}\cdot e_1\le x_n^{u,\t',\xi\pm}\cdot e_1,\ \ 
	 x_n^{u,\t,\xi-}\cdot e_1\le x_n^{u,\t,\xi+}\cdot e_1,\ \ 
	 \text{and }\ x_n^{u,\t,\xi+}\cdot e_1\le x_n^{u,\t,\zeta-}\cdot e_1.
	\end{align}
	
  The  leftmost and rightmost tie-breaking rules are defined by $\tmin_x=e_2$ and $\tmax_x=e_1$ for all $x\in\Z^2$. 
 The cocycle limit \eqref{eq:cont} forces the cocycle geodesics to converge also, as the next lemma shows. 

\begin{lemma}\label{lm:cont geod}
Fix $\xi$ and let  $\zeta_n\to\xi$ in $\ri\Uset$.  
If  
$\zeta_n\cdot e_1>\xi\cdot e_1$  $\forall n$  then for all $u\in\Z^2$ 
	\begin{align}\label{eq:cont geod}
	\Pbig\{\forall k\ge0\;\exists n_0<\infty:n\ge n_0\Rightarrow x_{0,k}^{u,\tmax,\zeta_n\pm}=x_{0,k}^{u,\tmax,\xi+}\}=1.
	\end{align}
Similarly,  if $\zeta_n\cdot e_1\nearrow\xi\cdot e_1$ the limit holds with $\xi-$ on the right and $\tmax$ replaced by $\tmin$.  
\end{lemma}

\begin{proof}
It is enough to prove the statement for $u=0$. By  \eqref{eq:cont}, for a given $k$ and  large enough $n$,  if $x\ge0$ with $\abs{x}_1\le k$ and
$B^{\xi+}(x,x+e_1)\ne B^{\xi+}(x,x+e_2)$, then $B^{\zeta_n\pm}(x,x+e_1)-B^{\zeta_n\pm}(x,x+e_2)$ does not vanish and has the same sign as $B^{\xi+}(x,x+e_1)-B^{\xi_+}(x,x+e_2)$.
From such an $x$ geodesics following the minimal gradient of $B^{\zeta_n\pm}$ or the minimal gradient of $B^{\xi+}$ stay together for their next step.
On the other hand, when $B^{\xi+}(x,x+e_1)= B^{\xi+}(x,x+e_2)$, monotonicity \eqref{eq:monotone} implies 
	\[B^{\zeta_n\pm}(x,x+e_1)\le B^{\xi+}(x,x+e_1)=B^{\xi+}(x,x+e_2)\le B^{\zeta_n\pm}(x,x+e_2).\]
Once again, both the geodesic following the minimal gradient of $B^{\zeta_n\pm}$ and rules $\tmax$ and the one following the minimal gradients of $B^{\xi+}$ and rules $\tmax$ will next take 
the same $e_1$-step. This proves \eqref{eq:cont geod}. The other claim is similar. 
\end{proof}

Recall the line segments  $\Uset_\xi$,  $\Uset_{\xi\pm}$ defined in \eqref{eq:sector1}--\eqref{eq:sector2}.  
The endpoints  of   $\Uset_\xi=[\ximin,  \ximax]$   are  given  by
	\begin{align*}
	&\ximax\cdot e_1=\sup\{\alpha:(\alpha,1-\alpha)\in\Uset_{\xi+}\}\quad\text{and}\quad\ximin\cdot e_1=\inf\{\alpha:(\alpha,1-\alpha)\in\Uset_{\xi-}\}.
	\end{align*}
By Lemma \ref{gppa-lm} both points are again in $\ri\Uset$.  When needed we extend this definition to the endpoints of $\Uset$ by  $\Uset_{e_i}=\Uset_{e_i\pm}=\{e_i\}$, $i\in\{1,2\}$.

The next theorem concerns the direction of the cocycle  geodesics.

\begin{theorem}\label{th:B-geod:direction}
We have these two statements:
	\be\label{g:77} \Pbig\Big\{\forall\xi\in\ri\Uset,\forall\t\in\{e_1,e_2\}^{\Z^2},\forall u\in\Z^2:x^{u,\t,\xi\pm}_{0,\infty}\text{ is $\Uset_{\xi\pm}$-directed}\Big\}=1.\ee
	If $\xi\in\Diff$ then the statement should be taken without the $\pm$.   
\end{theorem}

\begin{proof}
Fix $\xi\in\ri\Uset$ and abbreviate $x_n=x^{u,\tmax,\xi+}_n$.    Since $B^{\xi+}$ recovers weights $\w$, 
Lemma \ref{lm:grad flow}\eqref{lm:grad flow:a} implies that $\Gpp_{u,x_n}=B^{\xi+}(u,x_n)$. 
 Furthermore, $B^{\xi+}(x,y)+h(\xi+)\cdot(y-x)$ is a centered cocycle, as in Definition \ref{def:cK}. Theorem \ref{th:Atilla} implies then 
	\[\lim_{n\to\infty}\abs{x_n}_1^{-1} ( \Gpp_{u,x_n} + h(\xi+)\cdot x_n  ) =0\qquad\Pbig\text{-almost surely}.\]

Define $\zeta(\what)\in\Uset$ by $\zeta\cdot e_1 = \varlimsup \frac{x_n\cdot e_1}{\abs{x_n}_1}$. 
If $\zeta\cdot e_1 >  \ximax\cdot e_1 $ then $\zeta\not\in\Uset_{\xi+}$ and hence
	\[\gpp(\zeta) +h(\xi+)\cdot\zeta=\gpp(\zeta)  - \nabla \gpp(\xi+) \cdot\zeta  <  \gpp(\xi)  - \nabla\gpp(\xi+) \cdot\xi = 0.\]  
(The first and last equalities come from \eqref{eq:h=grad} and \eqref{eq:solution-pp}.)
Consequently, by the shape theorem (limit \eqref{lln5}), 
on the event $\{\zeta\cdot e_1>\ximax\cdot e_1\}$  
\[  \varliminf_{n\to\infty}  \abs{x_n}_1^{-1} ( \Gpp_{u,x_n} + h(\xi+)\cdot x_n ) <0.  \] 
This proves that 
\[\Pbig\Bigl\{\,\varlimsup_{n\to\infty}  \frac{x_n^{u,\tmax,\xi+}\cdot e_1}{\abs{x_n^{u,\tmax,\xi+}}_1}\le\ximax\cdot e_1\Bigr\}=1. \]
Repeat the same argument with $\tmax$ replaced by $\tmin$ and $\ximax$ by the other endpoint of 
$\Uset_{\xi+}$ (which is either $\xi$ or $\ximin$). To capture all $\t$ use   geodesics ordering \eqref{eq:geod-order}.   An analogous argument works for $\xi-$.  We have,  for a given $\xi$,  
	\be\label{g:78}  \Pbig\Bigl\{  \forall\t\in\{e_1,e_2\}^{\Z^2},\forall u\in\Z^2:x_{0,\infty}^{u,\t,\xi\pm}\text{  is }\Uset_{\xi\pm}\text{-directed}\Big\}=1.\ee
 
Let $\Ombig_0$ be an event of full $\Pbig$-probability on which all cocycle  geodesics satisfy the ordering \eqref{eq:geod-order}, and   the event in \eqref{g:78} holds for $\xi\pm$ in  a countable set $\Uset_0$ that contains all points of nondifferentiability of $\gpp$ and a countable dense subset of $\Diff$.   We argue that   $\Ombig_0$  is contained in the event in \eqref{g:77}.  

Let $\zeta\notin\Uset_0$ and let $\zetamax$ denote the right endpoint of $\Uset_{\zeta}$.  We show that 
\be\label{g:79}   \varlimsup_{n\to\infty}  \frac{x_n^{u,\tmax,\zeta}\cdot e_1}{\abs{x_n^{u,\tmax,\zeta}}_1}\le\zetamax\cdot e_1 \qquad\text{on the event $\Ombig_0$.} \ee
(Note that $\zeta\in\Diff$ so there is no $\zeta\pm$ distinction in the cocycle geodesic.)    The $\varliminf $ with $\tmin$ and $\ge \zetamin\cdot e_1$ comes of course with the same argument.  

If $\zeta\cdot e_1< \zetamax\cdot e_1$ pick  $\xi\in\Diff\cap\Uset_0$ so  that  $\zeta\cdot e_1< \xi\cdot e_1<\zetamax\cdot e_1$.    Then $\ximax=\zetamax$ and  \eqref{g:79} follows from the ordering.  

If $\zeta= \zetamax$, let $\e>0$ and  pick  $\xi\in\Diff\cap\Uset_0$ so  that  
$\zeta\cdot e_1< \xi\cdot e_1\le \ximax\cdot e_1<\zeta\cdot e_1+\e$.  This is possible because $\nabla\gpp(\xi)$ converges to but never equals $\nabla\gpp(\zeta)$  as  $\xi\cdot e_1\searrow \zeta\cdot e_1$.  Again by the ordering 
\[     \varlimsup_{n\to\infty}  \frac{x_n^{u,\tmax,\zeta}\cdot e_1}{\abs{x_n^{u,\tmax,\zeta}}_1}\le
\varlimsup_{n\to\infty}  \frac{x_n^{u,\tmax,\xi}\cdot e_1}{\abs{x_n^{u,\tmax,\xi}}_1}\le\ximax\cdot e_1 <\zeta\cdot e_1+\e. \] 
This completes the proof of Theorem \ref{th:B-geod:direction}.  
 \end{proof}



\begin{lemma}\label{lm:sandwich}

{\rm (a)}   Fix $\xi\in\ri\Uset$.  Then the  following statement holds $\Pbig$-almost surely.  For any  geodesic $x_{0,\infty}$  
\begin{align}\label{eq:sandwich}
\varlimsup_{n\to\infty}\frac{x_n\cdot e_1}{\abs{x_n}_1} \ge \xi\cdot e_1 \quad\text{implies that}\quad 
		x_n\cdot e_1  \ge  x^{x_0,\tmin,\ximin-}_n\cdot e_1 \quad\text{for all }n\ge0   
		\end{align}
and 
\begin{align}\label{eq:sandwich2}
\varliminf_{n\to\infty}\frac{x_n\cdot e_1}{\abs{x_n}_1} \le \xi\cdot e_1 \quad\text{implies that}\quad 
		x_n\cdot e_1  \le  x^{x_0,\tmax,\ximax+}_n\cdot e_1 \quad\text{for all }n\ge0. 
		\end{align}
		
	
{\rm (b)} Fix a maximal line segment $[\zetamin, \zetamax]$ on which $\gpp$ is linear and such that $\zetamin\cdot e_1<\zetamax\cdot e_1$.  Assume 	$\zetamin$ and  $\zetamax$ are both points of differentiability of $\gpp$.    Then the  following statement holds $\Pbig$-almost surely.  Any  geodesic $x_{0,\infty}$   such that   a limit point of  $x_n/\abs{x_n}_1$  lies in $[\zetamin, \zetamax]$
satisfies 
	\begin{align}\label{eq:sandwich4}
		x^{x_0,\tmin,\zetamin}_n\cdot e_1\le x_n\cdot e_1\le x^{x_0,\tmax,\zetamax}_n\cdot e_1\quad\text{for all }n\ge0.
	\end{align}
\end{lemma}


\begin{proof}  Part (a).  
We prove    \eqref{eq:sandwich}.  \eqref{eq:sandwich2} is proved similarly. 

  Fix a sequence 
 $\zeta_\ell\in\Diff$ such that   $\zeta_\ell\cdot e_1\nearrow \ximin\cdot e_1$ so that, in particular,  $\xi\not\in\Uset_{\zeta_\ell}$.     The good event of full $\Pbig$-probability is the one on which 
$x_{0,\infty}^{x_0,\tmin,\zeta_\ell}$ is $\Uset_{\zeta_\ell}$-directed (Theorem \ref{th:B-geod:direction}),    $x_{0,\infty}^{x_0,\tmin,\zeta_\ell}$ is the leftmost geodesic between any two of its points   (Lemma \ref{lm:grad flow}\eqref{lm:grad flow:b}  applied to cocycle $B^{\zeta_\ell}$) and  $x_{0,\infty}^{x_0,\tmin,\zeta_\ell}\to x^{x_0,\tmin,\ximin-}_{0,\infty}$  (Lemma \ref{lm:cont geod}).  

By the leftmost property,  if 
  $x_{0,\infty}^{x_0,\tmin,\zeta_\ell}$ ever goes strictly to the right of
$x_{0,\infty}$, these two geodesics  cannot touch   again at any later time.  
But   by virtue of the limit points,  $x_n^{x_0,\tmin,\zeta_\ell}\cdot e_1<x_n\cdot e_1$ for  infinitely many $n$.  
Hence  $x_{0,\infty}^{x_0,\tmin,\zeta_\ell}$ stays  weakly to the left of $x_{0,\infty}$.  Let $\ell\to\infty$. 
 
Part (b) is proved similarly.  The differentiability assumption implies that the geodesic $x_{0,\infty}^{x_0,\tmin,\zetamin}$ can be approached from the left by geodesics $x_{0,\infty}^{x_0,\tmin,\zeta_\ell}$ such that  $\zetamin\not\in\Uset_{\zeta_\ell}$.  
\end{proof}

 Recall the set $\EP$ of exposed points of differentiability defined below \eqref{eq:epod}.  Define sets $\EPlim^+$ and $\EPlim^-$  of their one-sided limit points  by
	\begin{align*}
	\EPlim^\pm=\{\zeta\in\ri\Uset:\exists\xi_n\in\EP\text{ such that }\xi_n\to\zeta \text{ and }\pm\xi_n\cdot e_1>\pm\zeta\cdot e_1\}. 
	\end{align*}

\begin{corollary}\label{cor:leftmost}
  Fix $\xi\in\ri\Uset$ such that $\ximin\in\EP\cup\EPlim^-$.  Then  $\Pbig$-almost surely and for all $u\in\Z^2$,  $x^{u,\tmin,\ximin-}_{0,\infty}$  
is the leftmost $\Uset_{\xi-}$-geodesic out of $u$.
Similarly, for a fixed $\xi$ with $\ximax\in\EP\cup\EPlim^+$, $\Pbig$-almost surely and for all $u\in\Z^2$, $x^{u,\tmax,\ximax+}_{0,\infty}$ is the rightmost $\Uset_{\xi+}$-geodesic out of  $u$.
\end{corollary}

\begin{proof}
Theorem \ref{th:B-geod:direction} implies that $x_{0,\infty}^{u,\tmin,\ximin-}$ is a $\Uset_{\ximin-}$-geodesic.
If $\ximin\in\EP\cup\EPlim^-$ then either $\Uset_{\ximin-}=\Uset_{\xi-}$ or $\Uset_{\ximin-}=\{\ximin\}\subset\Uset_{\xi-}$.
Thus, $x_{0,\infty}^{u,\tmin,\ximin-}$ is a $\Uset_{\xi-}$-geodesic.   $\ximin\in\EP\cup\EPlim^-$ implies that there cannot be  a linear segment adjacent to $\ximin$ to the left, and consequently $\underline\ximin=\ximin$.  
Lemma \ref{lm:sandwich}(a) implies that $x_{0,\infty}^{u,\tmin,\ximin-}$ is to the left of any other $\Uset_{\xi-}$-geodesic out of $u$. 
The claim about rightmost geodesics is proved similarly.
\end{proof}

The next result concerns coalescence of geodesics $\{x^{u,\t,\xi\pm}_{0,\infty}:u\in\Z^2\}$, for a fixed $\xi\in\ri\Uset$. Its proof uses the stationarity of the distribution of this family of geodesics, so the tie-breaking rule has to be shift-covariant, like the cocycles. For this reason we allow the rule to depend on $\what$. A \emph{shift-covariant tie-breaking rule} is a measurable function $\t:\Ombig\times\Z^2\to\{e_1,e_2\}$ such that, for $\Pbig$-almost every $\what$,
\be\label{t:stat}
\t(\what,x+z)=\t(T_z\what,x)\qquad\text{for all }x,z\in\Z^2.
\ee
The geodesic $x^{u,\t,\xi\pm}_{0,\infty}(\what)$ is then defined as above with the function $\t(\what,\cdot)\in\{e_1,e_2\}^{\Z^2}$. The statements above that hold simultaneously for all functions $\t\in\{e_1,e_2\}^{\Z^2}$, such as \eqref{eq:geod-order} and \eqref{g:77}, apply to it. The constant rules $\tmin$ and $\tmax$ are shift-covariant. A function $\t\in\{e_1,e_2\}^{\Z^2}$ that is not constant is not a shift-covariant rule.

The proof also modifies the weights on a finite set of sites, and this modification must not change the tie-breaking steps at the other sites. This is ensured by the following condition. For a finite set $\cV\subset\Z^2$ let $\cV^*=\bigcap_{x\in\cV}\{y\in\Z^2:y\not\le x\}$. The weights and a shift-covariant tie-breaking rule $\t$ satisfy the \emph{finite energy condition}
 if, for every finite $\cV\subset\Z^2$ and every $K\in\R$ such that $\P\{\w_0\le K\}>0$,
\be\label{t:fe}
\begin{aligned}
&\Pbig\bigl\{\w_x\le K\ \,\forall x\in\cV\,\big\vert\,\mathcal G_\cV\bigr\}>0\qquad\text{$\Pbig$-almost surely, where}\\
&\mathcal G_\cV=\sigma\bigl\{\w_x,\,\t(x):x\notin\cV;\ \ \w^{i,\alpha}_y: y\in\cV^*,\,i\in\{1,2\},\,\alpha\in\cA_0\bigr\}.
\end{aligned}
\ee
For the constant rules $\tmin$ and $\tmax$ the tie-breaking steps carry no information, and \eqref{t:fe} holds because, by Theorem \ref{th:cocycles}\eqref{th:cocycles:indep}, the weights $\{\w_x:x\in\cV\}$ are independent of $\mathcal G_\cV$. For the same reason, and because then $\t(x)$ is a function of $\w_x$, \eqref{t:fe} holds for a rule of the form $\t(\what,x)=f(\w_x)$ with a measurable function $f:\R\to\{e_1,e_2\}$.

\begin{theorem}\label{th:geod:coal:t}
  Fix $\xi\in\ri\Uset$, let $\t$ be a shift-covariant tie-breaking rule, and assume the finite energy condition \eqref{t:fe}.
Then $\Pbig$-almost surely, for all $u,v\in\Z^2$, there exist $n,m\ge0$ such that
$x^{u,\t,\xi-}_{n,\infty}=x^{v,\t,\xi-}_{m,\infty}$, with a similar statement for $\xi+$.
\end{theorem}
By the paragraph before the theorem, Theorem \ref{th:geod:coal:t} applies to the constant rules $\tmin$ and $\tmax$, which are the only rules used in the rest of the paper.

Theorem \ref{th:geod:coal:t} is proved by  adapting the argument  of \cite{Lic-New-96},  originally presented  for first passage percolation and  later ported  by \cite{Fer-Pim-05} to the exactly solvable  corner growth model  with exponential weights.   Briefly, the idea is the following.    Stationarity and two nonintersecting geodesics create three nonintersecting geodesics.   A modification of the weights turns  the middle geodesic of the triple   into a geodesic that stays disjoint from all geodesics that emanate from sufficiently far away.    Stationarity again gives at least $\delta L^2$ such disjoint geodesics emanating from an $L\times L$ square.  This gives a contradiction because there are only $2L$ boundary points for these geodesics to exit through.  
   The details are in Appendix \ref{app:geod}.  

\smallskip 
 
To get to  uniqueness of geodesics, we   show that continuity of the distribution of   $\w_0$ prevents   ties between cocycle  weights.    (The construction of the cocycles implies, through eqn.~\eqref{eq:ne-induction:q}, that the variables $B^{\xi\pm}(x,y)$ have continuous marginal distributions, but here we need a property of the joint distribution.)    Consequently  for a given $\xi$,   $\Pbig$-almost surely     cocycle geodesics $x_{0,\infty}^{u,\t,\xi\pm}$ do not depend on $\t$.

\begin{lemma}\label{lm:geod:unique:t}
Assume \eqref{2d-ass} and  that $\P\{\w_0\le r\}$ is a continuous function of $r\in\R$. Fix $\xi\in\ri\Uset$. Then for all $u\in\Z^2$, 
\[\Pbig\{B^{\xi+}(u,u+e_1)=B^{\xi+}(u,u+e_2)\}=\Pbig\{B^{\xi-}(u,u+e_1)=B^{\xi-}(u,u+e_2)\}=0.\]
\end{lemma}


\begin{proof} 
Due to shift invariance it is enough to prove the claim for $u=0$. 
We  work with the case   $\xi+$, the other case being similar.

Assume by way of contradiction that the probability in question is positive. Let $\t=\tmax$. By Theorem \ref{th:geod:coal:t}, $x_{0,\infty}^{e_2,\t,\xi+}$ and $x_{0,\infty}^{e_1,\t,\xi+}$   coalesce with probability one. 
Hence there exists   $v\in\Z^2$ and  $n\ge1$ such that
	\[\P\big\{B^{\xi+}(0,e_1)=B^{\xi+}(0,e_2),\,x_n^{e_1,\t,\xi+}=x_n^{e_2,\t,\xi+}=v\big\}>0.\]
Note that if $B^{\xi+}(0,e_1)=B^{\xi+}(0,e_2)$ then both are equal to $\w_0$. Furthermore, by Lemma \ref{lm:grad flow}\eqref{lm:grad flow:a} we have
	\[B^{\xi+}(e_1,v)=\sum_{k=0}^{n-1}\w(x_k^{e_1,\t,\xi_+})\quad\text{and}\quad B^{\xi+}(e_2,v)=\sum_{k=0}^{n-1}\w(x_k^{e_2,\t,\xi_+}).\]
(For aesthetic reasons we wrote $\w(x)$ instead of $\w_x$.) Thus
	\begin{align*}
	&\w_0+\sum_{k=0}^{n-1}\w({x_k^{e_1,\t,\xi_+}})=B^{\xi+}(0,e_1)+B^{\xi+}(e_1,v)= B^{\xi+}(0,v)\\
	&\qquad=B^{\xi+}(0,e_2)+B^{\xi+}(e_2,v)=\w_0+\sum_{k=0}^{n-1}\w({x_k^{e_2,\t,\xi_+}}).
	\end{align*}
The fact that this happens with positive probability contradicts the assumption that $\w_x$ are i.i.d.\ and have a continuous distribution. The lemma is proved.
\end{proof}

It is known that, in general, uniqueness of geodesics cannot hold simultaneously for all directions.  In our development this is a consequence of Theorem \ref{th:two-geo} below.  


\begin{proof}[Proof of Theorem \ref{thm-1:geod}]  
Part \eqref{thm-1:geod:exist}.   The existence of $\Uset_{\xi\pm}$-directed geodesics for $\xi\in\ri\Uset$  follows by fixing $\t$ and  taking geodesics  $x^{u,\t,\xi\pm}_{0,\infty}$ from  Theorem \ref{th:B-geod:direction}.   For $\xi=e_i$  geodesics are simply  $x_{0,\infty}=(ne_i)_{n\ge0}$.   

 Let  $\Diff_0$ be a dense countable subset of $\Diff$.  Let   $\Ombig_0$  be the event   of full $\Pbig$-probability on which event \eqref{g:77} holds  
and Lemma \ref{lm:sandwich}(a) holds  for each  $u\in\Z^2$ and   $\zeta\in\Diff_0$.  We  show that   on  $\Ombig_0$,  every geodesic is $\Uset_\xi$-directed for some $\xi\in\Uset$.  

  Fix $\what\in\Ombig_0$ and an arbitrary geodesic $x_{0,\infty}$.  
Define $\xi'\in\Uset$ by 
\[   \xi'\cdot e_1=\varlimsup_{n\to\infty}\frac{x_n\cdot e_1}{\abs{x_n}_1}.  \]  
Let $\xi=\ximin'=$ the left endpoint of $\Uset_{\xi'}$.  
We claim that   $x_{0,\infty}$ is $\Uset_\xi=[\ximin, \ximax]$-directed.  
If $\xi'=e_2$ then  $x_n/\abs{x_n}_1\to e_2$ and $\Uset_\xi=\{e_2\}$ and the case is closed.
Suppose  $\xi'\ne e_2$.

The definition of $\xi$ implies that $\xi'\in\Uset_{\xi+}$ and so 
\[  \varlimsup_{n\to\infty}\frac{x_n\cdot e_1}{\abs{x_n}_1} =\xi'\cdot e_1\le     \ximax\cdot e_1.  \]

From the other direction,  for any $\zeta\in\Diff_0$ such that $ \zeta\cdot e_1<  \xi'\cdot e_1$ we have 
\[  \varlimsup_{n\to\infty}\frac{x_n\cdot e_1}{\abs{x_n}_1} > \zeta\cdot e_1 \]
which by \eqref{eq:sandwich} implies  $x_n\cdot e_1  \ge  x^{x_0,\tmin,\zetamin}_n\cdot e_1$. 
Then  by  \eqref{g:77} 
\[  \varliminf_{n\to\infty}\frac{x_n\cdot e_1}{\abs{x_n}_1} \; \ge \; \varliminf_{n\to\infty} \frac{x^{x_0,\tmin,\zetamin}_n\cdot e_1}{\abs{x^{x_0,\tmin,\zetamin}_n}_1} 
\;\ge\; \underline\zetamin\cdot e_1 \]
where  $\underline\zetamin=$  the left endpoint of $\Uset_{\zetamin}$.  It remains to observe that we can take $\underline\zetamin\cdot e_1$ arbitrarily close to $\ximin\cdot e_1$.    If $ \xi\cdot e_1<  \xi'\cdot e_1$  then we   take  $ \xi\cdot e_1< \zeta\cdot e_1<  \xi'\cdot e_1$ in which case $\zetamin=\xi$ and $\underline\zetamin=\ximin$.    If $\xi=\xi'$ then also $\ximin=\ximin'=\xi$. In this case,  as  $\Diff_0\ni\zeta\nearrow \xi$, $\nabla g(\zeta)$ approaches but never equals $\nabla g(\xi-)$ because there is no flat segment of $\gpp$ adjacent to $\xi$ on the left.  This forces both $\zetamin$ and $\underline\zetamin$ to converge to $\xi$.


\smallskip 

Part  \eqref{thm-1:geod:direct} .   If $\gpp$ is strictly concave then  $\Uset_\xi=\{\xi\}$ for all $\xi\in\ri\Uset$ and part  \eqref{thm-1:geod:direct}  follows from part \eqref{thm-1:geod:exist}.  

\smallskip 

Part \eqref{thm-1:geod:cont}.   
By Theorem \ref{th:cocycles}\eqref{th:cocycles:flat} there is a single cocycle $B^\xi$ simultaneously  for all $\xi\in[\zetamin,\zetamax]$.   Consequently cocycle geodesics 
$x_{0,\infty}^{x_0,\t,\zetamin}$ and $x_{0,\infty}^{x_0,\t,\zetamax}$ coincide for any given tie breaking function $\t$.  
By Corollary  \ref{cor:buse} this cocycle $B^\xi$ is $\kS$-measurable and hence so are the cocycle geodesics. On the event of full $\P$-probability   on which there are no ties between $B^\xi(x,x+e_1)$ and $B^\xi(x,x+e_2)$ the tie breaking function $\t$ makes no difference.  Hence 
the left and right-hand side of \eqref{eq:sandwich4} coincide.   Thus there is no room for two $[\zetamin,\zetamax]$-directed geodesics from any point. Coalescence comes from Theorem \ref{th:geod:coal:t}, applied with $\t=\tmin$. 
  \end{proof}

\begin{proof}[Proof of Theorem \ref{thm:buse-geo}]
Part \eqref{thm:buse-geo:i} follows from Lemma \ref{lm:grad flow}.\medskip

Part \eqref{thm:buse-geo:ii}. 
Take sequences $\eta_n, \zeta_n\in\ri\Uset$ with $\eta_n\cdot e_1<\ximin\cdot e_1\le \ximax\cdot e_1<\zeta_n\cdot e_1$ and $\zeta_n\to\ximax$,  $\eta_n\to\ximin$. Consider  the full measure event on which Theorem \ref{th:construction} holds for each  $\zeta_n$ and $\eta_n$ 
with sequences $v_m=\fl{m\zeta_n}$ and $\fl{m\eta_n}$, and on  which 
continuity \eqref{eq:cont} holds as $\zeta_n\to\ximax$,  $\eta_n\to\ximin$. In the rest of the proof we drop the index $n$ from $\zeta_n$ and $\eta_n$.

We prove the case of an infinite geodesic 
  $x_{0,\infty}$   that satisfies $x_0=0$ and  \eqref{geod-98}. 
For large $m$, $\fl{m\eta\cdot e_1}<x_m\cdot e_1<\fl{m\zeta\cdot e_1}$.

Consider first the case $x_1=e_1$. If there exists a geodesic from $0$ to $\fl{m\zeta}$ that goes through $e_2$, then this geodesic would intersect $x_{0,\infty}$ and thus there would
exist another geodesic that goes from $0$ to $\fl{m\zeta}$ passing through $e_1$. In this case we would have $\Gpp_{e_1,\fl{m\zeta}}=\Gpp_{e_2,\fl{m\zeta}}$. On the other hand, if there exists a geodesic from $0$ to $\fl{m\zeta}$ that goes through $e_1$, then we would have $\Gpp_{e_1,\fl{m\zeta}}\ge\Gpp_{e_2,\fl{m\zeta}}$. Thus,
in either case, we have 
	\[\Gpp_{0,\fl{m\zeta}}-\Gpp_{e_1,\fl{m\zeta}}\le \Gpp_{0,\fl{m\zeta}}-\Gpp_{e_2,\fl{m\zeta}}.\]
Taking $m\to\infty$ and applying Theorem \ref{th:construction} we have $B^{\overline\zeta+}(0,e_1)\le B^{\overline\zeta+}(0,e_2)$.
Taking $\zeta\to\ximax$ and applying \eqref{eq:cont} we have $B^{\ximax+}(0,e_1)\le B^{\ximax+}(0,e_2)$.
Since $\ximax$ and $\xi$ are points of differentiability of $\gpp$, we have $B^{\ximax+}=B^{\xi}$. Consequently, we have shown $B^\xi(0,e_1)\le B^{\xi}(0,e_2)$.  Since $B^\xi$ recovers the potential
(Definition \ref{def:bdry-model}),    the first step satisfies $\w_{0}=B^\xi(0,e_1)\wedge B^{\xi}(0,e_2)=B^\xi(0, x_1)$.

When $x_1=e_2$ repeat the same argument   with $\eta$ in place of $\zeta$ to get  $B^{\xi}(0,e_2)\le B^\xi(0,e_1)$. This proves the theorem for the first step of the geodesic and that 
  is enough.\medskip

Part \eqref{thm:buse-geo:iii}. 
The statement holds if $B^\xi(0,e_1)=B^\xi(0,e_2)$, since then both are equal to $\w_0$ by potential recovery \eqref{eq:VB}.
If $\w_0=B^\xi(0,e_1)<B^\xi(0,e_2)$ then convergence \eqref{eq:grad:coc1} implies that for $n$ large enough $\Gpp_{e_1,v_n}>\Gpp_{e_2,v_n}$.
In this case any maximizing path from $0$ to $v_n$ will have to start with an $e_1$-step and the claim of the lemma is again true.
The case $B^\xi(0,e_1)>B^\xi(0,e_2)$ is similar.
\end{proof}

\begin{proof}[Proof of Theorem \ref{thm:lr-geod}]
Part \eqref{thm:lr-geod:i} follows from Corollary \ref{cor:leftmost} and Theorem \ref{th:geod:coal:t}, applied with the constant rules $\tmin$ and $\tmax$. Part \eqref{thm:lr-geod:ii} 
follows from Theorem \ref{thm:buse-geo}\eqref{thm:buse-geo:iii} and the fact that the geodesics in Corollary \ref{cor:leftmost} are 
the Busemann geodesics from Theorem \ref{thm:buse-geo}.
\end{proof}

\section{Competition interface}\label{sec:ci-pf} 

In this section we prove the results of Section \ref{sec:ci}.  As before, we begin by studying the situation on the extended space $\Ombig$ with the help of the cocycles $B^{\zeta\pm}$ of Theorem \ref{th:cocycles}.

\begin{lemma}\label{aux:ci}
Define $B^{e_1-}$ as the monotone limit of $B^{\zeta\pm}$ when $\zeta\to e_1$. Then $\Pbig$-almost surely $B^{e_1-}(0,e_1)=\w_0$ and $B^{e_1-}(0,e_2)=\infty$.
A symmetric statement holds for the limit as $\zeta\to e_2$.
\end{lemma}

\begin{proof}
The limit in the claim exists due to monotonicity \eqref{eq:monotone}. Furthermore, by potential recovery we have almost surely 
$B^{e_1-}(0,e_1)\ge\w_0$. On the other hand, dominated convergence, \eqref{grad g-1} and Lemma \ref{gppa-lm} give
	\[\Ebig[B^{e_1-}(0,e_1)]=\lim_{\zeta\to e_1}\Ebig[B^{\zeta\pm}(0,e_1)]=\lim_{\zeta\to e_1} e_1\cdot\nabla\gpp(\zeta\pm)=\Ew=\Ebig[\w_0].\]
  Thus, $B^{e_1-}(0,e_1)=\w_0$ almost surely. 

The cocycle property (Definition \ref{def:cK}\eqref{def:cK:coc}) and recovery (Definition \ref{def:bdry-model}), both of which are satisfied by $B^{e_1-}$, imply
the relation 
\begin{align*}
&B^{e_1-}(ne_1,ne_1+e_2) \\
&\qquad =\w_{ne_1}+\big(B^{e_1-}((n+1)e_1,(n+1)e_1+e_2)-B^{e_1-}(ne_1+e_2,(n+1)e_1+e_2)\big)^+\\
&\qquad =\w_{ne_1}+\big(B^{e_1-}((n+1)e_1,(n+1)e_1+e_2)-\w_{ne_1+e_2}\big)^+\,.
\end{align*}
The second equality is from the just proved identity $B^{e_1-}(x,x+e_1)=\w_x$.

Repeatedly dropping the outer $+$-part and applying the same formula inductively leads to 
 \begin{align*} B^{e_1-}(0,e_2)&\ge\w_0+\sum_{i=1}^{n}(\w_{ie_1}-\w_{(i-1)e_1+e_2})\\
 &\qquad + \big(B^{e_1-}((n+1)e_1,(n+1)e_1+e_2)-\w_{ne_1+e_2}\big)^+.\end{align*} 
Since the  summands are i.i.d.\ with mean $0$, taking $n\to\infty$ gives  $B^{e_1-}(0,e_2)=\infty$ almost surely.
\end{proof}

Next we use the cocycles to define a random variable on $\Ombig$  that represents the asymptotic direction of the competition interface.    Assume now that $\P\{\w_0\le r\}$ is continuous in $r$.
By Lemma \ref{lm:geod:unique:t}, with $\Pbig$-probability one,  $B^{\xi\pm}(0,e_1)\ne B^{\xi\pm}(0,e_2)$ for all rational $\xi\in\ri\Uset$. Furthermore, monotonicity \eqref{eq:monotone}
gives that
	\[B^{\zeta+}(0,e_1)-B^{\zeta+}(0,e_2)\le B^{\zeta-}(0,e_1)-B^{\zeta-}(0,e_2) \le B^{\eta+}(0,e_1)-B^{\eta+}(0,e_2)\]
when $\zeta\cdot e_1>\eta\cdot e_1$. Lemma \ref{aux:ci} implies that $B^{\zeta\pm}(0,e_1)-B^{\zeta\pm}(0,e_2)$ converges to $-\infty$ as $\zeta\to e_1$ and to $\infty$ as $\zeta\to e_2$. 
Thus there exists   a unique $\cid(\what)\in\ri\Uset$ such that   for rational $\zeta\in\ri\Uset$, 
\begin{align}\label{ci}
\begin{split}
&B^{\zeta\pm}(\what, 0,e_1)<B^{\zeta\pm}(\what, 0,e_2)\quad\text{if }\zeta\cdot e_1>\cid(\what)\cdot e_1\\
\text{ and}\qquad &B^{\zeta\pm}(\what, 0,e_1)>B^{\zeta\pm}(\what, 0,e_2)\quad\text{if }\zeta\cdot e_1<\cid(\what)\cdot e_1.
\end{split}
\end{align}

\begin{theorem}\label{th:two-geo}
Assume $\P\{\w_0\le r\}$ is continuous in $r$. Then on the extended space $\OBPbig$ of Theorem \ref{th:cocycles}  the  random variable $\cid(\what)\in\ri\Uset$  defined by \eqref{ci}   has the following properties. 
\begin{enumerate}[\ \ \rm(i)]
\item\label{th:two-geo:i}  $\Pbig$-almost surely, for every $x\in\Z^2$, there exist at least two $\Uset_{\cid(T_x\what)}$-geodesics out of $x$ that do not coalesce.  
\item\label{th:two-geo:ii} Recall  $\Uset_{\cid(\what)}=[\cidmin(\what),\cidmax(\what)]$ from \eqref{eq:sector2}. Then the following holds $\Pbig$-almost surely. Let   $x'_{0,\infty}$ and $x''_{0,\infty}$ be any geodesics with   
\[  \varliminf_{n\to\infty}  \frac{x'_n\cdot e_1}n<\cidmin(\what) 
\qquad\text{and}\qquad 
\varlimsup_{n\to\infty}  \frac{x''_n\cdot e_1}n>\cidmax(\what). \]  
Then   $x'_1=e_2$ and  $x''_1=e_1$.  
\item\label{th:two-geo:iii} $\Phat\{\what: \cid(\what)=\xi\}=0$ for any $\xi\in\Diff$. 
\item\label{th:two-geo:iv}   Fix $\zeta,\eta\in\ri\Uset$  such that $\zeta\cdot e_1<\eta\cdot e_1$ and $\nabla\gpp(\zeta+)\not=\nabla\gpp(\eta-)$,  with    $\Uset_\zeta=[\underline\zeta,\overline\zeta]$ and $\Uset_\eta=[\underline\eta,\overline\eta]$.
Then for $\Pbig$-almost every $\what$ there exists $z\in\Z^2$ such that $\cid(T_z\what)\in[\underline\zeta,\overline\eta]$. 
\end{enumerate}
\end{theorem}


\begin{proof}

%
%

Define 
\begin{align}\label{B*} 
\begin{split}
\Bci^+(\what,x,y)&=\lim_{\zeta\cdot e_1\searrow\cid(\what)\cdot e_1}B^{\zeta\pm}(\what,x,y)  \\
\text{ and}\qquad \Bci^-(\what,x,y)&=\lim_{\zeta\cdot e_1\nearrow\cid(\what)\cdot e_1}B^{\zeta\pm}(\what,x,y)\,.
\end{split} \end{align}   
As pointed out in Remark \ref{rm:cocycles}(c), we have to keep the  $\Bci^\pm$ distinction   even if $\gpp$ is everywhere differentiable, because
direction $\cid$ is random and continuity \eqref{eq:cont} has not been shown simultaneously  for all directions with a single  $\Pbig$-null set.   

In any case, $\Bci^\pm$ satisfy the cocycle property (Definition \ref{def:cK}\eqref{def:cK:coc}) and recovery $\w_x=\min_{i=1,2} \Bci^{\pm}(\what,x,x+e_i)$ (Definition \ref{def:bdry-model}).  
From \eqref{ci} we have  $\Bci^+(0,e_1)\le \Bci^+(0,e_2)$  and $\Bci^-(0,e_1)\ge \Bci^-(0,e_2)$.
By Lemma \ref{lm:grad flow}   there exists  a geodesic from $0$   through $e_1$  (by following  minimal $\Bci^+$ gradients) and another through $e_2$ (by following  minimal $\Bci^-$ gradients), with an arbitrary tie breaking function $\t$. These two geodesics cannot coalesce because $\w_0$ has a continuous distribution.

Let $\zeta\cdot e_1<\cid\cdot e_1<\eta\cdot e_1$. By  the limits in \eqref{B*} and monotonicity \eqref{eq:monotone}, 
\begin{align*}
\begin{split}
B^{\zeta+}(\what,x,x+e_1)&\ge \Bci^\pm(\what,x,x+e_1)\ge B^{\eta-}(\what,x,x+e_1)\\
\text{ and}\qquad B^{\zeta+}(\what,x,x+e_2)&\le \Bci^\pm(\what,x,x+e_2)\le B^{\eta-}(\what,x,x+e_2). 
\end{split} \end{align*}  
These inequalities imply that the $\Bci^\pm$-geodesics stay 
  to the right of $x_{0,\infty}^{0,\zeta+}$   and to the left of $x_{0,\infty}^{0,\eta-}$.   
  By Theorem \ref{th:B-geod:direction} these geodesics are $\Uset_{\zeta+}$- and $\Uset_{\eta-}$-directed, respectively.
Hence the $\Bci^\pm$-geodesics are $\Uset_{\cid}$-directed. Part \eqref{th:two-geo:i} is proved.

In part \eqref{th:two-geo:ii}  we prove the first claim, the other claim being similar.  The assumption allows us to pick a rational $\eta\in\ri\Uset$ such that
$\varliminf x_n'\cdot e_1/n<\underline\eta\cdot e_1\le\eta\cdot e_1<\cid\cdot e_1$. 
Since $\w_0$ has a continuous distribution and geodesic  $x_{0,\infty}^{0,\eta-}$ is $\Uset_{\eta-}$-directed, geodesic $x_{0,\infty}$ has to stay always to the left of it.
  \eqref{ci} implies $x_1^{0,\eta-}=e_2$. Hence  also $x_1=e_2$. The claim is proved.

For part \eqref{th:two-geo:iii} fix $\xi\in\ri\Diff$, which implies   $B^{\xi\pm}=B^\xi$.   By Lemma \ref{lm:geod:unique:t},   $B^\xi(0,e_1)\ne B^\xi(0,e_2)$ almost surely.    Let $\zeta\cdot e_1\searrow\xi\cdot e_1$ along rational points $\zeta\in\ri\Uset$.  By \eqref{eq:cont},  $B^{\zeta\pm}(0, e_i)\to B^\xi(0,e_i)$ a.s.         Then on the event  $B^\xi(0,e_1)> B^\xi(0,e_2)$  there almost surely exists  a rational $\zeta$ such that $\zeta\cdot e_1>\xi\cdot e_1$ and  $B^{\zeta\pm}(0,e_1)> B^{\zeta\pm}(0,e_2)$.  By \eqref{ci} this forces  
$\cid\cdot e_1\ge\zeta\cdot e_1>\xi\cdot e_1$.   Similarly on the event $B^\xi(0,e_1)< B^\xi(0,e_2)$ we   have almost surely $\cid\cdot e_1<\xi\cdot e_1$.  The conclusion is that 
 $\P(\cid=\xi)=0$ and part  \eqref{th:two-geo:iii} is proved. 
 

In part \eqref{th:two-geo:iv}, 
  $\Uset_{\zeta+}\ne\Uset_{\eta-}$ and directedness   (Theorem \ref{th:B-geod:direction}) force  the   cocycle geodesics $x_{0,\infty}^{0,\eta-}$ and $x_{0,\infty}^{0,\zeta+}$   to separate.   If $n\ge 0$ is the time after which they separate, then by
    cocycle geodesics ordering \eqref{eq:geod-order}   there exists   $z\in\Z^2$ such that $x_n^{0,\eta-}=x_n^{0,\zeta+}=z$, $x_{n+1}^{0,\eta-}=z+e_1$, and $x_{n+1}^{0,\zeta+}=z+e_2$. 
  Definition \eqref{ci} implies that $\zetamin\cdot e_1\le\cid(T_z\what)\cdot e_1\le\etamax\cdot e_1$.  For suppose  $\etamax\cdot e_1< \cid(T_z\what)\cdot e_1$.   Pick a rational point strictly between $\etamax$ and $\cid(T_z\what)$.   The second line of \eqref{ci} and  ordering \eqref{eq:monotone} imply that $B^{\eta-}(\what, z,z+e_1)> B^{\eta-}(\what, z,z+e_2)$, contradicting the choice $x_{n+1}^{0,\eta-}=z+e_1$.  
\end{proof}

\begin{corollary}\label{cor:ci-support}
Assume $\P\{\w_0\le r\}$ is continuous in $r$ and $\gpp$ is differentiable at the endpoints of all  its linear segments. Then $\cid$ lies almost surely outside the union of the closed  linear segments of $\gpp$.  Equivalently, $\cid$ is almost surely an exposed point.  
\end{corollary}

\begin{proof}
By Theorem \ref{thm-1:geod}\eqref{thm-1:geod:cont}  each linear segment has a unique geodesic from $0$  directed into it.  Since there are at most countably many linear segments of $\gpp$, Theorem \ref{th:two-geo}\eqref{th:two-geo:i}   contradicts  
  $\cid$ lying on a flat segment.    Under the differentiability assumption endpoints of flat segments are not exposed.  
\end{proof}

The next theorem identifies   the asymptotic direction of the competition interface $\varphi=(\varphi_k)_{0\le k<\infty}$ defined in Section \ref{sec:ci}.
 
  \begin{theorem}
Assume $\P\{\w_0\le r\}$ is continuous in $r$.
\begin{enumerate}[\ \ \rm(i)]
\item All limit points of the asymptotic velocity of the competition interface are in 
$\Uset_{\cid(\what)}$: for $\Phat$-almost every $\what$
\begin{align}\label{ci-lln1}
\cidmin(\what)\cdot e_1\le\varliminf_{n\to\infty} n^{-1}\varphi_n(\w)\cdot e_1\le\varlimsup_{n\to\infty} n^{-1}\varphi_n(\w)\cdot e_1\le\cidmax(\what)\cdot e_1.
\end{align}
\item If $\gpp$ is differentiable at the endpoints of its linear segments then $\cid$ is $\kS$-measurable and gives the asymptotic direction of the competition interface: $\Pbig$-almost surely
\begin{align}\label{ci-lln3}
\lim_{n\to\infty} n^{-1}\varphi_n(\w)=\cid(\what).
\end{align}
\end{enumerate}
\end{theorem}

\begin{proof}
Fix $\t\in\{e_1,e_2\}^{\Z^2}$. 
By \eqref{ci}, if $\zeta\cdot e_1<\cid(\what)\cdot e_1<\eta\cdot e_1$, then 
$x_1^{0,\t,\zeta\pm}=e_2$
and $x_1^{0,\t,\eta\pm}=e_1$.
Since the path  $\varphi$ separates the geodesics that go through $e_1$ and $e_2$, it  has to stay  between 
$x_{0,\infty}^{0,\t,\zeta+}$ and $x_{0,\infty}^{0,\t,\eta-}$. By Theorem \ref{th:B-geod:direction} these geodesics are $\Uset_{\zeta+}$ and $\Uset_{\eta-}$ directed,
and we have
\[\underline\zeta\cdot e_1\le\varliminf_{n\to\infty} n^{-1}\varphi_n\cdot e_1\le\varlimsup_{n\to\infty} n^{-1}\varphi_n\cdot e_1\le\overline\eta\cdot e_1.\]
Claim \eqref{ci-lln1} follows by taking $\zeta$ and $\eta$ to $\cid$.

If $\gpp$ is differentiable at the endpoints of its linear segments, then cocycles are $\kS$-measurable and hence so is $\cid$. Furthermore, $\cid$ is an exposed point by Corollary \ref{cor:ci-support}. In this case, $\cidmin=\cidmax$ and 
claim \eqref{ci-lln3} is proved.
\end{proof}

\begin{proof}[Proof of Theorem \ref{thm:ci-1}]  Limit \eqref{ci-lln2} is in \eqref{ci-lln3}.  The fact that the limit lies in $\ri\Uset$ is in the construction in the paragraph that contains \eqref{ci}, and the properties  of the limit    are in Theorem  \ref{th:two-geo}\eqref{th:two-geo:iv} 
and Corollary \ref{cor:ci-support}.   
\end{proof}

\begin{proof}[Proof of Theorem \ref{thm:ci-2}]    Part \eqref{thm:ci-2:i} comes directly from Theorem \ref{th:two-geo}\eqref{th:two-geo:i}.   For part \eqref{thm:ci-2:ii} assume  $\gpp$  strictly concave.  Then   $\Uset_\xi=\{\xi\}$ and   by Theorem \ref{thm-1:geod}\eqref{thm-1:geod:direct} every geodesic is directed. In this case, Theorem \ref{th:two-geo}\eqref{th:two-geo:ii}  implies  that with $\P$-probability one, there cannot be two distinct geodesics from $0$ with a common direction other than   $\cid$.  
\end{proof}


As mentioned at the end of Section \ref{sec:ci}, if   $\P\{\w_0\le r\}$ is not continuous in $r$, we have competition interfaces $\varphi^{(l)}$ and $\varphi^{(r)}$ for the trees of leftmost and rightmost geodesics.  Their limiting directions $\cidr(\what),\,\cidl(\what)\in\ri\Uset$ are defined by 
 \begin{align}\label{ci12}
\begin{split}
&B^{\zeta\pm}(\what, 0,e_1)> B^{\zeta\pm}(\what,0,e_2)\quad\text{if }\zeta\cdot e_1<\cidr(\what)\cdot e_1,\\
&B^{\zeta\pm}(\what,0,e_1)=B^{\zeta\pm}(\what,0,e_2)\quad\text{if }\cidr(\what)\cdot e_1<\zeta\cdot e_1<\cidl(\what)\cdot e_1\\
\text{and}\qquad&B^{\zeta\pm}(\what,0,e_1)<B^{\zeta\pm}(\what,0,e_2)\quad\text{if }\zeta\cdot e_1>\cidl(\what)\cdot e_1.
\end{split}
\end{align}
With this definition    limit  \eqref{ci-lln1}  is valid also with superscripts $(l)$ and $(r)$.    Consequently    $n^{-1}\varphi^{(a)}_n(\w)\to \cid^{(a)}(\what)$ a.s.\ for $a\in\{l,r\}$ under the assumption that $\gpp$ is strictly concave.  

\section{Exactly solvable models}\label{sec:solv}
 
We go through some details for the exactly solvable models  discussed in Section \ref{subsec:solv}.   
 
\subsection{Geometric weights} 
The weights $\{\w_x\}$  are  i.i.d.\   with $\P(\w_x=k)=(1- {\Ew}^{-1})^{k-1} {\Ew}^{-1}$ for $k\in\N$, mean $\Ew=\E(\w_0)>1$ and variance $\sigma^2=\Ew(\Ew-1)$.    



Begin by investigating the queueing fixed point.   With $\{\S_{n,0}\}$ i.i.d.\ geometric with mean $\Ew$, let the initial arrival process $\{A_{n,0}\}$ be  i.i.d.\ geometric with mean $\alpha$.  Let $J_n=S_{n,0}+W_{n,0}$.   Then equations \eqref{q:lin}    and  \eqref{def:A} show that   the process $\{(A_{n,1}, J_{n+1}): n\in\Z\}$ is an irreducible  aperiodic  Markov chain with transition probability 
\be\label{g:trans}  \begin{aligned}
&P( A_{n,1}=b, J_{n+1}=j\,\vert\, A_{n-1,1}=a, J_{n}=i)\\
&\qquad\qquad 
=  P\bigl\{  (A_{n,0}-i)^+ +S_{n+1,0}=b,  \,  (i-A_{n,0})^+ +S_{n+1,0}= j\bigr\}.  
 \end{aligned}\ee
Note that the equations also show that $(A_{n,0}, S_{n+1,0})$ are independent of $(A_{n-1,1}, J_{n})$.    Since  the process $\{(A_{n,1}, J_{n+1}): n\in\Z\}$ is  stationary, its marginal must be the unique invariant distribution of transition \eqref{g:trans}, namely 
\[   P(A_{n-1,1}=k, J_{n}=j) =   (1-{\alpha}^{-1})^{k-1}  {\alpha}^{-1} \cdot   (1- {\f(\alpha)}^{-1})^{j-1}  {\f(\alpha)}^{-1} \qquad \text{for $k,j\in\N$.}  
\]
with $\f(\alpha) = \Ew\frac{\alpha-1}{\alpha-\Ew} $. 
This shows that i.i.d.\ mean $\alpha$ geometric is a queuing  fixed point. 
Next solve for $\gppa(s)=\inf_{\alpha>\Ew}\{\alpha s+\f(\alpha)\}$.  The unique minimizing $\alpha$ in terms of $s=\xi\cdot e_1/\xi\cdot e_2$ is 
\[\alpha = \Ew+\sigma  \sqrt{\xi\cdot e_2/\xi\cdot e_1}\]
  which defines the  bijection between  $\xi\in\ri\Uset$  and $\alpha\in(\Ew,\infty)$.  From this  
\[\f(\alpha) = \Ew\frac{\alpha-1}{\alpha-\Ew}  = \Ew+\sigma \sqrt{\xi\cdot e_1/\xi\cdot e_2}.  
 \]

The terms in the sum $J_n=S_{n,0}+W_{n,0}$ are independent, so we can  also find the distribution of the waiting time: 
\[  P(W_{n,0}=0)= \frac{\alpha-\Ew}{\alpha-1}, \quad 
P(W_{n,0}=k)= \frac{\Ew-1}{\alpha-1} \cdot   \Bigl(1-\frac1{\f(\alpha)}\,\Bigr)^{k-1} \frac1{\f(\alpha)} \quad (k\ge 1). 
\]

The distributions of $\cidr$ and $\cidl$ claimed in \eqref{geom:theta} come from \eqref{ci12}, knowing that  $B^{(a,1-a)}(0,e_1)$ and $B^{(a,1-a)}(0,e_2)$ are independent geometrics with means  \eqref{geom:B}.   The calculation  for $\cidr$ goes   
\begin{align*} \P\{\cidr\cdot e_1>a\} &= \P\{ B^{(a,1-a)}(0,e_1)>B^{(a,1-a)}(0,e_2) \} 
= \frac{\alpha-\Ew}{\alpha}\\
&= \frac{\sqrt{(\Ew-1)(1-a)}}{\sqrt{\Ew a}+ \sqrt{(\Ew-1)(1-a)}} 
\end{align*}
from which the first formula of \eqref{geom:theta} follows.   Similar computation for $\cidl$.

\subsection{Exponential weights}\label{subsec:exponential}
The weights $\{\w_x\}$  are  i.i.d.\ exponential with mean $\Ew=\E(\w_x)>0$ and variance $\sigma^2=\Ew^2$, with marginal distribution 
\[  \P(\w_x> t)= \Ew^{-1}e^{-t/\Ew}\qquad\text{for $t\ge0$.}  \]
The queuing fixed point can be derived as in the geometric case.  
 The distribution of   $\cid$ comes from knowing that 
  $B^{(a,1-a)}(0,e_1)$ and $B^{(a,1-a)}(0,e_2)$ are independent exponentials   with parameters 
$\sqrt{a}/(\sqrt{a}+\sqrt{1-a})$ and $\sqrt{1-a}/(\sqrt{a}+\sqrt{1-a})$.
Hence   
	\[\P\{\cid\cdot e_1>a\} = \P\{ B^{(a,1-a)}(0,e_1)>B^{(a,1-a)}(0,e_2) \}= \frac{\sqrt{1-a}}{\sqrt{a}+\sqrt{1-a}}\,.\]
Equation \eqref{eq:ci-exp} follows.

\appendix

\section{Cocycles from queuing fixed points}
\label{app:q} 

This section proves Theorem \ref{th:cocycles}.  By shifting the variables $\{\w_x, B^{\xi\pm}(x,x+e_i)\}$    in Theorem \ref{th:cocycles} if necessary, we can assume without loss of generality that $\P\{\w_0\ge0\}=1$.   Then the weights $\w_x$ can represent service times and  we can tap into queueing theory.    We switch now to terminology and notation from queuing theory  to enable the reader to relate this appendix to the existing queueing literature.

	Consider an infinite sequence of  $\;\cdot/\text{G}/1/\infty/\text{FIFO}$   queues in tandem.  That is, each  queue or service station (these terms are used almost interchangeably)   has a  general service time distribution (the law of $\w_x$ under $\P$), a single server,   unbounded room for customers waiting to be served,  and customers obey first-in-first-out discipline.  The   service stations are  
	indexed by $k\in\Z_+$ and a bi-infinite sequence  of customers is indexed by $n\in\Z$.  Customers enter the system at station 0 and move from station to station in order.   The server at station $k$ serves one customer at a time.  Once the service of customer  $n$ is complete at station $k$, customer $n$ moves to the  back of the queue at station $k+1$ and customer $n+1$ enters service at station $k$ if he was already waiting in the queue.   If the queue at station $k$ is empty after the departure of customer $n$, then server $k$ remains idle until customer $n+1$ arrives.   Each customer retains his integer label as he moves through the system.  
	
Here is the   mathematical apparatus.    The system needs two ingredients: an initial inter-arrival process $\{\A_{n,0}:n\in\Z\}$ and the service times $\{\S_{n,k}:n\in\Z, k\in\Z_+\}$.  $A_{n,0}\ge 0$ is the time between the arrival of customer $n$ and customer $n+1$ at queue $0$.  $S_{n,k}\ge 0$ is the amount  of time the service of customer $n$ takes at station $k$.  Let  $\{\S_{n,k}:n\in\Z, k\in\Z_+\}$ be i.i.d.\ such that  $\S_{0,0}$ has the   distribution of  $\w_0$ under $\P$.    Assume $\{\A_{n,0}:n\in\Z\}$ is stationary, ergodic, and independent of 
	$\{\S_{n,k}:k\in\Z_+,n\in\Z\}$.    Assume 
\be\label{EA>ES}	 E[\S_{0,0}]=\Ew<E[\A_{0,0}]<\infty. \ee
This guarantees in particular a  {\sl stable} system where queues do not blow up.  	
	   The service time distribution is taken to be fixed,  while the input $\{A_{n,0}\}$ varies, by analogy with varying the initial distribution of a Markov process.  

 As a product of an ergodic process and an i.i.d.\ process  
$(\A_{n,0}, \S_{n,k})_{n\in\Z, k\in\Z_+}$  is   stationary and ergodic under   translations of the $n$-index.   	
Consequently  the entire queuing system is  stationary and ergodic under translations of the $n$-index.   The  issue of interest is finding  input processes $\{A_{n,0}\}$ such that the system is also stationary   under translations of the $k$-index.  Such a process or its distribution on $\R^\Z$ will be called a {\sl fixed point}  of the queueing operator.  

  Next we develop the iterative equations that describe the evolution of the system from station to station, as $k$ increases.  
 These are the variables.    $\A_{n,k}$ is  the inter-arrival time between customers $n$ and $n+1$ at queue $k$, or, equivalently, the inter-departure time between customers $n$ and $n+1$ from  queue $k-1$.   $W_{n,k}$ is the waiting time of customer $n$ at queue $k$, that is, the time between the arrival of customer $n$ at queue $k$ and the beginning of his service at queue $k$.  
 The total time customer $n$ spends at station $k$ is  
   	the sojourn time  $\W_{n,k}+\S_{n,k}$.  

	
The development begins with the waiting times.  Define  the stationary, ergodic process $\{W_{n,0}\}_{n\in\Z}$   by  	
 		\begin{align}
		\label{def:W}
   			\W_{n,0} = \Big( \sup_{j\le n-1} \sum_{i=j}^{n-1} ( \S_{i,0} - \A_{i,0} ) \Big)^+ . 
   		\end{align}
By the ergodic theorem and \eqref{EA>ES}  			
		\begin{align*}
			\W_{n,0}<\infty\text{ for all }n\in\Z.
		\end{align*}
 Process $\{W_{n,0}\}$ satisfies  Lindley's equation: 
	\begin{align}
   			&\W_{n+1,0} 
			=(\W_{n,0}+\S_{n,0}-\A_{n,0})^+.\label{q:lin}
  		 \end{align}
This equation agrees naturally with  the queuing interpretation.    If  $W_{n,0}+S_{n,0}< A_{n,0}$ then customer $n$ leaves  station $0$ before customer $n+1$ arrives, and consequently customer $n+1$ has no wait  and $W_{n+1,0}=0$.  In the complementary case customer $n+1$   waits time $W_{n+1,0}=W_{n,0}+S_{n,0}-A_{n,0}$ before entering service at station $0$.


  	With some additional work we prove the following. 	
		
	\begin{lemma}  $n^{-1}W_{n,0}\to 0$ almost surely  as $n\to\infty$.  
	\label{t:lm-W_n} \end{lemma}

	\begin{proof}   Abbreviate $\U_n=\S_{n,0}-\A_{n,0}$.   For $a\ge 0$ and $\e>0$  define 
	\begin{align*}
 	W^\e_0(a)&=a\\
 	W^\e_{n+1}(a)  &= \bigl(   W^\e_n(a)  + U_n -E(U_0) +\e \bigr)^+\qquad\text{for $n\ge 0$}.
	\end{align*}
	Check inductively  that 
	\[    W^\e_n(0)  = \Bigl( \;\max_{0\le m<n}  \sum_{k=m}^{n-1}  [U_k -E(U_0) +\e] \Bigr)^+ .\]
	Consequently 
	\[   W^\e_n(a)\ge  W^\e_n(0)  \ge  \sum_{k=0}^{n-1}  [U_k -E(U_0) +\e]
	\longrightarrow \infty   \quad\text{ as $n\to\infty$.} \]  
	Thus $W^\e_n(a)>0$ for large $n$ which implies, from its definition, that 
	for large $n$ 
	\[    W^\e_{n+1}(a)  =    W^\e_n(a)  + U_n -E(U_0) +\e.  \] 
	Another application of the ergodic theorem gives   $n^{-1}   W^\e_n(a)  \to \e$  $\P$-a.s.\ as $n\to\infty$.  

	Now for the conclusion.   Since $W_{0,0}=W^\e_0(W_{0,0}) $, we can check inductively that 
	\begin{align*}   W_{n+1,0}&=(\W_{n,0}+U_n)^+\le \bigl(W^\e_n(\W_{0,0})+U_n\bigr)^+ \\
	&\le  \bigl(W^\e_n(\W_{0,0})+U_n-E(U_0) +\e \bigr)^+=W^\e_{n+1}(\W_{0,0}). 
	\end{align*}
	From this,  $0\le n^{-1}W_{n,0}\le n^{-1}   W^\e_n(W_{0,0}) \to\e$, and we let $\e\searrow0$. 
	\end{proof} 			
		
	The  stationary and ergodic process  $\{\A_{n,1}:n\in\Z\}$ of inter-departure times from queue 0 (equivalently, inter-arrival times at queue $1$)   is  defined by 
		\begin{align}
			\A_{n,1} = (\A_{n,0} - \S_{n,0} - \W_{n,0} )^+ + \S_{n+1,0}, 
		\label{def:A}
		\end{align}
again by considering the two cases: either customer $n+1$ arrives before customer $n$ departs ($\A_{n,0} < \S_{n,0} + \W_{n,0}$) and goes   into service the moment customer $n$ departs, or   server $0$  is empty waiting for customer $n+1$ for time $(\A_{n,0} - \S_{n,0} - \W_{n,0} )^+$ before service of customer $n+1$ begins.  
	Process $\{\A_{n,1}:n\in\Z\}$ is  
	independent of $\{\S_{n,k}:k\ge 1,n\in\Z\}$. 

Combining  equations \eqref{q:lin} and \eqref{def:A}  	 and iterating gives 
		\[
			\W_{1,0}+\S_{1,0}+\sum_{i=1}^n\A_{i,1}=\W_{n+1,0}+\S_{n+1,0}+\sum_{i=1}^n\A_{i,0} \qquad\text{for $n\ge 1$.} 
		\]
This and  Lemma \ref{t:lm-W_n} imply  $E[\A_{0,1}]=E[\A_{0,0}]$.
 	(In the queueing literature, this has been observed in \cite{Loy-62}.) 
	
These steps are repeated at each queue.  At queue $k$ we have the stationary, ergodic  arrival process  $\{\A_{n,k}\}_{n\in\Z}$  that is  independent of the service times  $\{\S_{n,j}: n\in\Z, j\ge k\}$.  Waiting times at queue $k$ are defined by 
\begin{align}
		\label{def:Wk}
   			\W_{n,k} = \Big( \sup_{j\le n-1} \sum_{i=j}^{n-1} ( \S_{i,k} - \A_{i,k} ) \Big)^+ . 
   		\end{align}	
Properties  $\W_{n,k}<\infty$, Lemma \ref{t:lm-W_n},   and $E[\A_{n,k}]=E[\A_{0,0}]$ are preserved along the way. 		
This procedure constructs the  process $\{ \A_{n,k}, \S_{n,k}, \W_{n,k}:  n\in\Z, k\in\Z_+\}$  that satisfies the   following system of equations:   
		\begin{align}
			\begin{split}\label{eq:ne-induction:q}
 				&\W_{n+1,k}+\S_{n+1,k}=\S_{n+1,k}+(\W_{n,k}+\S_{n,k}-\A_{n,k})^+,\\
				&\A_{n,k+1} = ( \A_{n,k} - \S_{n,k} - \W_{n,k} )^+ + \S_{n+1,k},\\
				&\S_{n+1,k}=(\S_{n+1,k}+\W_{n+1,k})\wedge \A_{n,k+1}.
			\end{split}
		\end{align} 
	The third equation follows directly from the first two.  
	A useful consequence of \eqref{eq:ne-induction:q} is the ``conservation law''
		\begin{align}
		\label{eq:closed-loop:q}
			\W_{n+1,k}+\S_{n+1,k}+A_{n,k}=\W_{n,k}+\S_{n,k}+\A_{n,k+1}.
		\end{align}

%

The next  four statements  
summarize the situation with fixed points, quoted  from articles \cite{Mai-Pra-03, Pra-03}. 
	Given a stationary ergodic probability measure $\mu$ on $\R^\Z$ consider random variables 
		\[
			\{\A_{n,0},\,\S_{n,0},\,\W_{n,0},\,\A_{n,1}:n\in\Z\}
		\] 
	where $\{\A_{n,0}:n\in\Z\}$ has distribution $\mu$, $\{\S_{n,0}:n\in\Z\}$ 
	are i.i.d.\ with distribution $\P$, 
	the two collections are independent of each  other, $\W_{n,0}$ are defined via \eqref{def:W}, 
	and $\A_{n,1}$ are defined via \eqref{def:A}.
	Let $\Phi(\mu)$ denote the distribution of $\{\A_{n,1}:n\in\Z\}$.  $\Phi$ is the {\sl queueing operator} whose fixed points are the focus now.  

	Let  $\M_e^\alpha(\R^\Z)$ be the space of translation-ergodic probability measures $\mu$ on $\R^\Z$ with marginal mean  $E^\mu[\A_{0,0}]=\alpha$.   We are mainly interested in ergodic fixed points, so we define
		\begin{align*}
			\cA=\{\alpha>\Ew:\exists\mu\in\M_e^\alpha(\R^\Z)\text{ such that } \Phi(\mu)=\mu\}.
		\end{align*}

         		\begin{theorem} \label{th:Pra-03}  \text{\cite[Thm.~1]{Pra-03}}
			Let $\alpha\in\cA$. Then there exists a unique $\mu^\alpha\in\M_e^\alpha(\R^\Z)$ with $\Phi(\mu^\alpha)=\mu^\alpha$.
			Furthermore, let $\A^0=\{\A_{n,0}:n\in\Z\}$ be ergodic with mean 
			$E[\A_{0,0}]=\alpha$ and $\{\S_{n,k}:n\in\Z,k\in\Z_+\}$  
			i.i.d.\ with distribution $\P$ and independent of   $\A^0$.
			Let $A^k=\{\A_{n,k}:n\in\Z\}$, $k\in\N$,  be defined via inductions 
			\eqref{def:W} and \eqref{eq:ne-induction:q}.
			Then the distributions of $A^k$ converge weakly to $\mu^\alpha$.  
		\end{theorem}
 
		\begin{theorem}   \text{\cite[Thm.~5.1 and 6.4 and Lm.\ 6.3(a)]{Mai-Pra-03}} 
			The set $\cA$ is closed and nonempty, $\inf\cA=\Ew$, and $\sup\cA=\infty$. If $\alpha<\beta$ are both in $\cA$ then $\mu^\alpha\le\mu^\beta$ in the usual sense of  
	  stochastic ordering.  
		\end{theorem}
		
				\begin{lemma} \text{\cite[Lm.\ 6.3(b)]{Mai-Pra-03}}  
		\label{lm:EJ=f}
			Let $\alpha\in\cA$,  $A^0\sim\mu^\alpha$, and $\{S_{n,k}\}\sim\P$ independent of $A^0$.  
			Define $\W_{n,0}$ via \eqref{def:W}. 
			Then 
\be\label{q:500} E^{\mu^\alpha\otimes\P}[\W_{0,0}+\S_{0,0}]=\f(\alpha).\ee
		\end{lemma}
          
Suppose $\alpha\in(\Ew,\infty)\cap\cA^c$.  Let 
		\[
			\amin=\sup\bigl( \cA\cap(\Ew,\alpha]\bigr) \in\cA\quad\text{and}\quad\amax=\inf\bigl(\cA\cap[\alpha,\infty)\bigr)\in\cA, 
		\]
$t=(\amax-\alpha)/(\amax-\amin)$ and   $\mu^\alpha=t\mu^{\amin}+(1-t)\mu^{\amax}$. 
	Now $\mu^{\alpha}$ is a mean $\alpha$ fixed point of $\Phi$.   This fixed point is again attractive, in the following sense. 
 
		\begin{theorem} \text{\cite[Prop.\ 6.5]{Mai-Pra-03}} 
		\label{th:attractive}
			Let $\alpha>\Ew$.
			Let $\{\A_{n,0}:n\in\Z\}$ be ergodic with mean 
			$E[\A_{0,0}]=\alpha$ and $\{\S_{n,k}:n\in\Z,k\in\Z_+\}$  
			i.i.d.\ with distribution $\P$ and independent of the $\A$-process.
			Let $\{\A_{n,k}:n\in\Z,k\in\N\}$ be defined via inductions 
			\eqref{def:W} and \eqref{eq:ne-induction:q}.
			The Ces\`aro mean of the distributions of $\{\A_{n,k}:n\in\Z\}$ converges weakly to $\mu^\alpha$.
		\end{theorem}

\begin{lemma}\label{q:lm5} 
\begin{enumerate}[\ \ \rm(a)]
\item\label{q:lm5:a}  Let $\amin< \amax$ be points in $\cA$ such that  $(\amin, \amax)\subset\cA^c$.  Then $\f$ is linear on the interval  $[\amin, \amax]$.  
\item\label{q:lm5:b}  Let $\xi\in\Diff$, $s=\xi\cdot e_1/\xi\cdot e_2$ and $\alpha=\gppa'(s)$.  Then $\alpha\in\cA$.  
\end{enumerate}
\end{lemma}

\begin{proof}  Part \eqref{q:lm5:a}.   Let $0<t<1$ and  $\alpha= t\amin+(1-t)\amax$.  In the notation of \cite{Mai-Pra-03}, consider a sequence of tandem queues
$(\A^k, \S^k, \W^k, \A^{k+1})_{k\in\Z_+}$ where the initial arrival process  $\A^0=(A_{n,0})_{n\in\Z}$ is ergodic with mean $E(A_{n,0})=\alpha$,  the service times   $\{S^k\}_{k\in\Z_+}=\{\S_{n,k}\}_{n\in\Z, k\in\Z_+}$ are independent of $\A^0$ and i.i.d.\ $\P$-distributed, and the remaining variables are defined iteratively.    Let $(\wh\A, \wh\S, \wt\W, \wt\D)$ denote a weak limit point of the Ces\`aro averages of the distributions of $(\A^k, \S^k, \W^k, \A^{k+1})$.   Then, as shown in \cite[eqn.~(29)]{Mai-Pra-03} in the course of the proof of their Theorem 5.1,   $\wt\W=\Psi(\wh\A, \wh\S)$ where the mapping  $\Psi$ encodes definition \eqref{def:W}.   By Theorem \ref{th:attractive}   \cite[Prop.~6.5]{Mai-Pra-03}    the distribution of $\wh\A$ is  $t\mu^{\amin}+(1-t)\mu^{\amax}$.   By \cite[Theorem 4.1]{Mai-Pra-03}, 
\be\label{MP-lln5} n^{-1}\sum_{k=0}^{n-1} \W_{0,k} \to M(\alpha)\equiv  \f(\alpha)-\Ew
\qquad\text{almost surely.}\ee
  Combine these facts as follows.   First  
\begin{align*}
E(\wt\W_0)&=E[\Psi(\wh\A, \wh\S)_0]= t E^{\mu^{\amin}\otimes\P}[\Psi(\wh\A, \wh\S)_0] + (1-t) E^{\mu^{\amax}\otimes\P}[\Psi(\wh\A, \wh\S)_0]\\
  &=   tM(\amin)+(1-t)M(\amax)
\end{align*}
where the last equality comes from \cite[Lemma 6.3(b)]{Mai-Pra-03} restated as Lemma \ref{lm:EJ=f} above.    The weak limit, combined with the law of large numbers \eqref{MP-lln5}
and dominated convergence, gives, for any $c<\infty$ and along a subsequence, 
\begin{align*}
E(\wt\W_0\wedge c)&= \lim_{n\to\infty}  n^{-1}\sum_{k=0}^{n-1} E(\W_{0,k}\wedge c) \le  \lim_{n\to\infty}  E\Bigl[\,\Bigl(  n^{-1}\sum_{k=0}^{n-1} \W_{0,k}\Bigr) \wedge c\Bigr] = M(\alpha)\wedge c\\
&\le M(\alpha).  
\end{align*}  
Letting $c\nearrow\infty$ gives 
\[    tM(\amin)+(1-t)M(\amax) \le  M(\alpha).   \]
Since $M$ is convex and $\f$ differs from $M$ by a constant, this implies 	
$\f(\alpha)=  t\f(\amin)+(1-t)\f(\amax)$ and completes the proof of part \eqref{q:lm5:a}. 	

\smallskip

Part \eqref{q:lm5:b}.    To get a contradiction, suppose $\alpha\in\cA^c$.  Then there exist $\amin< \amax$   in $\cA$ such that  $\alpha\in (\amin, \amax)\subset\cA^c$.    By part \eqref{q:lm5:a}  $f$ is linear on  $[\amin, \amax]$.  Basic convex analysis   implies that $\gppa$ has multiple tangent slopes at $s$ and hence cannot be differentiable  at $s$. Here is the argument.  


By \eqref{eq:f=g-sa} the assumption $\gppa'(s)=\alpha$ implies that $\gppa(s)=\f(\alpha)+\alpha s$.   It follows that $s$ must be the slope of $\f$ on $(\amin, \amax)$.  For suppose this slope is $s_1$ and let $\alpha_1\in(\amin, \amax)$.  Then  by the duality  \eqref{eq:ga=f*} 
\begin{align*}
\gppa(s)&\le \f(\alpha_1)+\alpha_1s  = \f(\alpha)+s_1(\alpha_1-\alpha) +\alpha_1 s \\
&=\f(\alpha)+\alpha s + (\alpha_1-\alpha)(s_1-s) 
\end{align*}
which  contradicts $\gppa(s)=\f(\alpha)+\alpha s$  unless $s_1=s$ because we can make $\alpha_1-\alpha$ both positive and negative.  

Since $s$ is the slope of $\f$ on $(\amin, \amax)$,  we have 
$   \f(\alpha)+\alpha s = \f(\alpha_1)+\alpha_1 s$ 
\text{for all  $\alpha_1\in[\amin, \amax]$.}   
Hence  for any $t\ne s$ and any  $\alpha_1\in[\amin, \amax]$
\begin{align*}  
\gppa(t)-\gppa(s) &\le  \bigl(  \f(\alpha_1)+\alpha_1t\bigr)  -\bigl(  \f(\alpha_1)+\alpha_1s\bigr) = \alpha_1(t-s) 
\end{align*}
which contradicts $\gppa'(s)=\alpha$ because we can choose $\alpha_1$ smaller and larger than $\alpha$.  
\end{proof}

To prepare for the proof of Theorem \ref{th:cocycles}, 
	fix $\alpha>\Ew$. Let $\{\A_{n,0}:n\in\Z\}$ 
	have the $\Phi$-invariant distribution $\mu^\alpha$, let $\{\S_{n,k}:n\in\Z,k\in\Z_+\}$ 
	be i.i.d.\ with distribution $\P$, let the two collections be independent, 
	and define $\{\W_{n,k},\A_{n,k+1}:n\in\Z,k\in\Z_+\}$ using \eqref{def:W} 
	and \eqref{eq:ne-induction:q}.
	Because $\Phi(\mu^\alpha)=\mu^\alpha$, process 
	$\{\A_{n,k},\S_{n,k},\W_{n,k}:n\in\Z, k\in\Z_+\}$ is stationary, both in $n$ and in $k$.
	This  allows us to extend the process to entire lattice $\Z^2$ and thereby define the $\Z^2$-indexed stationary process $(\A, \S, \W)=\{\A_{n,k},\S_{n,k},\W_{n,k}:n,k\in\Z\}$.   
Define also another 	$\Z^2$-indexed stationary process $(\Ap, \Sp, \Wp)$ by  
\[ ( \Ap_{i,j},\Sp_{i,j}, \Wp_{i,j}) = 
(\W_{j-1,i+1}+\S_{j-1,i+1} , \, \S_{j,i}, \, \A_{j-1,i+1}-\S_{j,i} ).  \]

\begin{lemma}\label{q:erg-lm} 
Suppose  $\alpha\in\cA$.  Then     the process $(\A, \S, \W)$ is ergodic   under   translation $T_{e_1}$,  and also ergodic under  $T_{e_2}$.   Furthermore,  $\f(\alpha)\in\cA$.   $(\Ap, \Sp, \Wp)$ is a stationary queueing system where $\{\Ap_{n,0}:n\in\Z\}$ has distribution  $\mu^{\f(\alpha)}$, and is also ergodic under both $T_{e_1}$  and    $T_{e_2}$. 
\end{lemma} 

\begin{proof}
 Ergodicity under $T_{e_1}$ 
 follows from the construction. Process
$(\A_{n,0}, \S_{n,k})_{n\in\Z, k\in\Z_+}$  is ergodic under $T_{e_1}$, as a product of an ergodic process and an i.i.d.\ process.   
  The equations developed above define $(\A_{0,k},\S_{0,k},\W_{0,k})_{k\in\Z_+}$ as  a function   of the process 
$(A_{n,0}, S_{n,k})_{n\in\Z, k\in\Z_+}$, and for each $m\in\Z$,     $(\A_{m,k},\S_{m,k},\W_{m,k})_{k\in\Z_+}$  is obtained by applying the same  function to the  $T_{me_1}$-shift  of the process 
$(A_{n,0}, S_{n,k})_{n\in\Z, k\in\Z_+}$.   Thus $(\A_{n,k},\S_{n,k},\W_{n,k})_{n\in\Z, k\in\Z_+}$ is ergodic   under    $T_{e_1}$.  
   
   The same argument gives  $T_{e_1}$-ergodicity of $(\A_{n,k},\S_{n,k},\W_{n,k})_{n\in\Z, k\ge\ell}$ for any $\ell\in\Z$.   For the final step, let $B$ be a $T_{e_1}$-invariant event  of the full process $\{\A_{n,k},\S_{n,k},\W_{n,k}: n, k\in\Z\}$.    Write $\cG_\ell$ for the $\sigma$-algebra generated by $(\A_{n,k},\S_{n,k},\W_{n,k})_{n\in\Z, k\ge\ell}$.   The conditional expectations 
   $E(\one_B\vert\cG_\ell)$  are $T_{e_1}$-invariant, hence a.s.\ constant by the ergodicity proved thus far.   $E(\one_B\vert\cG_\ell)\to\one_B$  almost surely as $\ell\to-\infty$, and  consequently $\one_B$ is a.s.\ constant.   This completes the proof of ergodicity under $T_{e_1}$.   

To get ergodicity under $T_{e_2}$ we transpose, and that leads us to look at $( \Ap, \Sp,\Wp)$. 
To see that   $( \Ap, \Sp,\Wp)$ is another queueing system with the same i.i.d.\ service time distribution  $\Sp_{i,j}=\S_{j,i}$, we need to check three items.

(i)  Independence of $\{\Ap_{i,\ell}\}_{i\in\Z}$ and $\{\Sp_{i,j}\}_{i\in\Z, j\ge \ell}$, for each $\ell\in\Z$.  This follows from the structure of equations \eqref{eq:ne-induction:q} and the independence of the $\{S_{i,j}\}$.  

(ii)  $\Ap_{i, j+1}=(\Ap_{ij} -\Sp_{ij} -\Wp_{ij} )^+ +\Sp_{i+1, j} $.  This follows from the top equation of \eqref{eq:ne-induction:q}.  

(iii)  The third point needed is
 \be\label{q:700} \Wp_{k+1,j}=\Bigl( \;\sup_{n:\,n\le k}\, \sum_{i=n}^{k} (\Sp_{ij}-\Ap_{ij})  \Bigr)^+. \ee This needs a short argument.  Fix $k,j$.   The middle equation of \eqref{eq:ne-induction:q} gives 
\be\label{q:701}  \Wp_{ij}=(\Wp_{i-1,j} +\Sp_{i-1, j} -\Ap_{i-1, j} )^+  \ee
which  can be iterated to give  
\[  \Wp_{k+1,j}=\Bigl(  \Bigl\{\Wp_{\ell j} +  \sum_{i=\ell}^{k} (\Sp_{ij}-\Ap_{ij})  \Bigr\}  \;\vee\; 
\Bigl\{ \max_{n:\ell<n\le k}  \sum_{i=n}^{k} (\Sp_{ij}-\Ap_{ij})  \Bigr\} \,  \Bigr)^+  
\qquad \text{for $\ell\le k$.}  \]
Thus \eqref{q:700} follows if $\Wp_{\ell j}=0$ for some $\ell\le k$.  Suppose on the contrary that  $\Wp_{ij}>0$ for all $i\le k$.   Apply \eqref{q:701}  to all $\Wp_{ij}$ for $n<i\le k$ and divide by $\abs n$ to get   
\[   \frac{\Wp_{kj}}{\abs n}= \frac{\Wp_{nj}}{\abs n} +  \frac1{\abs n}\sum_{i=n}^{k-1} (\Sp_{ij}-\Ap_{ij})  \]
which is the same as 
 \be\label{q:705}  \frac{\A_{j-1,k+1}}{\abs n} -  \frac{\S_{jk}}{\abs n}  \; = \;  \frac{\A_{j-1,n+1}}{\abs n} -  \frac{\S_{jn}}{\abs n} +  \frac1{\abs n}\sum_{i=n}^{k-1} (\S_{ji}-\W_{j-1,i+1}-\S_{j-1,i+1}) .  \ee
Let $n\to-\infty$.   The i.i.d.\ property of the $\{\S_{ij}\}$ and  Theorem 4.1 of  \cite{Mai-Pra-03},  combined with \eqref{q:500} from above,  give  the a.s.\ limit 
 \be\label{q:707}   \lim_{n\to-\infty}  \frac1{\abs n}\sum_{i=n}^{k-1}  \Ap_{ij} =     \lim_{n\to-\infty}  \frac1{\abs n}\sum_{i=n}^{k-1} (\W_{j-1,i+1}+\S_{j-1,i+1}) =  f(\alpha). \ee
 The four leftmost  terms  of \eqref{q:705}  vanish as $n\to-\infty$ (by stationarity and finite expectations). 
Hence  letting $n\to-\infty$ in \eqref{q:705} leads to 
$0=\Ew-f(\alpha)<0$  (the last inequality from Lemma \ref{lm:f-properties}). 
This contradiction  verifies \eqref{q:700}. 

At this point we have shown  that  the stationary process $\{\Ap_{n,0}\}_{n\in\Z}$ is a fixed point for $\Phi$ with the deterministic  pathwise limit \eqref{q:707}.  By Prop.~4.4 of  \cite{Mai-Pra-03} the process $\{\Ap_{n,0}\}_{n\in\Z}$ must be ergodic. 
%
%
  We have shown that $f(\alpha)\in\cA$.   The   part of the lemma already  proved gives the ergodicity of
the process  
\[  \{ \Ap_{ij},\Sp_{ij}, \Wp_{ij}\} = 
\{ \W_{j-1,i+1}+\S_{j-1,i+1} , \, \S_{j,i}, \, \A_{j-1,i+1}-\S_{j,i} \}  \]
under translations of the index $i$.  Since ergodicity is preserved by mappings that respect translations,  a suitable mapping of the right-hand side  above gives the $T_{e_2}$-ergodicity of $\{ \A_{nk},\S_{nk}, \W_{nk}\}$. 
\end{proof} 

\begin{proof}[Proof of Theorem \ref{th:cocycles}]    
We begin by constructing a convenient 
\label{A_0} 
   countable subset $\cA_0$ of $\cA$.   Let $\Uset_0$ be a dense countable subset of $\ri\Uset$ such that $\Uset_0$ contains all (at most countably many)  points of nondifferentiability of $\gpp$  and  
    a dense countable subset of points of differentiability of $\gpp$.   Then put 
  $\cA_0=\{ \gppa'(s\pm):   (\tfrac{s}{1+s}, \tfrac{1}{1+s})\in\Uset_0\}$.    $\cA_0\subset\cA$ by virtue of Lemma \ref{q:lm5}\eqref{q:lm5:b} and the closedness of $\cA$.  
  

We construct a measure $\bar\mu$  on $\R^{\cA_0\times\Z}$ that couples together 
the distributions $\mu^\alpha$ for $\alpha\in\cA_0$ so  that the coordinates $\{\eta^{\alpha}_n\}_{ \alpha\in\cA_0, n\in\Z}$  satisfy 
$\{\eta^{\alpha}_n\}_{n\in\Z} \sim\mu^\alpha$  
 and 
$\eta^\alpha_{n}\le \eta^\beta_{n} $ $\bar\mu$-a.s.\ for $\alpha<\beta$  in $\cA_0$.   This measure $\bar\mu$ comes from a weak limit of a coupled system of queues.     For each $\alpha\in\cA_0$  let an  initial inter-arrival process be the deterministic constant process $\A^\alpha_{n,0}=\alpha$.  As before use the iterative equations to construct the  variables 
   $(\A^{\alpha, k}, \S^k, \W^{\alpha, k})= \{ \A^\alpha_{n,k}, \S_{n,k}, \W^\alpha_{n,k}:  n\in\Z\}$ for $k\in\Z_+$.  Each process uses the same version of the service times $\{S_{n,k}\}$. 
  According to Theorem \ref{th:Pra-03}  \cite[Thm.~1]{Pra-03},  each $\A^{\alpha, k}$ converges weakly to $\mu^\alpha$.  Let $\bar\mu$ be any weak limit point of the joint distributions of the systems  $\{\A^{\alpha, k}:  \alpha\in\cA_0\}$ as $k\to\infty$.      The inequalities  
 \be \label{q:ineq13}     \A^\alpha_{n,k}\le \A^\beta_{n,k}   \quad\text{and}\quad 
     \W^\alpha_{n,k}\ge \W^\beta_{n,k}  \quad\text{ for $\alpha<\beta$}   
 \ee 
 are true for the $A$-processes at $k=0$ by construction.  They are propagated for all $k$ by equations \eqref{def:Wk} and \eqref{eq:ne-induction:q}.    Consequently $\bar\mu$ has the desired properties.   

Next we  construct a joint  queueing process that couples together the stationary queuing processes for all $\alpha\in\cA_0$.  Let the inputs $(\{\A^{\alpha,0}:\alpha\in\cA_0\}, \{S^k:k\in\Z_+\})$ have distribution $\bar\mu\otimes\P$.  Construct again the variables   $\{\A^\alpha_{n,k},\S_{n,k},  \W^\alpha_{n,k}:  n\in\Z, k\in\Z_+,  \alpha\in\cA_0\}$ with the iterative equations.   Use the stationarity under translations of $k$ to extend the joint distribution to a process indexed by $\Z^2$, denoted by  $\{(\A^\alpha, \S,  \W^\alpha): \alpha\in\cA_0\}= 
\{ \A^\alpha_{n,k}, \S_{n,k}, \W^\alpha_{n,k}:  n,k\in\Z, \alpha\in\cA_0\}$.    
Then  for each $\alpha\in\cA_0$,  $(\A^\alpha, \S,  \W^\alpha)$ is as described in Lemma \ref{q:erg-lm}: stationary and ergodic under both translations, $\{ \A_{n,0}^\alpha: n\in\Z\}\sim\mu^\alpha$,  and  $\{\S_{0,k}+\W_{0,k}^\alpha: k\in\Z\}\sim\mu^{\f(\alpha)}$.     Furthermore, inequalities \eqref{q:ineq13}  continue to hold almost surely  in this coupling.  

Define the following mapping from the coordinates $\{(\A^\alpha, \S, \W^\alpha): \alpha\in\cA_0\}$ to the coordinates   $\{ (\w_x)_{x\in\Z^2},  (\w^{i,\alpha}_x)_{i\in\{1,2\}, \,\alpha\in\cA_0, \, x\in\Z^2}\}$ of the space 
 $\Ombig=\Omega\times \R^{\{1,2\}\times\cA_0\times\Z^2}$:   for $(n,k)\in\Z^2$ and  $\alpha\in\cA_0$, 
 \begin{align}\label{q:map} 
( \om_{n,k}, \, \w^{1,\alpha}_{n,k}, \, \w^{2,\alpha}_{n,k})
=  (\S_{-n,-k},\,\A^\alpha_{-n-1,-k+1},\,\W^\alpha_{-n,-k}+\S_{-n,-k})  . 
 \end{align}
Let $\Pbig$ be the distribution induced on $\Ombig$ by this mapping, from the joint distribution of the coupled stationary queueing processes.  

The probability space $\OBPbig$ of Theorem \ref{th:cocycles}  has now been constructed.  
For $\xi\in\Uset_0$ and $i=1,2$   define the functions   $B^{\xi\pm}(x,x+e_i)$ as the following coordinate projections: 
\be\label{q:B3}     B^{\xi\pm}(\what, x,x+e_i)= \w^{i,\gppa'(s\pm)}_x \qquad\text{for $s=\xi\cdot e_1/\xi\cdot e_2$.}   
\ee
The set $\cA_0$ was constructed to ensure  $\gppa'(s\pm)\in\cA_0$ for each $\xi\in\Uset_0$ so these functions  are well-defined.  

The remainder  of the proof consists of two steps: (a) verification that the processes 	$B^{\xi\pm}(x,x+e_i)$ defined in \eqref{q:B3}  for $\xi\in\Uset_0$ satisfy all the properties required by Theorem \ref{th:cocycles} and  (b) definition of  $B^{\xi\pm}(x,x+e_i)$ for {\sl all} $\xi\in\ri\Uset$  through monotone limits followed by  another verification of the required properties.    

\medskip

In part \eqref{th:cocycles:indep} of Theorem \ref{th:cocycles} the measurability claim comes from the construction.  The stationarity and ergodicity of each process 
$\varphi^{\xi\pm}_x(\what)=(\w_x, B^{\xi\pm}(x,x+e_1), B^{\xi\pm}(x,x+e_2))$  under both translations $T_{e_1}$ and $T_{e_2}$  are a consequence  of Lemma \ref{q:erg-lm}. 
The independence claim follows from the fact that in the   queuing construction the triple  $(\A^\alpha_{-n-1,-k+1}, \S_{-n,-k}, \W^\alpha_{-n,-k})$ is  
a function of $\{ \A^\alpha_{i,m}, \S_{i,j}: i\le -n,  m\le j\le -k\}$ for any $m<-k$. 

Part \eqref{th:cocycles:exist} of Theorem \ref{th:cocycles} requires the cocycle properties. 
  The conservation law \eqref{eq:closed-loop:q} of the queueing construction implies that, almost surely,   for all $\alpha\in\cA_0$
\[			\W^\alpha_{-n,-k}+\S_{-n,-k}+A^\alpha_{-n-1,-k}=\A^\alpha_{-n-1,-k+1}+\W^\alpha_{-n-1,-k}+\S_{-n-1,-k}.  \]
Via \eqref{q:map} and \eqref{q:B3} this translates into the  $\Pbig$-almost sure property 
\[	B^{\xi\pm}(x,x+e_2) + B^{\xi\pm}(x+e_2,x+e_1+e_2) =  B^{\xi\pm}(x,x+e_1) + B^{\xi\pm}(x+e_1,x+e_1+e_2)	  \]	
for $x=(n,k)$ and all $\xi\in\Uset_0$.  Thus each process $B^{\xi\pm}(x,x+e_i)$  extends to a  cocycle   $\{B^{\xi\pm}(x,y): x,y\in\Z^2\}$.   Shift-covariance (Definition \ref{def:cK}\eqref{def:cK:stat}) holds because $B^{\xi\pm}(\what,x,x+e_i)$ is a function of the coordinates $\w^{i,\alpha}_x$, $\alpha\in\cA_0$, of $\what$ at $x$, and integrability comes from the next calculation.  

 The tilt vectors satisfy 
\begin{align*}
h(\xi\pm) &=  -\, \bigl( \,\Ebig[B^{\xi\pm}(0,e_1)]\,, \, \Ebig[B^{\xi\pm}(0, e_2)]\, \bigr) 
=  - \bigl( E[A^{\gppa'(s\pm)}_{0,0}],  E[W^{\gppa'(s\pm)}_{0,0}+S_{0,0}]\,\bigr)\\
&= -\bigl(   \gppa'(s\pm) , \f(\gppa'(s\pm))\bigr)  = -\nabla \gpp(\xi\pm).  
\end{align*} 
The fact that one-sided gradients satisfy the duality \eqref{eq:duality} is basic convex analysis. 

 Via \eqref{q:map} and \eqref{q:B3} the bottom equation of \eqref{eq:ne-induction:q}  translates into the potential-recovery property 
\[\w_x= B^{\xi\pm}(x,x+e_1) \wedge  B^{\xi\pm}(x,x+e_2) \qquad \text{$\Pbig$-a.s.}   \]
Part \eqref{th:cocycles:exist} of Theorem \ref{th:cocycles} has been verified for $B^{\xi\pm}(\what, x,x+e_i)$ for $\xi\in\Uset_0$. 

Part \eqref{th:cocycles:flat} of Theorem \ref{th:cocycles} is the equality of cocycles that share the tilt vector.  This is clear from definition \eqref{q:B3}  because $h(\xi\pm)$ determines $\gppa'(s\pm)$.  

For the inequalities of part \eqref{th:cocycles:cont},  let  $s=\xi\cdot e_1/\xi\cdot e_2$  and $t=\zeta\cdot e_1/\zeta\cdot e_2$ for $\xi, \zeta\in\Uset_0$.   Then 
 $\xi\cdot e_1<\zeta\cdot e_1$ implies $s<t$. By concavity $\gppa'(s-)\ge \gppa'(s+)\ge \gppa'(t-)$    and the first inequality of \eqref{q:ineq13} gives 
 $A^{\gppa'(s-)}_{n,k}\ge A^{\gppa'(s+)}_{n,k}\ge A^{\gppa'(t-)}_{n,k}$ which translates into the first inequality of \eqref{eq:monotone}.   Assuming   $\xi_n\cdot e_1\searrow\zeta\cdot e_1$, monotonicity gives a.s.\ existence of the limit and 
\be\label{q:Blim}  \lim_{n\to\infty}B^{\xi_n\pm}(x,x+e_1) \le  B^{\zeta+}(x,x+e_1) \qquad\text{$\Pbig$-a.s.}   
\ee
Monotonicity of the family of cocycles gives a bound that justifies dominated convergence, and hence 
\[ \Ebig\bigl[ \, \lim_{n\to\infty}  B^{\xi_n\pm}(x,x+e_1)\bigr] =    \lim_{n\to\infty}  \gppa'(s_n\pm) =  \gppa'(t+)  =   \Ebig\bigl[  B^{\zeta+}(x,x+e_1)\bigr] .
\]
Equality of expectations forces a.s.\ equality in \eqref{q:Blim}.    To complete part \eqref{th:cocycles:cont} of Theorem \ref{th:cocycles}  replace $e_1$ with $e_2$,  take limits from below, and adapt these arguments.  
 
 \smallskip
 
 Theorem \ref{th:cocycles}  has now been verified for  $B^{\xi\pm}(x,x+e_i)$ defined in \eqref{q:B3}  for $\xi\in\Uset_0$.   The next step is to define $B^{\zeta}(x,x+e_i)=B^{\zeta\pm}(x,x+e_i)$ for   $\zeta\in(\ri\Uset)\smallsetminus\Uset_0$. Since all points of nondifferentiability of $\gpp$ were included in $\Uset_0$,  $\zeta$ must be a point of differentiability in which case we define $B^{\zeta\pm}(x,x+e_i)$ as equal and denote the process by $B^{\zeta}(x,x+e_i)$.  
 
 In order to secure a single null set for all $\xi, \zeta\in\ri\Uset$  for the monotonicity in \eqref{eq:monotone},   we  define the remaining cocycles as  one-sided limits.   Hence define 
\be\label{q:Blim3} \begin{aligned}  B^{\zeta}(\what, x,x+e_1) =B^{\zeta\pm}(\what, x,x+e_1) =  \inf_{\xi\in\Uset_0 \,: \,  \xi\cdot e_1<\zeta\cdot e_1}B^{\xi\pm}(\what, x,x+e_1 ) \\   
 B^{\zeta}(\what, x,x+e_2) =B^{\zeta\pm}(\what, x,x+e_2) =  \sup_{\xi\in\Uset_0\,: \,  \xi\cdot e_1<\zeta\cdot e_1}B^{\xi\pm}(\what, x,x+e_2 ).  
\end{aligned} \ee 

Fix an event $\Ombig_0$ of full $\Pbig$-probability on which cocycles   are finite and  inequalities  \eqref{eq:monotone} hold for all $\xi,\zeta\in\Uset_0$.  
Definition \eqref{q:Blim3} extends \eqref{eq:monotone}   to all $\xi, \zeta$.  

  Pick sequences  $\xi'_n$ and $\xi''_n$ 
   in $\Uset_0$  such that $\xi'_n\cdot e_1\nearrow\zeta\cdot e_1$ and $\xi''_n\cdot e_1\searrow\zeta\cdot e_1$.   Let $s'_n= \xi'_n\cdot e_1/  \xi'_n\cdot e_2$ and similarly $s''_n$.   Definition \eqref{q:Blim3} implies that on the event $\Ombig_0$ 
\be\label{q:Blim4} B^{\zeta}(\what, x,x+e_i)   =  \lim_{n\to\infty}B^{\xi'_n\pm}(\what, x,x+e_i)   
\ee 
 and   by  monotonicity and  integrable bounds  the limit also holds in $L^1(\Pbig)$.  

Next we argue that for the price of a $\Pbig$-null set that is specific to $\zeta$, we can also take the limit in \eqref{q:Blim4} from the right,  as $\xi''_n\to\zeta$.   Consider the edge $(x,x+e_1)$ first.   Monotonicity gives  
\[   B^{\zeta}(x,x+e_1) \ge \lim_{n\to\infty}  B^{\xi''_n\pm}(\what, x,x+e_1) \qquad\text{on the event $\Ombig_0$.}  
 \]
Again   monotonicity and integrability of the cocycles  give both almost sure and   $L^1(\Pbig)$ convergence.  From    differentiability of $\gppa$ at $t=\zeta\cdot e_1/\zeta\cdot e_2$ follows 
\begin{align*}
\E[B^{\zeta}(x,x+e_1)]&=\lim_{n\to\infty} \E[B^{\xi'_n\pm}(x,x+e_1)]=\lim_{n\to\infty} \gppa'(s'_n\pm)=\gppa'(t) = \lim_{n\to\infty} \gppa'(s''_n\pm)\\
&=\lim_{n\to\infty} \E[B^{\xi''_n\pm}(x,x+e_1)]=\E[\, \lim_{n\to\infty} B^{\xi''_n\pm}(x,x+e_1)].
\end{align*}
Consequently $B^{\zeta}(x,x+e_1)=\lim_{n\to\infty} B^{\xi''_n\pm}(x,x+e_1)$   $\Pbig$-a.s.   The same argument with reversed inequalities works for $e_2$.  
Now we have  the limit 
\be\label{q:Blim5} B^{\zeta}(x,x+e_i) = \lim_{n\to\infty}B^{\xi_n\pm}(x,x+e_i)   
\qquad\text{for $\Uset_0 \ni\xi_n\to\zeta\in(\ri\Uset)\smallsetminus\Uset_0$}   
\ee
both $\Pbig$-a.s.\ and  $L^1(\Pbig)$, but with a $\Pbig$-null set that can depend on $\zeta$.  

We turn to verifying the remaining claims of  Theorem \ref{th:cocycles} for the newly defined processes $B^{\zeta}(x,x+e_i)$.  

Part \eqref{th:cocycles:indep}.  The measurability claim again comes from the construction.
Stationarity and the independence claim are  preserved by limits but ergodicity is not.   To verify the ergodicity of  $\varphi^{\zeta}_x(\what)=(\w_x, B^{\zeta}(x,x+e_1), B^{\zeta}(x,x+e_2))$  under both translations $T_{e_1}$ and $T_{e_2}$ we return to the queuing picture.  The limit \eqref{q:Blim3} can also be taken in the queueing processes.   First $\cA_0\ni\alpha_n=\gppa'(s_n^-)\nearrow \gppa'(t)=\beta$.    Since $\cA$ is closed,  $\beta\in\cA$.  Hence there is a stationary queueing process $(\A^\beta, \S, \W^\beta)$ that satisfies Lemma \ref{q:erg-lm} and that we can include in the coupling with  the queueing processes indexed by  $\cA_0$.   The coordinatewise  monotone a.s.\ limit $\lim_{n\to\infty}(\A^{\alpha_n}, \S, \W^{\alpha_n})$ must coincide with $(\A^\beta, \S, \W^\beta)$ by the same reasoning used above:  there are inequalities, namely  $\lim_{n\to\infty}\A^{\alpha_n}_{m,k}\le \A^{\beta}_{m,k}$ and $\lim_{n\to\infty}\W^{\alpha_n}_{m,k}\ge \W^{\beta}_{m,k}$,  
 but the expectations agree and hence force agreement.   The continuous mapping \eqref{q:map}  transports the distribution of 
$\{  (\S_{-n,-k},\,\A^\beta_{-n-1,-k+1},\,\W^\beta_{-n,-k}+\S_{-n,-k}) : n,k\in\Z\}$  to the process 
$\{ (\w_x, B^{\zeta}(x,x+e_1),B^{\zeta}(x,x+e_2)): x\in\Z^2\}$,  which   thereby inherits from Lemma \ref{q:erg-lm}  the   ergodicity claimed in part \eqref{th:cocycles:indep} of  Theorem \ref{th:cocycles}.  

The cocycle properties of part \eqref{th:cocycles:exist} are preserved by pointwise limits.   The identities of part \eqref{th:cocycles:flat} continue to hold without null sets if we refine the limit definition \eqref{q:Blim5}  by defining  $B^\zeta(x,x+e_i)=B^\xi(x,x+e_i)$  whenever $\zeta\in\ri\Uset\smallsetminus\Uset_0$, $\xi\in\Uset_0\cap\Diff$, and $\nabla\gpp(\zeta)=\nabla\gpp(\xi)$.    The   inequalities and limits of part \eqref{th:cocycles:cont} were discussed above.  
 \end{proof}

\section{Coalescence of directional geodesics}
\label{app:geod}

%

This section is dedicated to the proof of Theorem \ref{th:geod:coal:t}, namely, that two cocycle geodesics defined by the same cocycle and the same shift-covariant tie-breaking rule coalesce almost surely under the finite energy condition. Throughout the section $\t$ is a shift-covariant tie-breaking rule that satisfies \eqref{t:stat}, and the finite energy condition \eqref{t:fe} is assumed. The proof  follows   closely   the ideas in \cite{Lic-New-96} for first passage percolation, ported later by \cite{Fer-Pim-05} to   last passage percolation   with exponential $\w_0$. 
A key portion of the proof  is a modification of the weights to produce an event of positive probability that leads to a contradiction.   We present the modification  argument first.  

 Given $\cV\subset\Z^2$ let $\cV^*=\cap_{x\in\cV}\{y\not\le x\}$. Define  two  mappings  on $\Ombig$ by
	\begin{align*}
	&\phi_\cV(\what)=\big\{\w_x,\t(\what,x),\w_y^{i,\alpha}:x\not\in\cV, y\in\cV^*, i\in\{1,2\}, \alpha\in\cA_0\big\}\\
	\text{ and }\quad &\psi_\cV(\what)=\big\{\w_x,\w_y^{i,\alpha}: x\in\cV, y\not\in\cV^*, i\in\{1,2\},\alpha\in\cA_0\big\}.
	\end{align*}
Then   $\what$ is determined by the pair  $(\phi_\cV(\what),\psi_\cV(\what))$. Coordinates   $\psi_\cV(\what)$ are to be  modified, while the tie-breaking steps at the sites outside $\cV$ are part of $\phi_\cV(\what)$ and will be kept.
For each finite subset $\cV \subset \Z^2$ we are  given a Borel set  $H_\cV\subset \R^\cV$. Let $R_\cV=\{\what\in\Ombig:(\w_x)_{x\in\cV}\in H_\cV\}$, and assume that
\be\label{t:fe-R}
\Pbig\bigl( R_{\cV}\,\big\vert\,\phi_\cV\bigr) > 0\qquad\text{$\Pbig$-almost surely.}
\ee
By \eqref{t:fe}, this is satisfied by the sets $H_\cV=(-\infty,K]^\cV$ with $\P\{\w_0\le K\}>0$, which are the ones used below.
%
Event
$R_\cV$ represents the desirable  properties that  the modified  weights will have.  
For each $\what\in\Ombig$ let   $\cW(\what)$ be  a finite subset of $\Z^2$ that depends on $\what$ in a $\kS$-measurable manner.  
This is the set of sites whose  weights will be modified to lie  in $H_{\cW(\what)}$.  

%
%
	
For $\tilde\w\in\Ombig$ let $\Pbig\{\,\cdot\,\vert\,\phi_\cV=\phi_\cV(\tilde\w)\}$ denote a regular conditional probability measure of $\Pbig$, given $\phi_\cV$, evaluated at the value $\phi_\cV(\tilde\w)$.  
For $\Pbig$-almost every $\tilde\w$  this conditional measure is supported on the event $\{\what: \phi_\cV(\what)=\phi_\cV(\tilde\w)\}$.  
For $A\in\kSbig$ define
	\be \Psi(A) = \bigcup_{\cV}\Bigl[  R_\cV\cap    \big\{\tilde\w\in\Ombig:\Pbig\bigl(A\cap\{\cW=\cV\}\,\big\vert\,\phi_\cV=\phi_\cV(\tilde\w)\bigr)>0\big\} \Bigr] \;\in\;\kSbig.   \label{eq:3:6} \ee	

The next lemma says that almost every  $\tilde\w\in\Psi(A)$ is a modification of an $\what\in A$ whose coordinates  $\{\w_x :x\not\in \cW(\what)\}$ and
$\{\w^{i,\alpha}_y : y\in{\cW(\what)}^*, i\in\{1,2\}, \alpha\in\cA_0\}$ and tie-breaking steps $\{\t(x):x\not\in\cW(\what)\}$ 
were left unchanged but whose  weights $\{\w_x:  x\in \cW(\what)\}$ were replaced by new weights  $\{\tilde\w_x:  x\in \cW(\what)\}$ in $H_{\cW(\what)}$.
The association of $\what$ to  $\tilde\w$  might not be measurable but that is not a problem. 


\begin{lemma}\label{lm:chuck}\text{\rm\cite[Lemma 3.1]{Lic-New-96}}   $\Pbig(A)> 0$ implies $\Pbig(\Psi(A))>0$.	
For $\Pbig$-almost every  $\tilde\w\in\Psi(A)$ there exists  $\what\in A$ and a finite $\cV\in\Z^2$ such that $\tilde\w\in R_\cV$, $\cW(\what)=\cV$, and $\phi_{\cV}(\what)=\phi_{\cV}(\tilde\w)$. 
\end{lemma}

\begin{proof} 
To show that $\Pbig(\Psi(A))>0$, fix $\cV$ so that $\Pbig(A\cap\{\cW=\cV\})>0$.
The event $D=\bigl\{\Pbig\bigl(A\cap\{\cW=\cV\}\,\big\vert\,\phi_\cV\bigr)>0\bigr\}$ is $\sigma(\phi_\cV)$-measurable and has positive probability.
Hence, by \eqref{t:fe-R},
 \[\Pbig(\Psi(A))\ge  \Pbig(R_\cV\cap D) = \Ebig\bigl[ \,\mathbf 1_D\,  \Pbig\bigl(R_\cV\,\big\vert\,\phi_\cV\bigr)\,\bigr]  >0.    \]

Let  $\tilde\w\in\Psi(A)$ be such that   $\Pbig\{\,\cdot\,\vert\,\phi_\cV=\phi_\cV(\tilde\w)\}$  is supported on the event $\{\what: \phi_\cV(\what)=\phi_\cV(\tilde\w)\}$ for all finite $\cV$.  
Pick a  finite $\cV\subset\Z^2$ such that 
$\tilde\w\in R_\cV$  and  $\Pbig\bigl(A\cap\{\cW=\cV\}\,\big\vert\,\phi_\cV=\phi_\cV(\tilde\w)\bigr)>0$.   A set of positive measure cannot be empty so  there exists $\what\,\in\,A\cap\{\cW=\cV\}\cap  \{  \phi_\cV=\phi_\cV(\tilde\w)\}$. 
 \end{proof}

We turn to the proof of coalescence.  Beginning with two geodesics that never intersect,   stationarity and  the modification argument show that with positive probability  the following happens for some fixed  rectangle:  from the north boundary of the rectangle  emanates   a geodesic  that intersects no  geodesic that starts to the west or south of the rectangle.   By stationarity   this gives  at least $cL^2$ disjoint  geodesics that start  inside  an $L\times L$ square.  For large $L$ this is a contradiction because there are only $2L$ north and east boundary points through which these geodesics can exit.  

	Fix $\xi \in \ri\Uset$, and recall that $\t$ is a shift-covariant tie-breaking rule and that the finite energy condition \eqref{t:fe} holds. We  consider geodesics $x_{0,\infty}^{u,\t,\xi-}$ and  drop
	$\xi-$ and $\t$ from the notation.
	By \eqref{t:stat} and the shift-covariance of the cocycle $B^{\xi-}$, $x^{u+z}_{0,\infty}(\what)=z+x^{u}_{0,\infty}(T_z\what)$ for all $u,z\in\Z^2$, so, since $\Pbig$ is invariant under the shifts, the distribution of the family $\{x^u_{0,\infty}\}_{u\in\Z^2}$ is invariant under translations.  This is the stationarity used below.
	Because these geodesics  follow the same rule $\t(\what,\cdot)$ and cocycle $B^{\xi-}$, any  two that   intersect  stay together forever. Therefore, we  need
	to prove only  that geodesics eventually intersect.
	The proof is done by way of contradiction. 

	Suppose  $\Pbig\{x_{0,\infty}^a\cap x_{0,\infty}^b=\varnothing\}>0$ for some $a,b\in\Z^2$.  These geodesics are directed into the interior of the first quadrant (Theorem \ref{th:B-geod:direction} and  $\Uset_{\xi-}\subset\ri\Uset$)  and hence must 
  cross any  vertical line to the right of $a$ and $b$.   By restarting   the geodesics   and by  stationarity   we can assume $a=0$,  $b=me_2$ for some $m\in\N$  and  $x_1^0=e_1$. Thus we can take the following   assumption as the basis from which a contradiction will come.    
		\be\label{g:ass99} \Pbig\big\{x_{0,\infty}^0\cap x_{0,\infty}^{m e_2}=\varnothing, \, x^0_1=e_1\big\}>0.\ee
	
	Again by  shift invariance   $x_{0,\infty}^{ie_2}\cap x_{0,\infty}^{(i+m)e_2}=\varnothing$ for infinitely many $i\ge0$,  with   positive probability.
	Consequently, there exists $i>m$ such that 
		\[\Pbig\big\{x_{0,\infty}^0\cap x_{0,\infty}^{m e_2}=\varnothing, \, x_{0,\infty}^{ie_2}\cap x_{0,\infty}^{(i+m)e_2}=\varnothing, \,x^0_1=e_1\big\}>0.\]
	Let  $\ell=i+m$. If $x_{0,\infty}^{m e_2}\cap x_{0,\infty}^{\ell e_2}\not=\varnothing$ then by planarity  $x_{0,\infty}^{m e_2}$   coalesces with $x_{0,\infty}^{i e_2}$, forcing $x_{0,\infty}^{i e_2}$ to 	intersect   with $x_{0,\infty}^{\ell e_2}$.
	So we have $0<m<\ell$ such that 
			\[\Pbig\big\{x_{0,\infty}^0\cap x_{0,\infty}^{m e_2}=\varnothing,\, x_{0,\infty}^{m e_2}\cap x_{0,\infty}^{\ell e_2}=\varnothing,\,x_1^0=e_1\big\}>0.\]

	The remainder  of the proof   depends on whether   the essential infimum of $\w_0$ is taken with positive probability.\medskip

	{\bf Case 1. }  Suppose  $K=\P\text{-}\mathrm{ess\,inf}\,\w_0>-\infty$  and  $\P\{\w_0=K\}>0$. 
	Since   $n^{-1}\sum_{k=0}^{n-1}\w_{ke_2}\to \Ew>K$ a.s.\ we can follow the geodesic $x_{0,\infty}^0$ and  fix  large enough deterministic  $M_1>0$ and $M_2>\ell$  
	such that 
	\be \label{g:A1} 	\begin{aligned}
		\Pbig\Big\{ \what: x_{0,\infty}^0\cap x_{0,\infty}^{m e_2}= \varnothing,&\ x_{0,\infty}^{m e_2}\cap x_{0,\infty}^{\ell e_2}=\varnothing,\, x^0_1=e_1, \\
		&x^0_{M_1+M_2}=(M_1,M_2),\sum_{j=\ell}^{M_2-1}\w_{0,j}>K(M_2-\ell)\Big\}>0.
		\end{aligned} \ee
	Denote by $A$ the intersection of the  event above with the full probability event on which cocycle $B^{\xi-}$ recovers potential $V(\what)=\w_0$ (i.e.\  
	\eqref{eq:VB} holds for $T_x\what$ for all $x$) and  satisfies the additivity condition  \eqref{def:cK:coc} in Definition  \ref{def:cK}.  Let $u_1, u_2$ and  $u_3$ be the points  where geodesics 
$x_{0,\infty}^{\ell e_2}$, $x_{0,\infty}^{me_2}$ and $x_{0,\infty}^0$  (respectively) first intersect  the line $M_2 e_2+\R e_1$.  By definition  $u_3=(M_1, M_2)$.  (See Figure \ref{fig:coal}.) 

The geodesic $x_{0,\infty}^{u_2}$ will be the one that does not intersect any   geodesic that starts from outside, and to the east and/or south of, 
the rectangle $[0,M_1]\times[0,M_2]$.  To make this happen with positive probability, we apply the modification argument to the event $A$ defined above.

	
Let $\cR$ be the lattice region 
	strictly between $x_{0,\infty}^0(\what)$ and $x_{0,\infty}^{\ell e_2}(\what)$, strictly east of $\R e_2$, and strictly south of $M_2 e_2+\R e_1$ (shaded region in Figure \ref{fig:coal}). 
	Define $\cW(\what) = \{x \in \cR : \w_x > K\}$. For a finite set $\cV\subset\Z^2$ let $H_\cV=(-\infty,K]^\cV$, so that $R_{\cV}=\{\what:\w_x\le K\;\forall x\in\cV\}$. Note that $\P\{\w_0\le K\}>0$, so \eqref{t:fe-R} holds by \eqref{t:fe}. Event $\Psi(A)$ is   given in \eqref{eq:3:6} and 
by	Lemma \ref{lm:chuck}   $\Pbig(A)>0$ implies  $\Pbig(\Psi(A))>0$.    The claim to be proved now is this: 

\begin{lemma} \label{g:lm:77}  In  almost every environment  $\tilde\w\in\Psi(A)$,   geodesic $x_{0,\infty}^{u_2}(\tilde\w)$   does not intersect any  geodesic that starts 
  outside the rectangle $[0,M_1]\times[0,M_2]$ at a point $(a,b)$  with either $a\le 0$  or  $b\le 0$. 
\end{lemma}  
		 
\begin{proof} 	From Lemma \ref{lm:chuck} we   read  that almost  every $\tilde\w\in\Psi(A)$  is a modification of some $\what\in A$  
	so that the following items  hold.   
	
	(i)  If  $\w_x>K$ for $x\in\cR$ then  the modified  environment satisfies  $\tilde\w_x=K$. 
	
	(ii) Weights $\{\w_x: x\not\in\cR\}$ and tie-breaking steps $\{\t(x): x\not\in\cW(\what)\}$, as well as the values  $\{B^{\xi-}(y,y+e_i) : y\ge u_1, i=1,2\}$, remain  the same under both $\what$ and $\tilde\w$. In particular,   	geodesics $x_{0,\infty}^{u_1}(\tilde\w)$, $x_{0,\infty}^{u_2}(\tilde\w)$, and $x_{0,\infty}^{u_3}(\tilde\w)$ are  the same as the ones under $\what$. 

Part of the reason  that  $x_{0,\infty}^{u_2}(\tilde\w)$  does not intersect any geodesic that starts from east or south of the rectangle is that it is  ``shielded'' by geodesics $x_{0,\infty}^0(\tilde\w)$ and $x_{0,\infty}^{\ell e_2}(\tilde\w)$.  This is the point of the next lemma.   

	\begin{lemma}\label{lm:geo below}   Let  $\what\in A$ be the configuration  associated 
	to  $\tilde\w\in\Psi(A)$ by Lemma \ref{lm:chuck}.  Then  for any $v\in x^0_{0,\infty}(\what)$ and  $n\ge0$, $x_n^v(\tilde\w)\cdot e_2\le x_n^v(\what)\cdot e_2$.
	Similarly, for any $v\in x^{\ell e_2}_{0,\infty}(\what)$ and  $n\ge0$, $x_n^v(\tilde\w)\cdot e_2\ge x_n^v(\what)\cdot e_2$.
	\end{lemma}
	
	We defer the proof of this  lemma to the end of the appendix. See Figure \ref{fig:coal} for a summary of the construction thus far.

	\begin{figure}
		\begin{center}
		
			\begin{tikzpicture}[>=latex, scale =0.5]
				
				\draw(14,0)node[below]{$M_1$};
				\draw(-0.75,8)node{$M_2$};

				\draw(0,0)rectangle(14,8);
				
				\draw[color=nicosred, line width= 2pt] (0,0)--(0,8)--(14,8);

				\draw[darkgreen, line width=1.2pt,->]plot[smooth] coordinates{(1,0)(5,1)(10,2)(14,8)(15,8.7)};
				\draw[blue, line width=1.2pt,->]plot[smooth] coordinates{(1,0)(5,1)(10,1.5)(14,5.8)(15.2,6.5)};
				\draw[blue, line width=1.2pt] plot coordinates{(0,0)(1,0)};
				\draw[darkgreen, line width=1.2pt,->]plot[smooth] coordinates{(0,4)(1,4.3)(2,5)(3,6.2)(4,8)(6,9)(10,9.7)(13,11)};
				\draw[blue, line width=1.2pt,->]plot[smooth] coordinates{(0,4)(1,4.3)(2,5)(2.3,5.5)(3,8)(4,9.7)(6,11)};
				\draw[nicosred, line width=1.2pt,->]plot[smooth] coordinates{(7,8)(11,8.5)(13.5,9.6)(15,10)};
				\draw[line width=1.2pt]plot[smooth] coordinates{(0,1)(1,2)(3,2.5)(5,4)(7,4.5)(11,5)(12.4,5.4)};
				\draw[line width=1.2pt]plot[smooth] coordinates{(0,3)(1,3.2)(2.6,4.6)(3.1,6.3)};
				\draw[line width=1.2pt,->]plot[smooth] coordinates{(3.1,6.3)(3.3,7)(3.5,7.8)(3.7,8.2)(4,8.8)(4.5,9.2)(5.4,9.5)(6.5,10)(7,11)};
				\draw[line width=1.2pt,->]plot[smooth] coordinates{(12.4,5.4)(14,6.5)(15,7.3)};
				
				\shade[ball color=nicosred](12.4,5.4)circle(2mm); \draw(12.9,5.2)node{$v$};
				\shade[ball color=nicosred](3.1,6.3)circle(2mm); \draw(3.6,6.3)node{$v$};
				\shade[ball color=nicosred](0,0)circle(2mm); \draw(-0.5,0)node{$0$};
				\shade[ball color=nicosred](0,4)circle(2mm); \draw(-0.5,4)node{$\ell$};
				\shade[ball color=nicosred](4,8)circle(2mm); \draw(4.3,7.5)node{$u_1$};
				\shade[ball color=nicosred](7,8)circle(2mm); \draw(7,8.6)node{$u_2$};
				\shade[ball color=nicosred](14,8)circle(2mm); \draw(13.8,8.6)node{$u_3$};

				\fill[color=blue, nearly transparent]plot[smooth] coordinates{(0,0)(1,0)(5,1)(10,2)(14,8)}--(4,8)--(3,6.2)--(2,5)--(1,4.3)--(0,4);
				
				\draw(3.2,10.5)node{$x^{\ell e_2}_{0, \infty}(\tilde\w)$};
				\draw(9,10.2)node{$x^{\ell e_2}_{0, \infty}(\what)$};
				\draw(9,2.7)node{$x^{0}_{0, \infty}(\what)$};
				\draw(11,1)node{$x^{0}_{0, \infty}(\tilde\w)$};
				\draw(10.8,9.1)node{$x^{u_2}_{0,\infty}$};

\end{tikzpicture}

		\end{center}
	\caption{	The shaded region is $\cR$ where the weights are modified to be small.
			Weights on the thick west and north boundaries are large.
			The curved lines represent the various geodesics. The middle geodesic starting at $u_2$ is shielded by all the other ones around it: after the modification, the   geodesic starting at $\ell$ and going through $u_1$ becomes the top geodesic in the picture and the  geodesic starting at $0$ and going through $u_3$ becomes the bottom  
			one in the picture. Geodesics entering from $[0,\ell e_2]$ cannot exit the top between $u_1$ and $u_3$ and hence cannot touch the middle geodesic starting at $u_2$.
		     }	
	\label{fig:coal}
\end{figure}
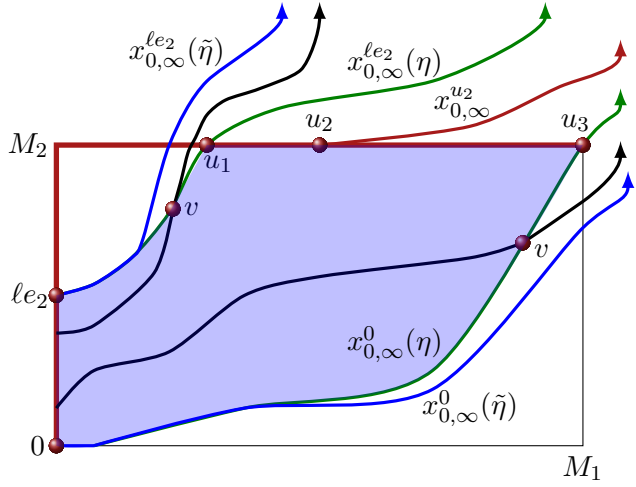

By Lemma \ref{lm:geo below}, any  geodesic $y_{0,\infty}(\tilde\w)$  that    intersects geodesic $x^{u_2}_{0,\infty}(\tilde\w)$ in violation of Lemma \ref{g:lm:77}  must  (i)  enter $\cR$ through the vertical line segment $]0,\ell e_2[$ and   (ii) exit $\cR$  through the line segment $]u_1, u_3[$.   The reason is that if $y_{0,\infty}(\tilde\w)$ exits $\cR$ through  $x^0_{0,\infty}(\what)$ or  $x^{\ell e_2}_{0,\infty}(\what)$,    Lemma \ref{lm:geo below} prevents it from ever touching $x^{u_2}_{0,\infty}(\tilde\w)$.  

To rule out this last possibility,  we simply observe that in environment $\tilde\w$ any path from  $]0,\ell e_2[$ to  $]u_1, u_3[$  through $\cR$ is inferior to following the east and north boundaries of the rectangle.  This is because for $x\in\cR$ each weight $\tilde\w_x=K$, while along the east and north boundaries each weight $\tilde\w_x=\w_x\ge K$ and  by \eqref{g:A1} some weight on the line segment $[\ell e_2, M_2e_2]$ is  $>K$.     Thus no geodesic $y_{0,\infty}(\tilde\w)$  from  outside the rectangle    can  follow this strategy to intersect   $x^{u_2}_{0,\infty}(\tilde\w)$. 
Lemma \ref{g:lm:77} has been proved.  
\end{proof} 

%
%
%
	The Burton-Keane lack of space argument \cite{Bur-Kea-89} gives a contradiction with Lemma \ref{g:lm:77}  and $\Pbig(\Psi(A))>0$,  and thereby contradicts assumption \eqref{g:ass99}. By the ergodic theorem and $\Pbig(\Psi(A))>0$, there exist a fixed $\delta>0$ and an event of positive probability on which, for all  large enough $L$,  event
$\Psi(A)\circ T_z$ occurs for  at least $\delta L^2$ points 	$z\in[0,L]^2$ such that the rectangles $z+[0,M_1]\times[0,M_2]$ are pairwise disjoint and lie inside $[0,L]^2$.    This gives $\delta L^2$ pairwise disjoint geodesics that start inside $[0,L]^2$.  Each of these geodesics must exit through a boundary point of  $[0,L]^2$,  but for large enough $L$ the number of   boundary points is $< \delta L^2$.    Theorem \ref{th:geod:coal:t} has been  proved in Case 1.

	
	\medskip
	
	{\bf Case 2.} Assume $\P\text{-}\mathrm{ess\,inf}\,\w_0$  cannot be taken with positive $\P$-probability.  By taking $K$ close enough but strictly above  $\P\text{-}\mathrm{ess\,inf}\,\w_0$ we can again find   $M_1$ and $M_2$ such that 
		\begin{align*}
		\P\Big\{x_{0,\infty}^0\cap x_{0,\infty}^{m e_2}=\varnothing,&\ x_{0,\infty}^{m e_2}\cap x_{0,\infty}^{\ell e_2}=\varnothing,\, x^0_1=e_1,x^0_{M_1+M_2}=(M_1,M_2),\\
		&\ \w_{i,M_2}\ge K\; \forall i\in[0,M_1],\, \w_{0,j}>K \;\forall i\in[0,M_2]\Big\}>0
		\end{align*}
	and $\P\{\w_0\le K\}>0$. 
	Denote again by $A$ the intersection of the event above with the full measure event on which cocycle $B^{\xi-}$ recovers potential $V(\what)=\w_0$ (i.e.\  
	\eqref{eq:VB} holds for $T_x\what$ for all $x$) and is additive, as in Definition \ref{def:cK}\eqref{def:cK:coc}.
	
	The rest of the argument   is the same as in Case 1.  In particular,  again in environment $\tilde\w$  any path from the east to the north boundary  through $\cR$ is inferior to following the set and north boundaries.   We consider  the proof of   Theorem \ref{th:geod:coal:t} complete.  
	
	
\medskip 	
	
It remains to give the proof of Lemma  \ref{lm:geo below}.  	
	
	\begin{proof}[Proof of Lemma \ref{lm:geo below}]   We do the case $v\in x^0_{0,\infty}(\what)$.    Let $z$ be the first point after which  $x_n^v(\tilde\w)\cdot e_2\le x_n^v(\what)\cdot e_2$ is violated.   Then the two geodesics split at $z$ so that     $x_1^z(\tilde\w)=z+e_2$ and 
	$x_1^z(\what)=z+e_1$. 
	Point  $z$ lies inside the $[0,M_1]\times [0,M_2]$ rectangle because north and east of this rectangle $\tilde\w$ geodesics agree with those of $\what$: there the cocycle $B^{\xi-}$ and the tie-breaking steps are the same in the two environments.        Either  $x^z_{0,\infty}(\tilde\w)$ hits the north boundary of the $[0,M_1]\times [0,M_2]$ rectangle, or it hits the path  $x^0_{0,\infty}(\what)$ inside the rectangle.  
We treat  the   two cases separately. See Figure \ref{fig:coal2}.\medskip
	
		\begin{figure}
		\begin{center}
		
			\begin{tikzpicture}[>=latex, scale =0.5]
				
				\draw(14,0)node[below]{$M_1$};
				\draw(-0.75,8)node{$M_2$};

				\draw(0,0)rectangle(14,8);
				
				\draw[color=nicosred, line width= 2pt] (0,0)--(0,8)--(14,8);

				\draw[darkgreen, line width=1.2pt,->]plot[smooth] coordinates{(1,0)(5,1)(10,2)(14,8)(15,8.7)};
				\draw[darkgreen, line width=1.2pt] plot coordinates{(0,0)(1,0)};
				\draw[darkgreen, line width=1.2pt,->]plot[smooth] coordinates{(0,4)(1,4.3)(2,5)(3,6.2)(4,8)(6,9)(10,9.7)(13,11)};
				\draw[blue, line width=1.2pt]plot[smooth] coordinates{(2.2,0.3)(3.6,2.5)(5.5,4)(7,4.5)(11,5)(12.4,5.4)};
				\draw[blue, line width=1.2pt,->]plot[smooth] coordinates{(2.2,0.3)(2.5,4.6)(2.8,6.3)(3.3,8)(4,8.8)(4.5,9.2)(5.4,9.5)(6.5,10)(7,11)};
				\draw[blue, line width=1.2pt,->]plot[smooth] coordinates{(12.4,5.4)(13.5,6)};
				
				\shade[ball color=nicosred](12.4,5.4)circle(2mm); \draw(13,5.2)node{$y_2$};
				\shade[ball color=nicosred](3.3,8)circle(2mm); \draw(3.1,8.6)node{$y_1$};
				\shade[ball color=nicosred](2.2,0.3)circle(2mm); \draw(2.7,0.25)node{$z$};
				\shade[ball color=nicosred](0,0)circle(2mm); \draw(-0.5,0)node{$0$};
				\shade[ball color=nicosred](0,4)circle(2mm); \draw(-0.5,4)node{$\ell$};
				\shade[ball color=nicosred](4,8)circle(2mm); \draw(4.3,7.5)node{$u_1$};
				\shade[ball color=nicosred](14,8)circle(2mm); \draw(13.9,8.6)node{$u_3$};

				\fill[color=blue, nearly transparent]plot[smooth] coordinates{(0,0)(1,0)(5,1)(10,2)(14,8)}--(4,8)--(3,6.2)--(2,5)--(1,4.3)--(0,4);
				
				\draw(9,10.2)node{$x^{\ell e_2}_{0, \infty}(\what)$};
				\draw(9,2.7)node{$x^{0}_{0, \infty}(\what)$};
				\draw(8,5.5)node{$x^z_{0,\infty}(\tilde\w)$\text{ in Case (b)}};
				\draw(3.3,10.5)node{$x^z_{0,\infty}(\tilde\w)$\text{ in Case (a)}};

\end{tikzpicture}

		\end{center}
	\caption{	Illustration of cases (a) and (b) of the proof of Lemma \ref{lm:geo below}.  
		     }	
	\label{fig:coal2}
\end{figure}
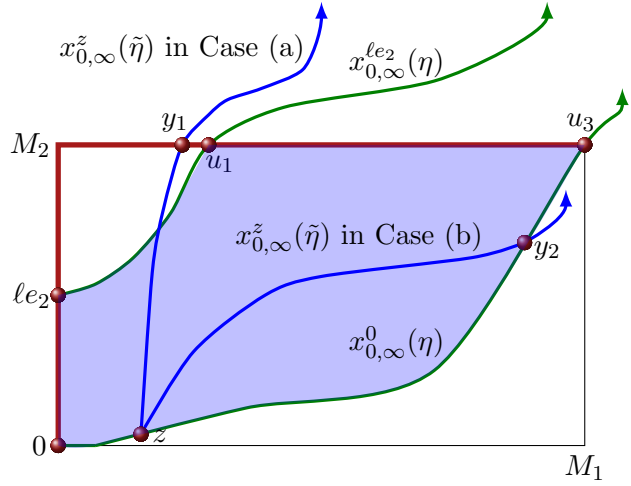

	{\bf Case (a).} $x^z_{0,\infty}(\tilde\w)$ intersects with $[M_2 e_2,u_3]$ at some point $y_1$.
	Since weights were not   modified on $x^0_{0,\infty}(\what)$ and  were not increased anywhere,  the last passage time  $\Gpp_{z,u_3}(\tilde\w)$   under $\tilde\w$ equals  $\Gpp_{z,u_3}(\w)$,   the time  under the old environment $\what$. Combine this with Lemma \ref{lm:grad flow}\eqref{lm:grad flow:a} for $\what$  and potential recovery for $\tilde\w$ to get  
	\[B^{\xi-}(\what,z,u_3)=\Gpp_{z,u_3}(\w)\le B^{\xi-}(\tilde\w,z,u_3).\]
Since $\tilde\w$ agrees with $\what$ on the north boundary of the rectangle, we have   $B^{\xi-}(\what,y_1,u_3)=B^{\xi-}(\tilde\w,y_2,u_3)$.  
	The additivity of cocycles then implies that $B^{\xi-}(\what,z,y_1)\le B^{\xi-}(\tilde\w,z,y_1)$.
	
	 On the other hand, we have
	\[B^{\xi-}(\tilde\w,z,y_1)= \Gpp_{z,y_1}(\tilde\w)\le \Gpp_{z,y_1}(\w)\le B^{\xi-}(\what,z,y_1).\]
	The first equality follows again from Lemma \ref{lm:grad flow}\eqref{lm:grad flow:a} for the cocycle geodesic $x^z_{0,\infty}(\tilde\w)$ in environment $\tilde\w$.  The first inequality comes from the fact that the modified weights are no larger than the original ones. 
	The second inequality is again due to 
	potential recovery. 
Combine all the inequalities above to conclude that
\be \label{g:86}  B^{\xi-}(\what,z,y_1) = B^{\xi-}(\tilde\w,z,y_1)     \ee
and  $B^{\xi-}(\tilde\w,z,u_3)=\Gpp_{z,u_3}(\w)$.
	Rewrite the last equality as 
		\[\sum_{i=0}^{\abs{u_3-z}_1-1}B^{\xi-}(\tilde\w,x_i^z(\what),x_{i+1}^z(\what))
		= B^{\xi-}(\tilde\w,z,u_3) = \Gpp_{z,u_3}(\w)=\sum_{i=0}^{\abs{u_3-z}_1-1}\w_{x_i^z(\what)}.\]
Potential recovery under $\tilde\w$ and $\tilde\w=\w$  on the path $x^0_{0,\infty}(\what)$  now  imply  that the sums agree term by term. In particular, \be \label{g:88} B^{\xi-}(\tilde\w,z,x_1^z(\what))=\w_z.\ee

In the same manner we deduce   the statement 
\be \label{g:89} 	B^{\xi-}(\what,z,x_1^z(\tilde\w))=\w_z. \ee
To see this last identity, consider this:  
	\begin{align*} 
	&\sum_{i=0}^{\abs{y_1-z}_1-1}\tilde\w_{x_i^z(\tilde\w)}\le \sum_{i=0}^{\abs{y_1-z}_1-1}\w_{x_i^z(\tilde\w)} \le \sum_{i=0}^{\abs{y_1-z}_1-1}B^{\xi-}(\what,x_i^z(\tilde\w),x_{i+1}^z(\tilde\w))  \\
		&\qquad\qquad = B^{\xi-}(\what,z,y_1) = B^{\xi-}(\tilde\w,z,y_1)  = \Gpp_{z,y_1}(\tilde\w)=
		\sum_{i=0}^{\abs{y_1-z}_1-1}\tilde\w_{x_i^z(\tilde\w)}.
		\end{align*} 
The first two inequalities are valid term by term,  by  the modification and potential recovery.  The third step is cocycle additivity, the fourth is \eqref{g:86} from above, and the last two are  due to $\{x_i^z(\tilde\w)\}$ being a cocycle geodesic.   The upshot is that the second and third sums must agree term by term.  The equality of the first terms is \eqref{g:89}  because $\tilde\w=\w$  on the path $x^0_{0,\infty}(\what)$.  

By \eqref{g:88}, potential recovery, and $\tilde\w_z=\w_z$, in environment $\tilde\w$ both steps out of $z$ have the minimal increment $\tilde\w_z$, and by \eqref{g:89} the same is true in environment $\what$. So the step from $z$ is decided by the tie-breaking rule in both environments. The point $z$ lies on $x^0_{0,\infty}(\what)$ and hence outside $\cW(\what)$, so $\t(\tilde\w,z)=\t(\what,z)$. This is incompatible with $x_1^z(\tilde\w)\ne x_1^z(\what)$.   Thus Case (a) led to a contradiction.

	\medskip

	{\bf Case (b).} 	$x^z_{0,\infty}(\tilde\w)$ intersects with $x^0_{0,\infty}(\what)$ at some point $y_2>z$. 
 	Start by observing that $\Gpp_{z,y_2}(\tilde\w)=\Gpp_{z,y_2}(\w)$.  
	  Hence,
	$B^{\xi-}(\what,z,y_2)=B^{\xi-}(\tilde\w,z,y_2)=\Gpp_{z,y_2}(\w)$. 
	An argument similar to  Case (a)  shows  that   Case (b)  cannot happen either. 
	
	We have proved the part of Lemma \ref{lm:geo below} that claims  $x_n^v(\tilde\w)\cdot e_2\le x_n^v(\what)\cdot e_2$ 
	 for any $v\in x^0_{0,\infty}(\what)$.  
	The claim for  geodesics starting from $v\in x^{\ell e_2}_{0,\infty}(\what)$ is proved similarly.
	\end{proof}

\section{Ergodic theorem for cocycles}  
\label{app:aux}

Cocycles satisfy a uniform ergodic theorem. The following is a special case of Theorem A.3 of  \cite{Geo-etal-15}. Note that a one-sided bound suffices for a hypothesis.   Recall   Definition \ref{def:cK} for the space $\Cor$ of centered cocycles.   

\begin{theorem}\label{th:Atilla}
Assume $\P$ is ergodic under the transformations  $\{T_{e_i}:i\in\{1,2\}\}$. 
Let $F\in\Cor$.   Assume there exists  a function $V$ such that  for $\P$-a.e.\ $\w$  
\be \label{cL-cond}
\varlimsup_{\e\searrow0}\;\varlimsup_{n\to\infty} \;\max_{x: \abs{x}_1\le n}\;\frac1n \sum_{0\le k\le\e n} 
\abs{V(T_{x+ke_i}\w)}=0\qquad\text{for $i\in\{1,2\}$ }\ee  
and 
$\max_{i\in\{1,2\}} F(\w,0,e_i)\le V(\w)$.  
Then  
\[\lim_{n\to\infty}\;\max_{\substack{x=z_1+\dotsm+z_n\\z_{1,n}\in\{e_1, e_2\}^n}} \;\frac{\abs{F(\w,0,x)}}n=0 \qquad\text{for   $\P$-a.e.\ $\w$.}\]
\end{theorem}
If the process  $\{V(T_x\w):x\in\Z^2\}$ is  i.i.d.,  then a sufficient condition for \eqref{cL-cond} is  $\E(\abs{V}^p)<\infty$ for some $p>2$  \cite[Lemma A.4]{Ras-Sep-Yil-13}

\section{Percolation cone}\label{sec:cone-pf}

This appendix proves Theorem \ref{th:flat-edge}. Section \ref{sec:flat-edge-pd} determines  the set $\{\xi: \gpp(\xi)=1\}\subset\Uset$  for each  $p_1<1$ (Theorem \ref{perc-flat-thm}). Section \ref{sec:flat-diff-pf} proves that if $p_1\ge\opc$, then $\gpp$ is differentiable at $(\beta_{p_1},1-\beta_{p_1})$ (Theorem \ref{thm:edge-diff}). For $p_1>\opc$ this point is $\etamax$, and differentiability at $\etamin$ follows by symmetry. At $p_1=\opc$,   $(\beta_{\opc},1-\beta_{\opc})=(\tfrac12,\tfrac12)$.

 The standing assumptions are those of Theorem \ref{th:flat-edge}: $\{\w_x\}_{x\in\Z^2}$ i.i.d., $\w_x\le1$, $p_1=\P\{\w_0=1\}<1$, and $\E\w_0^-<\infty$. Since $\w_x\le1$, the positive-tail condition in \eqref{shape-conditions} holds trivially, so by Theorem \ref{th:shape} the almost sure shape theorem \eqref{lln5} holds. 
 

   The oriented percolation weights are defined by $\sigma_x=\one\{\w_x=1\}$.  
The oriented percolation event $u\to v$ means that there exists an up-right path $u=x_0, x_1,\dotsc, x_m=v$ with $x_{i}-x_{i-1}\in\{e_1,e_2\}$, $m=\abs{v-u}_1$, and such that $\sigma_{x_i}=1$ for $i=1,\dotsc,m$. (The initial point $u$ may be open or closed.)   The percolation event $u\to\infty$ means that there is an infinite up-right path starting at $u$ along which all weights $\sigma_{x}=1$ except perhaps $\sigma_u$. For $n\in\Z$ let $D_n=\{(i,j)\in\Z^2:\ i+j=n\}$ be the set of points on level $n$, and write $u\to D_n$ if $u\to v$ for some $v\in D_n$.
 

\subsection{Flat edge} \label{sec:flat-edge-pd}




%

This section determines the set $\{\xi:\gpp(\xi)=1\}\subset\Uset$ for each $p_1\in[0,1)$. 


\begin{theorem}\label{perc-flat-thm}
Let $\xi\in\Uset$. Then $\gpp(\xi)\le1$, and the following hold.
\begin{enumerate}[\rm(a)]
\item If $0\le p_1<\opc$, then $\gpp(\xi)<1$.
\item If $p_1=\opc$, then $\gpp(\xi)=1$ if and only if $\xi=(\tfrac12,\tfrac12)$.
\item If $p_1>\opc$, then $\gpp(\xi)=1$ if and only if $1-\beta_{p_1}\le\xi\cdot e_1\le\beta_{p_1}$.
\end{enumerate}
\end{theorem}

The rest of the section gives the proof.  We begin with two lemmas, the first one from elementary probability. 

 \begin{lemma}\label{lm:t.aux7} 
Let $A_1,\dots,A_n$ be arbitrary events, and let $B_1,\dots,B_n$ be
mutually independent events with common probability $P(B_i)=p>0$. Suppose the family
$\{B_i\}_{i=1}^n$ is independent of the family $\{A_i\}_{i=1}^n$. Then  
\[
P\bigl({\textstyle\bigcup}_{i=1}^n A_i\bigr)
\le {p}^{-1} P\bigl({\textstyle\bigcup}_{i=1}^n(A_i\cap B_i)\bigr).
\]
\end{lemma}

\begin{proof}
Let  $N=\sum_{i=1}^n\mathbf{1}_{A_i}$.
Expand the product below, use the independence assumption, and then inclusion--exclusion:  
\[
\begin{aligned}
&E[1-(1-p)^N]
=E\Bigl[1-\prod_{i=1}^n(1-p\mathbf{1}_{A_i})\Bigr]
=E\Bigl[
  \sum_{\varnothing\ne I\subseteq\{1,\dots,n\}}
  (-1)^{|I|+1}p^{|I|}\prod_{i\in I}\mathbf{1}_{A_i}
  \Bigr]\\
&\quad =\sum_{\varnothing\ne I\subseteq\{1,\dots,n\}}
  (-1)^{|I|+1}p^{|I|}
  P\bigl({\textstyle\bigcap}_{i\in I}A_i\bigr) =\sum_{\varnothing\ne I\subseteq\{1,\dots,n\}}
  (-1)^{|I|+1} P\bigl({\textstyle\bigcap}_{i\in I}(A_i\cap B_i)\bigr) \\
  &\quad =P\bigl({\textstyle\bigcup}_{i=1}^n(A_i\cap B_i)\bigr).
\end{aligned}
\]
Then, from    $1-(1-p)^N\ge p \ind_{\{N\ge1\}}$, 
\[
P\bigl({\textstyle\bigcup}_{i=1}^n(A_i\cap B_i)\bigr)
=E[1-(1-p)^N]
\ge p P(N\ge1)
=p P\bigl({\textstyle\bigcup}_{i=1}^n A_i\bigr).
\qedhere  \]
\end{proof}

\smallskip 
  
\begin{lemma}\label{lm:perc-aux}
Fix $\rho\in(0,1)$ and $L>0$.  Then there exists $q_0=q_0(\rho, L)\in(0,1)$ 
such that the following holds.   If $q\ge q_0$ and $\{\tau_z: z\in\Z^2\}$  are stationary 
$\{0,1\}$-valued random variables with $\P\{\tau_0=1\}=q$  
and $\tau_0$ independent of $\{\tau_z: \abs{z}_1>L\}$, 
then there are  positive constants $a$ and $b$ such that,  for all $m\ge1$, 
\[P\Big\{\exists \text{ up-right path } z_{0,m}: z_0=0, \; \sum_{i=0}^m \tau_{z_i} \le \rho(m+1)\Big\}\le a e^{-bm}.\]
\end{lemma}

\begin{proof} First, fix $m\ge 1$ and an up-right path $z_{0,m}$ such that $z_0=0$.
For $k\le m+1$, on the event $\sum_{j=0}^m (1-\tau_{z_j})=k$,  
we can find, among the $k$ indices with $\tau_{z_j}=0$,  indices 
$j_1,\dotsc,j_{\ce{k/(L+1)}}$ such that $\abs{z_{j_i}-z_{j_r}}_1>L$ for all $r,i\le\ce{k/(L+1)}$. Given these indices, the probability that $\tau_{z_{j_r}}=0$ for all $r\le\ce{k/(L+1)}$
is bounded above by $(1-q)^{k/(L+1)}$. There are at most $\binom{m+1}{\ce{k/(L+1)}}$
many choices of these indices.  Consequently,
\[P\Big\{\sum_{j=0}^m (1-\tau_{z_j})=k\Big\}\le\binom{m+1}{\ce{k/(L+1)}}(1-q)^{k/(L+1)}.\]
This implies
\[P\Big\{\sum_{i=0}^m \tau_{z_j} \le \rho(m+1)\Big\}\le \sum_{(1-\rho)(m+1)\le k\le m+1}{\binom{m+1}{\ce{k/(L+1)}}}(1-q)^{k/(L+1)}.\]
Since there are $2^m$ paths $z_{0,m}$, 
the probability in the claim of the lemma is then bounded above by
\begin{align*}
&2^m\sum_{(1-\rho)(m+1)\le k\le m+1}{\binom{m+1}{\ce{k/(L+1)}}}(1-q)^{k/(L+1)}\\
&\qquad\le 2^m(1-q)^{(1-\rho)(m+1)/(L+1)} \times (L+1)\sum_{i=0}^{m+1}\binom{m+1}{i}\\
&\qquad\le 2(L+1)(1-q)^{(1-\rho)/(L+1)}\exp\Big\{\Big[\log 4+\frac{1-\rho}{L+1}\log(1-q)\Big]m\Big\}.
\end{align*}
This decays exponentially fast as soon as $q>1-4^{-(L+1)/(1-\rho)}$.
\end{proof}

The main work is in the next proposition.   

\begin{proposition}\label{prop:marchand}
Assume $p_1=\P\{\w_0=1\}>\opc$. 
Then for each $\e\in(0,1-\beta_{p_1})$ there exist finite positive constants $A$, $B$, and $\delta$ such that for all $\ell,k\in\N$ with $\ell/k\le (1-\beta_{p_1}-\e)/(\beta_{p_1}+\e)$ we have
\begin{align}\label{marchand:decay}
\P\{\Gpp_{0,(k,\ell)}\ge(1-\delta)(k+\ell)\} \le A e^{-Bk}.
\end{align}
\end{proposition}

\begin{proof}
For $N\in\N$ and $z\in\Z^2$ define 
	\[C_N(0)=\{x\ge0:\abs{x}_1<N\},\ C_N(z)=z+C_N(0),\text{ and }B_N(z)=C_{2N}(z)\setminus C_N(z).\]
	
Fix $\e_1\in(0,\e)$. Fix $\delta_0\in(0,1)$ such that $p_0=\P\{\w_0\ge1-\delta_0\}>\opc$ and $\beta_{p_1}\le\beta_{p_0}<\beta_{p_1}+\e$.  Here we used the continuity of $\beta_{p}$  as a function of $p\in(\opc,1)$ \cite[(3) on p.\ 1031]{Dur-84}.  

Abbreviate 
	\[\lambda_{p_1,\e_1}=\frac{1-\beta_{p_1}-\e_1}{\beta_{p_1}+\e_1}\,.\]	
Given $N\in\N$ and $\w\in\Omega$, color $z\in\Z^2$ {\sl black} if 
	\begin{align}\label{eq:black}
	\Gpp_{u,v}\le\abs{v-u}_1-\delta_0
	\end{align}
for every $Nz\le u\le v$ with $\abs{u-Nz}_1=N$, $\abs{v-u}_1=N$, and
	\[\frac{(v-u)\cdot e_2}{(v-u)\cdot e_1}\le\lambda_{p_1,\e_1}.\]
See the left panel of Figure \ref{fig:g-b}. Color $z$ {\sl white} if it is not black. 
Then
	\begin{align*}
	&\P\{0\text{ is white}\}\le(N+1)\P\big\{\exists v\ge0:\abs{v}_1=N, \tfrac{v\cdot e_2}{v\cdot e_1}\le\lambda_{p_1,\e_1},\Gpp_{0,v}>\abs{v}_1-\delta_0\big\}\\
		&\qquad\le p_1^{-1}(N+1)\P\big\{\exists v\ge0:\abs{v}_1=N,\tfrac{v\cdot e_2}{v\cdot e_1}\le\lambda_{p_1,\e_1},\Gpp_{0,v}>\abs{v}_1-\delta_0,\,\w_v=1\big\}.
	\end{align*}
    The inequality is Lemma \ref{lm:t.aux7}, coming from the independence of the i.i.d.\ weights $\{\w_v: \abs{v}_1=N\}$ and $\{\Gpp_{0,v}: \abs{v}_1=N\}$. 

Define the oriented site  percolation weights $\sigma_x=\one\{\w_x\ge1-\delta_0\}$. 
Since $\w_x\le1$ for all $x$, $\w_v=1$ and $\Gpp_{0,v}>\abs{v}_1-\delta_0$ imply the existence of an up-right path from $0$ to $v$
with $\w_x\ge1-\delta_0$ along the path. In other words $0\to v$ in the oriented percolation process.
Thus, 
	\[\P\{0\text{ is white}\}\le p_1^{-1}(N+1)\P\big\{N^{-1}a_N\ge 1/(\lambda_{p_1,\e_1}+1)=\beta_{p_1}+\e_1\big\}.\]
Since $\beta_{p_1}+\e_1>\beta_{p_0}$, the probability on the right-hand side decays exponentially fast. This is \cite[(2) on p.~1018]{Dur-84}, stated there for the right edge $\bar r_n$ of the half-line process. It applies to $a_N$ because $a_N=\bar r_N$ on the event that the cluster of $0$ reaches level $N$.  
Consequently, the probability the origin is white vanishes as $N\to\infty$.

Pick $\rho\in(0,1)$ such that
	\[\frac{1-\beta_{p_1}-\e}{1-\beta_{p_1}-\e_1} <\rho<1.\]
Pick $N$ large enough so that $\P\{0\text{ is black}\}\ge q_0$, where $q_0$ is from Lemma \ref{lm:perc-aux}.
Pick $\delta>0$ small so that 
\begin{align}\label{eq:delta}
\rho'=\frac{1-\beta_{p_1}-\e}{1-\beta_{p_1}-\e_1} + \frac{N\delta}{\delta_0}<\rho.
\end{align}
 Let $m_0$ be the least positive integer with $\rho'(m_0+3)\le\rho(m_0+1)$. 

 It suffices now to prove \eqref{marchand:decay} for $(k,\ell)$ such that  $k+\ell\ge3N$ and $m=\fl{(k+\ell)/N}-2\ge m_0$
 because  the finitely many excluded pairs $(k,\ell)$  can be  covered by enlarging $A$. 
 
Given an up-right path $x_{0,k+\ell}$ from $0$ to $(k,\ell)$
let  $m=\fl{(k+\ell)/N}-2$  and define the up-right path $z_{0,m}$ by
	\[x_{(j+1)N-1}\in C_{N}(Nz_j),\quad 0\le j\le m.\]
Vertices $Nz_j$ are  the south-west corners of the squares $\{y:Nz\le y\le Nz+(N-1,N-1)\}$, $z\in\Z_+^2$, that  path $x_{0,k+\ell}$ enters in succession. See Figure \ref{fig:triangle}.

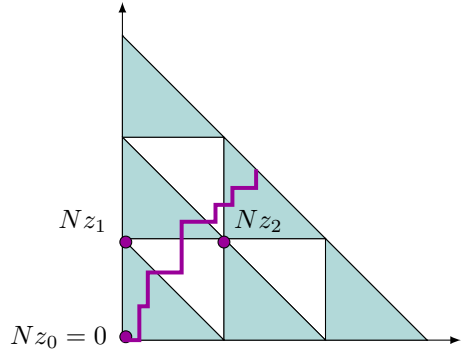
\begin{figure}[h!]
\begin{center}

\begin{tikzpicture}[>=latex,scale=0.45]

\definecolor{sussexg}{rgb}{0,0.5,0.5}
\definecolor{sussexp}{rgb}{0.6,0,0.6}

\draw[<->] (10,0)--(0,0)--(0,10);

\foreach \x in {3,6,9}
	{
		\draw (\x, 0) -- (0,\x);
		\draw(0, \x)--(9-\x, \x);
		\draw( \x,0)--(\x, 9- \x); 
	}
	
	\draw[fill=sussexg!30](0, 0) -- (0, 3) -- (3, 0)--(0,0);
	\draw[fill=sussexg!30](0, 3) -- (0, 6) -- (3, 3)--(0,3);
	\draw[fill=sussexg!30](0,6) -- (0, 9) -- (3,6)--(0,6);
	\draw[fill=sussexg!30](3, 0) -- (3, 3) -- (6, 0)--(3,0);
	\draw[fill=sussexg!30](3, 3) -- (3, 6) -- (6, 3)--(3,3);
	\draw[fill=sussexg!30](6,0) -- (6, 3) -- (9, 0)--(6,0);
	
	\draw[fill = sussexp] (0.1,0.1)circle(1.75mm)node[left]{\small$Nz_0 = 0\,\,$};
	\draw[fill = sussexp] (0.1,2.9)circle(1.75mm)node[above left]{\small$Nz_1 \,\,$};
	\draw[fill = sussexp] (3,2.9)circle(1.75mm)node[above right]{\small$Nz_2 \,\,$};;

	\draw[line width=1.5pt,scale=.5,color=sussexp](0,0)--(1,0)--(1,2)--(1.5,2)--(1.5,3.5)--(1.5, 4)--(3.5,4)--(3.5,7)--(5.5,7)--(5.5,8)--(6.5,8)--(6.5, 9)--(7.9,9)--(7.9,10.1);  
\end{tikzpicture}

\end{center}
\caption{Up-right path $z_{\centerdot}$ constructed from the up-right path $x_{\centerdot}$.}
\label{fig:triangle}
\end{figure}

We prove next that for $k$ and $\ell$ as in the claim of the proposition
	\begin{align}\label{eq:marchand}
	\begin{split}
	&\{\Gpp_{0,(k,\ell)}\ge(1-\delta)(k+\ell)\}\\
	&\qquad\subset\Big\{\exists \textrm{ an up-right path }z_{0,m}:z_0=0,\, \sum_{j=0}^m\one\{z_j\text{ is black}\}\le\rho(m+1)\Big\}.
	\end{split}
	\end{align}
The proposition then follows from Lemma \ref{lm:perc-aux}. 

Fix an up-right path $x_{0,k+\ell}$ such that $x_0=0$, $x_{k+\ell}=(k,\ell)$, and
	\begin{align}\label{eq:Gpp-large}
	\sum_{i=0}^{k+\ell-1}\w_{x_i}\ge(1-\delta)(k+\ell).
	\end{align}

Consider $j\le m$. If $z_j$ is black, then we  label $j$ as {\sl good} if
		\[\frac{(x_{(j+2)N}-x_{(j+1)N})\cdot e_2}{(x_{(j+2)N}-x_{(j+1)N})\cdot e_1}\le \lambda_{p_1,\e_1}.\]
If $z_j$ is black and $j$ is not good, then say $j$ is {\sl bad}. See Figure \ref{fig:g-b}.

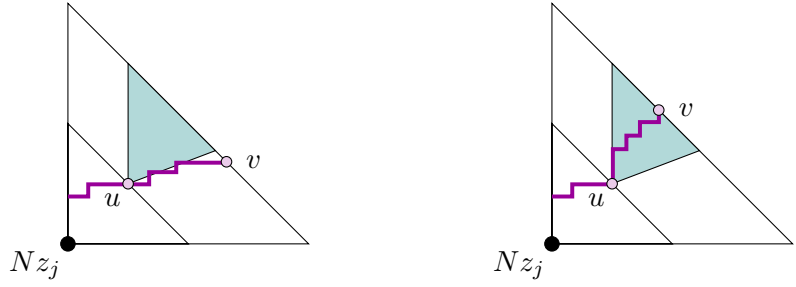
\begin{figure}[h!]

\begin{tikzpicture}[>=latex, scale=0.4]

\definecolor{sussexg}{rgb}{0,0.5,0.5}
\definecolor{sussexp}{rgb}{0.6,0,0.6}

\draw (0,0)--(4,0)--(0,4)--(0,0)--(0,8)--(8,0)--(0,0);

\draw[fill = sussexg!30](2,2)--(2,6)--(4.9,3.1)--(2,2); 

\draw[line width=1.5pt,scale=.45,color=sussexp](0,3.5)--(1.5,3.5)--(1.5, 4.4)--(4.5,4.4)--(6,4.4)--(6,5.3)--(8,5.3)--(8,6)--(11.7, 6);


\draw[fill=sussexp!20](2,2)circle(1.75mm);
\draw[fill=sussexp!20](5.27,2.73)circle(1.75mm)node[right]{\,\,$v$};

\draw(2.1,2)node[below left]{\,$u$};

\draw[fill = black] (0,0)circle(2.4mm);
\draw(0.1,0.1)node[below left]{\,$Nz_j$};


\end{tikzpicture}
\hspace{2cm} 
\begin{tikzpicture}[>=latex, scale=0.4]

\definecolor{sussexg}{rgb}{0,0.5,0.5}
\definecolor{sussexp}{rgb}{0.6,0,0.6}

\draw (0,0)--(4,0)--(0,4)--(0,0)--(0,8)--(8,0)--(0,0);

\draw[fill = sussexg!30](2,2)--(2,6)--(4.9,3.1)--(2,2); 

\draw[line width=1.5pt,scale=.45,color=sussexp](0,3.5)--(1.5,3.5)--(1.5, 4.4)--(4.5,4.4)--(4.5,7)--(5.5,7)--(5.5,8)--(6.5,8)--(6.5, 9)--(7.9,9)--(7.9,10.1);  

\draw[fill=sussexp!20](2,2)circle(1.75mm);
\draw[fill=sussexp!20](3.55,4.45)circle(1.75mm)node[right]{\,\,$v$};

\draw(2.1,2)node[below left]{\,$u$};

\draw[fill = black] (0,0)circle(2.4mm);
\draw(0.1,0.1)node[below left]{\,$Nz_j$};


\end{tikzpicture}

\caption{A good $j$ (left) and a bad $j$ (right). Here, $u=x_{(j+1)N}$ and $v=x_{(j+2)N}$. 
The shaded region contains the points $x\ge u$ such that $(x-u)\cdot e_2/(x-u)\cdot e_1>\lambda_{p_1,\e_1}$.}
\label{fig:g-b}
\end{figure}

If $j$ is good then $x_{(j+1)N,(j+2)N}$, the portion of the path crossing $B_N(z_j)$, has a passage time no larger than $N-\delta_0$ (see again  \eqref{eq:black}). 
Since $\w_x\le1$ for all $x$,
	\[\sum_{i=0}^{k+\ell-1}\w_{x_i}\le k+\ell-\delta_0\abs{\{j\le m:j\text{ is good}\}}.\]
Together with \eqref{eq:Gpp-large} this implies
	\begin{align}\label{good}
	\abs{\{j\le m:j\text{ is good}\}}\le \frac\delta{\delta_0}(k+\ell)\le\frac{\delta N}{\delta_0} (m+3).
\end{align} 

If $j$ is bad, then 
\begin{align*}
(x_{(j+2)N}-x_{(j+1)N})\cdot e_2
&>\lambda_{p_1,\e_1}(x_{(j+2)N}-x_{(j+1)N})\cdot e_1\\
&=\lambda_{p_1,\e_1}[N-(x_{(j+2)N}-x_{(j+1)N})\cdot e_2].
\end{align*}
This implies
\begin{align}\label{eq:bad-bound}
(x_{(j+2)N}-x_{(j+1)N})\cdot e_2>\frac{N\lambda_{p_1,\e_1}}{1+\lambda_{p_1,\e_1}}\,.
\end{align}
Similarly, $\ell \le \lambda_{p_1,\e} k$ implies
\[x_{k+\ell}\cdot e_2=\ell \le\frac{(k+\ell)\lambda_{p_1,\e}}{1+\lambda_{p_1,\e}}\,.\]
Adding \eqref{eq:bad-bound} over the bad $j$ now leads to
\[\frac{N\lambda_{p_1,\e_1}}{1+\lambda_{p_1,\e_1}}\abs{\{j\le m:j\text{ is bad}\}}\le \frac{(k+\ell)\lambda_{p_1,\e}}{1+\lambda_{p_1,\e}}.\]
Consequently, 
	\begin{align}\label{bad}
	\abs{\{j\le m:j\text{ is bad}\}}\le \frac{1-\beta_{p_1}-\e}{1-\beta_{p_1}-\e_1}\cdot\frac{k+\ell}N\le\frac{1-\beta_{p_1}-\e}{1-\beta_{p_1}-\e_1} (m+3).
\end{align}

By the choice of $\delta$ in \eqref{eq:delta}, adding \eqref{good} and \eqref{bad} shows  that on a path satisfying \eqref{eq:Gpp-large} 
 the number of black $z_j$ is at most $\rho'(m+3)\le\rho(m+1)$ when $m\ge m_0$. Inclusion
\eqref{eq:marchand} has been verified  for $m\ge m_0$.  Proposition \ref{prop:marchand}  follows, with the constant  $A$ enlarged to cover the finitely many $k$ for which $m<m_0$ or $k+\ell<3N$. 
\end{proof}

\begin{proof}[Proof of Theorem \ref{perc-flat-thm}] 
Since $\w\le1$, $\Gpp_{0,x}\le\abs x_1$ for all $x$, so $\gpp\le1$ on $\Uset$. We prove part (c) first, since the proof of part (b) uses it.\medskip

{\bf Part (c).} On the positive probability event $\{0\to\infty\}$ in the oriented percolation with weights $\sigma_x=\one\{\w_x=1\}$ we have $0\to (a_n,n-a_n)$, and hence $\Gpp_{0,(a_n,n-a_n)}\ge n-1+\w_0$, for all $n$. Since $a_n/n\to\beta_{p_1}$ as $n\to\infty$, the shape theorem \eqref{lln5} implies $\gpp(\beta_{p_1},1-\beta_{p_1})=1$.
By symmetry, we have $\gpp(1-\beta_{p_1},\beta_{p_1})=1$ as well. Concavity of $\gpp$ and the fact that $\gpp(\xi)\le1$ for all $\xi\in\Uset$ imply 
that $\gpp(\xi)=1$ when $1-\beta_{p_1}\le\xi\cdot e_1\le\beta_{p_1}$.

For the other direction assume $\xi\cdot e_1>\beta_{p_1}$. If $\xi=e_1$, then $\gpp(\xi)=\E\w_0<1$. Otherwise apply \eqref{marchand:decay} with $\e<\xi\cdot e_1-\beta_{p_1}$ to $(k,\ell)=\fl{n\xi}$, together with the Borel--Cantelli lemma and \eqref{lln5}, to deduce that $\gpp(\xi)\le1-\delta$. The result for $\xi\cdot e_1<1-\beta_{p_1}$ comes by symmetry.
\medskip

{\bf Part (a).} Fix $\delta_0\in(0,1)$ with $p_0=\P\{\w_0\ge1-\delta_0\}<\opc$. 
Consider the oriented site percolation with weights $\sigma'_x=\one\{\w_x\ge1-\delta_0\}$, whose density is $p_0$. Since $p_0<\opc$, there exists $\gamma>0$ such that $\P\{0\to D_n\}\le e^{-\gamma n}$ for all $n$ in this percolation \cite[(1) on p.~1017]{Dur-84}.

Given $N\in\N$, color $z\in\Z^2$ \emph{black} if $\Gpp_{u,v}\le N-\delta_0$ for all $u,v$ with $Nz\le u\le v$, $\abs{u-Nz}_1=N$, and $\abs{v-u}_1=N$, and \emph{white} otherwise. Since weights are $\le1$, $\Gpp_{u,v}>N-\delta_0$ forces each weight $\w_{x_0},\dots,\w_{x_{N-1}}$ on a geodesic from $x_0=u$ to $x_N=v$ to be at least $1-\delta_0$. In particular $u$ is connected by a $\sigma'$-open path to the level $N-1$ above it. There are $N+1$ choices of $u$, so 
\[\P\{0\text{ is white}\}\le(N+1)e^{-\gamma(N-1)}\longrightarrow0\qquad\text{as }N\to\infty.\]
The process $\{\one\{z\text{ is black}\}\}_{z\in\Z^2}$ is stationary under shifts of the index $z\in\Z^2$.

The set $Nz+\{x\ge0:\ N\le\abs x_1<2N\}$ is contained in $Nz+[0,2N)^2$, and these squares are disjoint for $z,z'$ with $\abs{z-z'}_\infty\ge2$, in particular when $\abs{z-z'}_1>4$. Hence the color of $0$ is independent of the colors of the $z$ with $\abs z_1>4$. 
Fix $N$ so large that $\P\{0\text{ is black}\}\ge q_0(\tfrac12,4)$, where $q_0$ is from Lemma~\ref{lm:perc-aux}, and let $a,b$ be the constants of that lemma, applied to $\tau_z=\one\{z\text{ is black}\}$ with $\rho=\tfrac12$ and $L=4$.

Let $x_{0,n}$ be an up-right path with $x_0=0$ and $n\ge3N$, let $m=\fl{n/N}-2$, and define the up-right path $z_{0,m}$ with $z_0=0$ as in the proof of Proposition \ref{prop:marchand}. For each $j\le m$ with $z_j$ black, the segment $x_{(j+1)N,(j+2)N}$ has weight at most $N-\delta_0$. These segments are disjoint, and all other weights are at most $1$. Hence, outside an event of probability at most $ae^{-bm}$, every up-right path from $0$ of length $n$ has weight at most
\[n-\frac{\delta_0(m+1)}2\le n\Bigl(1-\frac{\delta_0}{2N}\Bigr)+\delta_0 .\]
By Borel-Cantelli and the almost sure shape theorem \eqref{lln5}, $\gpp(\xi)\le1-\delta_0/(2N)$ for every $\xi\in\Uset$.\medskip

{\bf Part (b).} Let $p_1=\opc$. First we treat $\xi\ne(1/2,1/2)$.
By symmetry it suffices to consider $\xi\in\Uset$ with $\xi\cdot e_1>\tfrac12$. If $\xi=e_1$, then $\gpp(\xi)=\E\w_0<1$, since $p_1<1$. So assume $\tfrac12<\xi\cdot e_1<1$. We have $\beta_{\opc}=\tfrac12$ \cite[(1) on p.~1025]{Dur-84}, and $p\mapsto\beta_p$ is right-continuous at $\opc$ \cite[Remarks on p.~1007]{Dur-84}. So there exists $p'\in(\opc,1)$ with $\beta_{p'}<\xi\cdot e_1$. Fix $\e>0$ with $\beta_{p'}+\e<\xi\cdot e_1$.

Let $\theta=(p'-p_1)/(1-p_1)\in(0,1)$, let $\{U_x\}_{x\in\Z^2}$ be i.i.d.\ uniform on $[0,1]$ and independent of $\w$, and set $\w'_x=1$ if $U_x<\theta$ and $\w'_x=\w_x$ otherwise. The weights $\w'_x$ are i.i.d., satisfy $\w\le\w'\le1$ and $\P\{\w'_0=1\}=p_1+(1-p_1)\theta=p'\in(\opc,1)$, and $\E\abs{\w'_0}<\infty$. Denote by $\Gpp'$ their last-passage times. Then $\Gpp\le\Gpp'$, because the last-passage times are nondecreasing in the weights.

Let $\delta,A,B$ be the constants of Proposition~\ref{prop:marchand}, applied to the weights $\w'$ with this $\e$. Let $z\in\Z_+^2$ with $z\cdot e_2\ge1$ and $z\cdot e_1\ge(\beta_{p'}+\e)\abs z_1$, so that $(k,\ell)=mz$ satisfies the hypothesis of the proposition for every $m\in\N$. Since $\Gpp'_{0,mz}\le m\abs z_1$, \eqref{marchand:decay} gives
\[\E\Gpp_{0,mz}\le\E\Gpp'_{0,mz}\le(1-\delta)m\abs z_1+m\abs z_1\,Ae^{-Bm\,z\cdot e_1}.\]
Dividing by $m$ and letting $m\to\infty$, the fixed-direction limit $\gpp(z)=\lim_m m^{-1}\E\Gpp_{0,mz}$ (which holds by the subadditive ergodic theorem, since $\E\abs{\w_0}<\infty$) gives $\gpp(z)\le(1-\delta)\abs z_1$. By homogeneity, $\gpp\le1-\delta$ at every rational direction $\zeta\in\Uset$ with $\zeta\cdot e_2>0$ and $\zeta\cdot e_1\ge\beta_{p'}+\e$. Since $\gpp$ is continuous, $\gpp(\xi)\le1-\delta<1$.

It remains to show $\gpp(\tfrac12,\tfrac12)=1$. Let $m'$ be the mean of $\mu(\,\cdot\,|\,(-\infty,1))$. This conditional law is well defined since $p_1<1$, and $m'$ is finite since $\E\w_0^-<\infty$. Fix $q\in(\opc,1)$. Let $\{J_x\}$, $\{\eta_x\}$ and $\{V_x\}$ be independent i.i.d.\ families, with $J_x$ Bernoulli($q$), $\eta_x$ Bernoulli($\opc/q$) and $V_x\sim\mu(\,\cdot\,|\,(-\infty,1))$. Set $\w_x=1$ if $J_x\eta_x=1$ and $\w_x=V_x$ otherwise. The $\w_x$ are i.i.d.\ with law $\mu$, since $\P\{J_x\eta_x=1\}=\opc=p_1$, and the limit $\gpp$ depends only on this law. Let $\Gpp^J$ and $g^J$ denote the last-passage times and the point-to-point limit corresponding to the weights $\{J_x\}$. For $n\in\N$ let $x_{0,2n}$ be a path from $0$ to $(n,n)$ that maximizes $\sum_{0\le k<2n}J_{x_k}$, chosen as a function of $\{J_x\}$ alone, for instance the first such path in a fixed deterministic order. Conditionally on $\{J_x\}$, the weight $\w_{x_k}$ has mean $\opc/q+(1-\opc/q)m'$ if $J_{x_k}=1$, and mean $m'$ otherwise. Hence
\[\E\Gpp_{0,(n,n)}\ge\E\sum_{k=0}^{2n-1}\w_{x_k}=\Bigl(\frac{\opc}q+\Bigl(1-\frac{\opc}q\Bigr)m'\Bigr)\E\Gpp^J_{0,(n,n)}+m'\bigl(2n-\E\Gpp^J_{0,(n,n)}\bigr).\]
The weights $J_x$ satisfy the hypotheses of the theorem with density $q>\opc$, and $\beta_q\ge\tfrac12$, so part (c) gives $g^J(\tfrac12,\tfrac12)=1$. Divide by $2n$ and let $n\to\infty$, using the fixed-direction $L^1$ limits for $\w$ and for $J$ (valid by the subadditive ergodic theorem). 
This gives
\[\gpp(\tfrac12,\tfrac12)\ge\frac{\opc}q+\Bigl(1-\frac{\opc}q\Bigr)m'.\]
Letting $q\downarrow\opc$ yields $\gpp(\tfrac12,\tfrac12)\ge1$. Together with $\gpp\le1$, this completes the proof.
\end{proof}

\subsection{Differentiability at the endpoints}  \label{sec:flat-diff-pf}

This section proves that if $p_1\ge\opc$, then $\gpp$ is differentiable at $(\beta_{p_1},1-\beta_{p_1})$.  It is convenient here to alter the definition \eqref{Gxy1} of the last-passage time $\Gpp_{x,y}$ so that $\w_x$ is excluded and $\w_y$ is included.  This of course makes no difference to the limit $\gpp$.     

\begin{theorem}\label{thm:edge-diff}
Assume $p_1\ge\opc$. Then $\gpp$ is differentiable at $(\beta_{p_1},1-\beta_{p_1})$.
\end{theorem}

The proof is completed at the end of this section, after some preparation.

We   define the oriented percolation process more generally.  
Recall that the successive levels  on which the process lives  are denoted by $D_n=\{(i,j)\in\Z^2: i+j=n\}$.    Let $S\subset D_m$ be a given initial occupied set.  Then at time $n\in\Z_+$ the occupied set is $\clb_n(S)=\{ v\in D_{m+n}:  \exists \,u\in S: u\to v\}$.  
If  $S$   is bounded below ($S$ has only finitely many points below the $x$-axis),    the lowest point  $r_n(S)$ of $\clb_n(S)$ is well-defined and satisfies $r_n(S)=(a_n(S), b_n(S))$ where 
\[  a_n(S) = \max_{u \in \clb_n(S)}\{ u \cdot e_1 \}
\qquad\text{and}\qquad 
b_n(S) = \min_{u \in \clb_n(S)}\{ u \cdot e_2 \}  . \]
A particular case of such an initial set is $\wt\Z_-=\{(-k,k): k\in\Z_+\}$, the antidiagonal copy of $\Z_-$.   Occasionally we also use the notation $b(\clb_n(S))=b_n(S)$.   Let $F_m=\{ (k,-k):  k=1,\dotsc, m\}$.   

\begin{lemma}\label{op:lm4}   For infinite sets $A\subset B\subset \wt\Z_-$,
\[  \E[ a_n(A \cup F_m)-a_n(A)]  \ge  \E[ a_n(B \cup F_m)-a_n(B)]  \ge   m    \]
with equality in the last inequality if $B=\wt\Z_-$.  
\end{lemma}

\begin{proof} Since $\clb_n(A \cup F_m) = \clb_n(A) \cup \clb_n(F_m)$,  
\begin{align*}
a_n(A \cup F_m)-a_n(A)&= a_n(A) \vee a_n(F_m)-a_n(A)
=0\vee ( a_n(F_m)-a_n(A)) \\
&\ge 0\vee ( a_n(F_m)-a_n(B)) =  a_n(B\cup F_m)-a_n(B).
\end{align*}  
By a shift of the underlying weights $\w$,   
\[ r_n(\wt\Z_-\cup F_m)(\w)= r_n(\wt\Z_-)(T_{m,-m}\w) + (m,-m) . 
\]
By the shift-invariance of $\P$  
\[ \E[a_n(\wt\Z_-\cup F_m)]- \E[ a_n(\wt\Z_-)]=    m. 
\qedhere \]
\end{proof}

Let $\aaa$ be an integer $\ge 2$.  Fix a constant $c_0<1$ such that  $\P(c_0\le \w_0<1)>0$.  
To have an  {\sl $\aaa$-triangle} (configuration) at $(a,b)\in\Z^2$ means that 
$\w_{a,b+i}<1$ for $i=1,\dotsc, \aaa$,  $c_0\le \w_{a+1,b}<1$, and  
except for $(a+1,b)$ all sites in the triangle 
$\{(i,j):  i\ge a+1, \,j\ge b, \, i-(a+1)+j-b\le \aaa-1\}$ have weight
$\w_{i,j}=1$.   See Figure \ref{fig:ELL}.  

\begin{figure}[h]
	\begin{center}
		\begin{tikzpicture}[>=latex, scale=0.6]
			\definecolor{sussexg}{rgb}{0,0.4,0.4}
				\definecolor{sussexp}{rgb}{0.6,0,0.6} 
				
				\draw[fill=sussexp!30,rounded corners](0.7,1.7)rectangle(1.3,6.3);
				\draw[fill=sussexg!30, rounded corners](1.7,0.7)rectangle(2.3,1.3);
				\draw[densely dotted, color=sussexp, line width=1.5pt, smooth,rounded corners](2.3, 5.3)..controls (2,5.5)..(1.7, 5.3)--(1.7,1.7)--(2.7, 1.7)--(2.7, 0.7)--(6.3,0.7)..controls(6.5,1).. (6.3,1.3)--(2.3, 5.3);
				
				\foreach \x in {1,...,6}{
						\draw[fill = sussexg!30 ] (1,\x)circle(1.7mm);
					}
					
				\foreach \x in {1,...,5}{
						\draw[fill = gray!30] (2,\x)circle(1.7mm);
					}
				\foreach \x in {1,...,4}{
						\draw[fill = gray!30 ] (3, \x)circle(1.7mm);
					}
				\foreach \x in {1,...,3}{
						\draw[fill = gray!30 ] (4, \x)circle(1.7mm);
					}
				\foreach \x in {1,...,2}{
						\draw[fill = gray!30 ] (5,\x)circle(1.7mm);
					}
				\foreach \x in {1,...,1}{
						\draw[fill = gray!30 ] (6,\x)circle(1.7mm);
					}	
				
				\draw[fill=white](1,1)circle(1.2mm);
					
				\draw(1,0.1)node{$a$};
				\draw(0.1,1)node{$b$};	
				\draw(6,0.1)node{$a+\ell $};
				\draw(0.5,6)node[left]{$b+\ell $};	
				
				\draw[fill=sussexp!30](2,1)circle(1.6mm);
				\draw[line width=0.2pt,<->](0.5,6.5)--(0.5, 0.5)--(7,0.5);				
		\end{tikzpicture}
	\end{center}
	\caption{\small An $\ell$-triangle at $(a,b)$ with $\ell = 5$.  The weight $\w_{a,b}$ is unrestricted.  The column above $(a,b)$ has weights $\omega_{a, b+i} < 1$.
			Point $(a+1,b)$ has weight $c_0 \le \omega_{a+1,b} < 1$. In the  region inside the dotted boundary  all weights are equal to $1$.
			}
	\label{fig:ELL}
\end{figure}
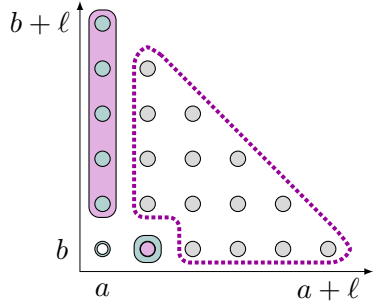

Let $\Saset_\aaa=\{ (\aaa-i,i): 0\le i\le \aaa-1\}$.   Suppose  there is an $\aaa$-triangle
at $(a,b)$.  Then two things happen that are relevant for the sequel. 
\be\label{op:45} \begin{aligned} &\text{There is an up-right path of $\w$-weight $ \aaa+\w_{a+1,b}-1\ge \aaa+c_0-1$ from
$(a,b)$}\\
 &\text{to each of the  $\aaa$ sites of  $(a,b)+\Saset_\aaa$, without counting the
weight at $(a,b)$. } 
\end{aligned}  
\ee  
Furthermore, no open oriented percolation path can go through any of the sites $\{ (a+1,b),  (a,b+i):  i=1,\dotsc, \aaa\}$.   

Now let $S$ be an  infinite  initial set that is bounded below.  Set $\clb_n^0=\clb_n(S)$,  $r^0_n=r_n(S)=(a^0_n, b^0_n)$. 
Define the stopping time 
\[ \tau_1=\inf\{ n>0:  \text{$\exists$ $\aaa$-triangle at $r^0_{n\aaa-\aaa}$}  \}.  \] 
The natural filtration is $\cH_{m}=\sigma\{\w_{i,j}:  i+j \le m\}$.

Since  $r^0_{\tau_1\aaa-\aaa}$ is the lowest point of the occupied set $\clb^0_{\tau_1\aaa-\aaa}$, no open oriented percolation  path can reach 
 $r^0_{\tau_1\aaa-\aaa}+\Saset_\aaa$:  
\[  (r^0_{\tau_1\aaa-\aaa}+\Saset_\aaa)\cap \clb^0_{\tau_1 \aaa}=\varnothing.  \]
To each point of the set $\Saset^0_{\tau_1 \aaa}=r^0_{\tau_1\aaa-\aaa}+\Saset_\aaa\subset D_{\tau_1 \aaa}$ there is an up-right  path
  of $\w$-weight 
\[  \tau_1 \aaa-1+\w_{ r^0_{\tau_1\aaa-\aaa} +e_1} \ge  \tau_1 \aaa-1 +c_0 \]
from some point on $S$.  

Start a new process $\clb^1_n$ at level  $\tau_1 \aaa$ by joining  $\Saset^0_{\tau_1 \aaa}$ to the occupied set: 
\[\clb^1_n=\begin{cases}  \clb^0_n,  &n\le \tau_1 \aaa-1\\
  \clb_{n-\tau_1 \aaa}(\clb^0_{\tau_1 \aaa} \cup \Saset^0_{\tau_1 \aaa}) , &n\ge \tau_1 \aaa.  
\end{cases} 
\]
Let   $r^1_n=(a^1_n, b^1_n)$ be the lowest  point of $\clb^1_n$.  
  
Continue in this manner, with 
\[ \tau_{k+1}=\inf\{ n>\tau_k:  \text{$\exists$ $\aaa$-triangle at $r^k_{n\aaa-\aaa}$}  \},  \   \Saset^k_{\tau_{k+1} \aaa}=r^k_{\tau_{k+1}\aaa-\aaa}+\Saset_\aaa, \] 
\[\clb^{k+1}_n=\begin{cases}  \clb^k_n,  &n\le \tau_{k+1} \aaa-1\\
  \clb_{n-\tau_{k+1} \aaa}(\clb^k_{\tau_{k+1} \aaa} \cup \Saset^k_{\tau_{k+1} \aaa}) , &n\ge \tau_{k+1} \aaa, 
\end{cases} \]
with lowest point $r^{k+1}_n=(a^{k+1}_n, b^{k+1}_n)$.

Let $\rho_\aaa>0$ be the probability of an $\aaa$-triangle and 
\begin{align}\label{Kn-def}
K_n=\max\{k: \tau_k\le n\}\sim\text{Binom}(n, \rho_\aaa).
\end{align}
Then 
\begin{align}\label{rho->0}
\rho_\aaa=\P\{\w_0<1\}^{\aaa}\,\P\{c_0\le\w_0<1\}\,p_1^{\aaa(\aaa+1)/2-1}\longrightarrow0\qquad \text{as }\aaa\to\infty.
\end{align}

\begin{lemma}\label{dif:lm5}  For any infinite initial occupied set that is bounded below, 
\[  \E[ a^{K_n}_{n\aaa}-a^0_{n\aaa}]\ge  n\aaa \rho_\aaa
\qquad\text{and}\qquad 
 \E[ b^{K_n}_{n\aaa}-b^0_{n\aaa}]\le  -n\aaa \rho_\aaa. 
 \]
\end{lemma}

\begin{proof}   The two statements are equivalent since $a^{K_n}_{n\aaa}+b^{K_n}_{n\aaa}=n\aaa= a^{0}_{n\aaa}+b^{0}_{n\aaa}$. 
  Since $a^k_{n\aaa}=a^{K_n}_{n\aaa}$   for $k>K_n$,   and by the strong Markov property and  Lemma \ref{op:lm4}, 
\begin{align*}
\E[ a^{K_n}_{n\aaa}-a^0_{n\aaa}]  &= \sum_{k=1}^n \E[ a^{k}_{n\aaa}-a^{k-1}_{n\aaa}]  
= \sum_{k=1}^n \E[ a^{k}_{n\aaa}-a^{k-1}_{n\aaa},\, \tau_k\le n] \\
&= \sum_{k=1}^n \E\bigl[ a_{n\aaa-\tau_k \aaa}(\clb^{k-1}_{\tau_{k} \aaa} \cup \Saset^{k-1}_{\tau_{k} \aaa})
 - a_{n\aaa-\tau_k \aaa}(\clb^{k-1}_{\tau_{k} \aaa}),\, \tau_k\le n\,\bigr]\\
 &\ge \aaa  \sum_{k=1}^n  \P(\tau_k\le n)  =  \aaa n\rho_\aaa.  
\qedhere 
\end{align*}
\end{proof}

For the remainder of the proof  the initial set for oriented percolation is $S=\wt\Z_-$. 


The final piece of preparation derives a bound on last-passage times.  

\begin{lemma}\label{D.7'}
Let the initial set for the construction of $r^k_n$ be $S=\wt\Z_-$. For all $n,k\ge0$ there exists $u\in\wt\Z_-$ with $u\le r^k_n$ and $\Gpp_{u,\,r^k_n}\ge n+k(c_0-1)$.
\end{lemma}

\begin{proof}
Induction on $k$. For $k=0$, $r^0_n\in\clb_n(\wt\Z_-)$, so there is an open path from some $u\in\wt\Z_-$ to $r^0_n$. Every site of this path other than $u$ has weight $1$, so its $\w$-weight is $n$.

Assume the claim is true for $k$ and all $n$. If $n\le\tau_{k+1}\aaa-1$, then $r^{k+1}_n=r^k_n$, and the claim for $k+1$ follows from the claim for $k$ because $c_0-1<0$. In particular this covers all $n$ when $\tau_{k+1}=\infty$.

Let $n\ge\tau_{k+1}\aaa$ and write $m=\tau_{k+1}\aaa-\aaa$. Since the process $\clb^k$ evolves by open paths after level $\tau_k\aaa\le\tau_{k+1}\aaa$, we have $\clb^k_n=\clb_{n-\tau_{k+1}\aaa}(\clb^k_{\tau_{k+1}\aaa})$, and therefore
\[\clb^{k+1}_n=\clb^k_n\cup\clb_{n-\tau_{k+1}\aaa}(\Saset^k_{\tau_{k+1}\aaa}).\]
If $r^{k+1}_n=r^k_n$, the claim for $k+1$ again follows from the claim for $k$. Otherwise the lowest point $r^{k+1}_n$ of $\clb^{k+1}_n$ is not in $\clb^k_n$, so it is reached by an open path from some point $w\in\Saset^k_{\tau_{k+1}\aaa}=r^k_m+\Saset_\aaa$. This open path has $\w$-weight $n-\tau_{k+1}\aaa$. By the definition of $\tau_{k+1}$ there is an $\aaa$-triangle at $r^k_m$, so by \eqref{op:45} there is an up-right path from $r^k_m$ to $w$ of $\w$-weight at least $\aaa+c_0-1$. By the induction assumption at level $m$ there exists $u\in\wt\Z_-$ with $u\le r^k_m$ and $\Gpp_{u,r^k_m}\ge m+k(c_0-1)$. Concatenating the three paths gives an up-right path from $u$ to $r^{k+1}_n$, so $u\le r^{k+1}_n$, and
\[\Gpp_{u,\,r^{k+1}_n}\ge m+k(c_0-1)+\aaa+c_0-1+n-\tau_{k+1}\aaa=n+(k+1)(c_0-1).\qedhere\]
\end{proof}

\begin{lemma}\label{D.7''}
Let $S_N=\{(-k,k):0\le k\le N\}$ and let $\cP_N$ be the set of pairs $(u,v)$ with $u\in S_N$, $v\in D_N$ and $u\le v$. Then
\[\E\Bigl[\max_{(u,v)\in\cP_N}\bigl(\Gpp_{u,v}-\gpp(v-u)\bigr)^+\Bigr]=o(N)\qquad\text{as }N\to\infty.\]
\end{lemma}

\begin{proof}
Since all weights are $\le1$, $\Gpp_{u,v}\le N$ for $(u,v)\in\cP_N$, and since $v-u\in N\Uset$, $\gpp(v-u)\ge N\min_\Uset\gpp$.   Hence
\begin{equation}\label{aux5198}
\bigl(\Gpp_{u,v}-\gpp(v-u)\bigr)^+\le N\bigl(1-\min_\Uset\gpp\bigr)\qquad\text{for all }(u,v)\in\cP_N .
\end{equation}

In this section the weight of a path excludes the initial vertex and includes the terminal one, while \eqref{lln5} uses the opposite convention. Let $\Gpp^\circ_{x,y}=\Gpp_{x,y}+\w_x-\w_y$ be the last-passage time with the convention of \eqref{lln5}. Since $\w\le1$,
\be\label{conv-switch}
\Gpp_{x,y}\le\Gpp^\circ_{x,y}+1+\w_x^-\qquad\text{and}\qquad-\Gpp_{x,y}\le-\Gpp^\circ_{x,y}+1+\w_y^-.
\ee

Fix $\epsilon,\delta\in(0,1)$. Let $C_g=\max_\Uset\abs\gpp$, let $\varpi$ be the modulus of continuity of $\gpp$ on $\{z\in\R_+^2:\abs z_1\le3\}$, and let
$Y=\sup_{z\in\Z_+^2}\bigl(\abs{\Gpp^\circ_{0,z}-\gpp(z)}-\epsilon\abs z_1\bigr)^+$.
By \eqref{lln5}, $Y<\infty$ almost surely.

For $0\le i\le\ce{2/\delta}$ let $o_i=(-\ce{(i+1)\delta N/2},\fl{i\delta N/2})$. For $u=(-k,k)\in S_N$ and $i=\fl{2k/(\delta N)}$ we have $o_i\le u$ and $\abs{u-o_i}_1\le\delta N+2$. Superadditivity gives $\Gpp_{o_i,v}\ge\Gpp_{o_i,u}+\Gpp_{u,v}$, hence
\be\label{t.aux18}\begin{aligned}\Gpp_{u,v}-\gpp(v-u)&\le[\Gpp_{o_i,v}-\gpp(v-o_i)]+[\gpp(u-o_i)-\Gpp_{o_i,u}]\\&\qquad+[\gpp(v-o_i)-\gpp(v-u)-\gpp(u-o_i)].
\end{aligned}\ee

Let $Y_i=Y\circ T_{o_i}$, the variable $Y$ computed from the base point $o_i$. Then $Y_i$ has the law of $Y$. Let \[ M_N=\max\{\w_x^-:\ x\in S_N\cup\{o_0,\dots,o_{\ce{2/\delta}}\}\}. \]
The vector $v-o_i$ lies in $\Z_+^2$ and has $\abs{v-o_i}_1\le(1+\delta)N+2\le2N+2$ and $\abs{u-o_i}_1\le\delta N+2$. So by \eqref{conv-switch} the first two brackets in \eqref{t.aux18} are at most $\epsilon(2N+2)+Y_i+1+M_N$  and $\epsilon(\delta N+2)+Y_i+1+M_N$. By homogeneity,
\[\gpp(v-o_i)-\gpp(v-u)=N\Bigl[\gpp\Bigl(\frac{v-o_i}N\Bigr)-\gpp\Bigl(\frac{v-u}N\Bigr)\Bigr]\le N\varpi(\delta+2/N)\]
and 
\[ -\gpp(u-o_i)\le\abs{u-o_i}_1\max_\Uset\abs\gpp\le C_g(\delta N+2). \]
Thus the last bracket in \eqref{t.aux18} is at most $N\varpi(\delta+2/N)+C_g(\delta N+2)$. For $N\ge2/\delta$, adding the three bounds gives, for all $(u,v)\in\cP_N$,
\[\Gpp_{u,v}-\gpp(v-u)\le\bigl(3\epsilon+2C_g\delta+\varpi(2\delta)\bigr)N+6+2C_g+2\max_{0\le i\le\ce{2/\delta}}Y_i+2M_N .\]

The number of base points depends only on $\delta$, and each $Y_i$ has the law of the almost surely finite variable $Y$. Also, $\P\{M_N>\epsilon N/4\}\le(N+\ce{2/\delta}+2)\,\P\{\w_0^->\epsilon N/4\}\to0$, because $\E\w_0^-<\infty$. Hence $\P\{6+2C_g+2\max_iY_i+2M_N>\epsilon N\}\to0$ as $N\to\infty$, and with $c=4\epsilon+2C_g\delta+\varpi(2\delta)$,
\[\P\Bigl\{\max_{(u,v)\in\cP_N}\bigl(\Gpp_{u,v}-\gpp(v-u)\bigr)>cN\Bigr\}\longrightarrow0\qquad\text{as }N\to\infty.\]
This and \eqref{aux5198} give
\begin{align*}
&\frac1N\,\E\Bigl[\max_{(u,v)\in\cP_N}\bigl(\Gpp_{u,v}-\gpp(v-u)\bigr)^+\Bigr]\\
&\qquad\qquad\qquad
\le c+\bigl(1-\min_\Uset\gpp\bigr)\,\P\Bigl\{\max_{(u,v)\in\cP_N}\bigl(\Gpp_{u,v}-\gpp(v-u)\bigr)>cN\Bigr\}.
\end{align*}
Taking $N\to\infty$ then $\epsilon$ and $\delta$ to $0$ gives the claim. 
\end{proof}

\begin{proof}[Proof of Theorem \ref{thm:edge-diff}]
Write $\bar g(s)=\gpp(s,1-s)$, $0\le s\le1$. By Theorem \ref{perc-flat-thm}(b--c), $\bar g(\beta_{p_1})=1$. By homogeneity of $\gpp$, it suffices to show that $\bar g$ is differentiable at $\beta_{p_1}$.

If $p_1>\opc$, then $\bar g=1$ on $[1-\beta_{p_1},\beta_{p_1}]$ by Theorem \ref{perc-flat-thm}(c), so $\bar g'(\beta_{p_1}-)=0$. If $p_1=\opc$, then $\beta_{p_1}=\tfrac12$, and the symmetry $\bar g(s)=\bar g(1-s)$ gives $\bar g'(\tfrac12-)=-\bar g'(\tfrac12+)$. In both cases it suffices to show $\bar g'(\beta_{p_1}+)=0$.

Since $\bar g$ is concave, $\bar g\le1$ and $\bar g(\beta_{p_1})=1$, the function $\bar g$ is nonincreasing on $[\beta_{p_1},1]$ and $\bar g'(\beta_{p_1}+)\le0$. By concavity, $\bar g'(\beta_{p_1}+)\ge(\bar g(\beta_{p_1}+h)-\bar g(\beta_{p_1}))/h$ for every $h\in(0,1-\beta_{p_1}]$. So it suffices to show that
\begin{align}\label{goal5241}
\bar g(\beta_{p_1}+\rho_\aaa)\ge1-\frac{1-c_0}{\aaa}\,\rho_\aaa\qquad\text{for sufficiently large integers }\aaa\ge2.
\end{align}
(Recall from \eqref{rho->0} that $\rho_\aaa\to0$ as $\aaa\to\infty$.)

Abbreviate $r_n=r^{K_n}_{n\aaa}$ and let $u_n$ be the point of $\wt\Z_-$ given by Lemma~\ref{D.7'} for $r_n$, with the largest first coordinate. 
Since $u_n\le r_n$ and $u_n$ is on level $0$, the vector $r_n-u_n$ lies in $\Z_+^2$ on level $n\aaa$. Hence $X_n=(r_n-u_n)\cdot e_1/(n\aaa)\in[0,1]$.
By homogeneity, $\gpp(r_n-u_n)={n\aaa}\bar g(X_n)$.      Let 
\[ A_n=\{a^0_{n\aaa}\ge0\} \quad\text{and}\quad  E_n=\max_{(u,v)\in\cP_{n\aaa}}(\Gpp_{u,v}-\gpp(v-u))^+. \]
Since $u_n\le r_n$ and $a^{K_n}_{n\aaa}\ge a^0_{n\aaa}$, we have $u_n\cdot e_2\le r_n\cdot e_2=n\aaa-a^{K_n}_{n\aaa}\le n\aaa-a^0_{n\aaa}$. So on $A_n$ we have $u_n\in S_{n\aaa}$, and $(u_n,r_n)\in\cP_{n\aaa}$ because $r_n\in D_{n\aaa}$. By Lemma~\ref{D.7'},
\[\bar g(X_n)\ge\frac{\Gpp_{u_n,r_n}-E_n}{n\aaa}\ge1+(c_0-1)\frac{K_n}{n\aaa}-\frac{E_n}{n\aaa} .\]
On $A_n^c$ use $\bar g\ge-C_g$, where $C_g=\max_\Uset\abs\gpp$. Since $c_0<1$, $K_n\ge0$, and $E_n\ge0$, the right-hand side is at most $1$. Hence, in all cases
\[\bar g(X_n)\ge1+(c_0-1)\frac{K_n}{n\aaa}-\frac{E_n}{n\aaa}-(1+C_g)\one_{A_n^c}.\]

Take expectations and use four facts: Jensen's inequality, $\E K_n=n\rho_\aaa$ from \eqref{Kn-def}, Lemma~\ref{D.7''}, and $\P(A_n^c)\to0$ (by \eqref{an/n} and $\beta_{p_1}\ge\tfrac12$). Thus, 
\[\bar g(\E X_n)\ge \E \bar g( X_n) \ge 1-(1-c_0)\frac{\rho_\aaa}{\aaa}-o(1).\]
Since $u_n\in\wt\Z_-$, we have $u_n\cdot e_1\le0$, so $X_n\ge a^{K_n}_{n\aaa}/(n\aaa)$. By Lemma~\ref{dif:lm5} and \eqref{an/n}, $\E X_n\ge\E a^0_{n\aaa}/(n\aaa)+\rho_\aaa\ge\beta_{p_1}+\rho_\aaa-o(1)$. Since $\bar g$ is nonincreasing on $[\beta_{p_1},1]$ and continuous, letting $n\to\infty$ proves \eqref{goal5241}.
\end{proof}

\begin{proof}[Proof of Theorem \ref{th:flat-edge}]
Theorem \ref{perc-flat-thm} gives the statements about $\{\gpp=1\}$. Theorem \ref{thm:edge-diff} gives differentiability at $(\beta_{p_1},1-\beta_{p_1})$ for every $p_1\ge\opc$. This point is $(\tfrac12,\tfrac12)$ when $p_1=\opc$, which proves the differentiability statement in (b). It is $\etamax$ when $p_1>\opc$, and differentiability at $\etamin$ follows by the symmetry $\gpp(s,t)=\gpp(t,s)$, which proves the differentiability statement in (c).
\end{proof}

\section{Shape theorem}  
\label{app:shape}

\begin{theorem}\label{th:shape}
Assume $\w_x$ are i.i.d.\ such that
	\begin{align}\label{shape-conditions}
    \E[\w_0^-]<\infty\quad\text{and}\quad\int_0^\infty\P\{\w_0>r\}^{1/2}\,dr<\infty.
    \end{align}
Then with $\P$-probability one
\begin{align}\label{lln5}
\lim_{n\to\infty} n^{-1}\max_{x\in\Z_+^2:\abs{x}_1=n}\abs{\Gpp_{0,x}-\gpp(x)}=0.
\end{align}
If also $\E[(\w_0^-)^2]<\infty$, then \eqref{lln5} also holds in $L^1$.
\end{theorem}

\begin{proof}[Proof of Theorem \ref{th:shape}]
The almost sure limit \eqref{lln5} follows from Theorem 5.1(i) and Remark 5.1(ii) in \cite{Mar-04}. 
We prove the $L^1$ convergence. 

Fix an integer $k\ge2$. 
Let $\xi_\ell=(\ell/k,1-\ell/k)$, $\ell\in\{0,\dotsc,k\}$. Given an integer $n\ge2$, let $m_n^-=\fl{n/(1+1/k)}\in[1,n)$ and $m_n^+=\ce{n/(1-1/k)}\ge n$.

Given $x\ge0$ with $\abs{x}_1=n$ there exists $\ell^-(x)\in\{0,\dotsc,k\}$ such that $m_n^-\xi_{\ell^-(x)}\le x$. (This is because $n-m_n^-\ge m_n^-/k$.)
Also, there exists $\ell^+(x)\in\{0,\dotsc,k\}$ such that $x\le m_n^+\xi_{\ell^+(x)}$. (This is because $m_n^+-n\ge m_n^+/k$.)

A path  from $0$ to $x$    can     first go to $\fl{m_n^-\xi_{\ell^-(x)}}$ and then take at most $n-m_n^-+1\le n/(k+1)+2$ $e_1$-steps   followed by at most 
$n/(k+1)+2$ $e_2$-steps. Hence,
\begin{align*}
\Gpp_{0,x}&\ge\Gpp_{0,\fl{m_n^-\xi_{\ell^-(x)}}}-\sum_{i=0}^{n/(k+1)+1} \abs{\w_{\fl{m_n^-\xi_{\ell^-(x)}}+ie_1}}\\
&\qquad-\max_{0\le i\le n/(k+1)+1}\sum_{j=0}^{n/(k+1)+1}\abs{\w_{\fl{m_n^-\xi_{\ell^-(x)}}+ie_1+je_2}}.
\end{align*}
This gives
\begin{align*}
\Gpp_{0,x}-\gpp(x)&\ge\gpp(\fl{m_n^-\xi_{\ell^-(x)}})-\gpp(x)+\Gpp_{0,\fl{m_n^-\xi_{\ell^-(x)}}}-\gpp(\fl{m_n^-\xi_{\ell^-(x)}})\\
&\qquad -\sum_{i=0}^{n/(k+1)+1} \abs{\w_{\fl{m_n^-\xi_{\ell^-(x)}}+ie_1}}\\
&\qquad-\max_{0\le i\le n/(k+1)+1}\sum_{j=0}^{n/(k+1)+1}\abs{\w_{\fl{m_n^-\xi_{\ell^-(x)}}+ie_1+je_2}}.
\end{align*}
Similarly,
\begin{align*}
\Gpp_{0,x}-\gpp(x)&\le\gpp(\fl{m_n^+\xi_{\ell^+(x)}})-\gpp(x)+\Gpp_{0,\fl{m_n^+\xi_{\ell^+(x)}}}-\gpp(\fl{m_n^+\xi_{\ell^+(x)}})\\
&\qquad +\sum_{i=0}^{n/(k-1)+1} \abs{\w_{\fl{m_n^+\xi_{\ell^+(x)}}-ie_1}} \\
&\qquad+\max_{0\le i\le n/(k-1)+1}\sum_{j=0}^{n/(k-1)+1}\abs{\w_{\fl{m_n^+\xi_{\ell^+(x)}}-ie_1-je_2}}.
\end{align*}

Note that $\abs{\fl{m_n^\pm\xi_{\ell^\pm(x)}}-x}_1\le 1+2\abs{n-m_n^\pm}\le 1+2n/(1\mp k)$. By continuity of $\gpp$ 
	\begin{align*}
	&\lim_{k\to\infty}\lim_{n\to\infty}n^{-1}\max_{x\in n\Uset}\abs{\gpp(\fl{m_n^\pm\xi_{\ell^+(x)}})-\gpp(x)}\\
	&=\lim_{k\to\infty}\lim_{n\to\infty}\max_{x\in n\Uset}\abs{\gpp(n^{-1}\fl{m_n^\pm\xi_{\ell^+(x)}})-\gpp(n^{-1}x)}=0.
	\end{align*}
	
By Proposition 2.1(i) of \cite{Mar-04}  
	\begin{align*}
	&\lim_{n\to\infty}n^{-1}\E\Big[\max_{0\le\ell\le k}\abs{\Gpp_{0,\fl{m_n^\pm\xi_\ell}}-\gpp(\fl{m_n^\pm\xi_\ell})}\Big]
	=0.
	\end{align*}
	
Next we put some distance  between the sums to make them independent: 
	\begin{align*}
	&n^{-1}\E\Big[\max_{0\le\ell\le k}\max_{0\le i\le n/(k-1)+1}\sum_{j=0}^{n/(k-1)+1}\abs{\w_{\fl{m_n^-\xi_{\ell}}+ie_1+je_2}}\Big]\\
	&\le n^{-1}\E\Big[\max_{0\le\ell\le k/2}\max_{0\le i\le n/(k-1)+1}\sum_{j=0}^{n/(k-1)+1}\abs{\w_{\fl{m_n^-\xi_{2\ell}}+ie_1+je_2}}\Big]\\
	&\qquad+ n^{-1}\E\Big[\max_{0\le\ell<k/2}\max_{0\le i\le n/(k-1)+1}\sum_{j=0}^{n/(k-1)+1}\abs{\w_{\fl{m_n^-\xi_{2\ell+1}}+ie_1+je_2}}\Big].
	\end{align*}
The proof of the theorem is complete if we prove that the  right-hand side vanishes as first $n\to\infty$ and then $k\to\infty$. 
We show the first limit, the second being similar. Centering the $\abs{\w_x}$ terms   does not change the limit. Hence, we will  show that 
	\[\lim_{k\to\infty}\lim_{n\to\infty} n^{-1}\E\Big[\max_{0\le\ell\le k/2}\max_{0\le i\le n/(k-1)+1}\sum_{j=0}^{n/(k-1)+1}(\abs{\w_{\fl{m_n^-\xi_{2\ell}}+ie_1+je_2}}-\E\abs{\w_0})\Big]=0.\]
Abbreviate $\sigma^2=\E[\abs{\w_0}^2]-(\E\abs{\w_0})^2$ and let \[S^{i,\ell}_{n,k}=\sum_{j=0}^{n/(k-1)+1}(\abs{\w_{\fl{m_n^-\xi_{2\ell}}+ie_1+je_2}}-\E\abs{\w_0}).\] 
Note that $\{S^{i,\ell}_{n,k}:0\le i\le n/(k-1)+1,0\le \ell\le k/2\}$ are i.i.d.
By taking $n$ and $k$ large enough  and restricting to $t\ge k^{-1/4}$ Chebyshev's inequality gives 
$P\{    S^{0,0}_{n,k}  \ge nt  \}\le 2\sigma^2/(nkt^2)<1/2$.   Using  $(\frac{k}2+1)(\frac{n}{k-1}+2)\le 2n$ and  $(1-\delta)^n\ge 1-n\delta$ we bound the  expectation: 
\begin{align*}
&E\Big[\max_{0\le\ell\le k/2}\;\max_{0\le i\le n/(k-1)+1} n^{-1}  S^{i,\ell}_{n,k}\Bigr] 
=    \int_0^\infty   \P\Big\{   \max_{0\le\ell\le k/2}\max_{0\le i\le n/(k-1)+1}   S^{i,\ell}_{n,k}  \ge  n t  \Big\}\,  dt    \\
&\qquad\qquad \le  k^{-1/4}  + \int_{k^{-1/4}}^\infty  \big[1 -  \bigl(  1 -  \tfrac{2\sigma^2}{nkt^2}     \bigr)^{2n}\,\big]  \,    dt \\
&\qquad\qquad \le  k^{-1/4} + \int_{k^{-1/4}}^\infty  \  \tfrac{4\sigma^2}{kt^2}      \,    dt
=  k^{-1/4} + 4\sigma^2 k^{-3/4}   \longrightarrow 0 
\end{align*}
  as $n\to\infty$ and then $k\to\infty$.
The argument   can be repeated for the other sums. 
%
\end{proof}


\bibliographystyle{abbrv}

\bibliography{BusemannMagnum}

\end{document}